\documentclass[11pt,a4paper]{amsart}

\usepackage{latexsym}
\usepackage{youngtab} 
\usepackage{latexsym}
\usepackage{amsfonts}
\usepackage{amsxtra}
\usepackage{amsmath}
\usepackage{amssymb}
\usepackage{mathrsfs}
\usepackage{amsthm}
\usepackage{enumitem}
\usepackage{fullpage}
\usepackage[colorlinks]{hyperref}
\usepackage{stmaryrd}
\usepackage{color}
\usepackage[all,knot,poly,ps]{xy}
\usepackage[colorlinks]{hyperref}

\theoremstyle{plain}
\newtheorem{theorem}{Theorem}[section]
\newtheorem{proposition}[theorem]{Proposition}
\newtheorem{lemma}[theorem]{Lemma}

\theoremstyle{definition}

\newtheorem{example}[theorem]{Example}
\newtheorem{remark}[theorem]{Remark}
\numberwithin{equation}{section}

\DeclareMathOperator{\End}{End}

\newcommand{\half}{\frac{1}{2}}
\newcommand{\thalf}{\frac{3}{2}}

\renewcommand{\ge}{\geqslant}

\renewcommand{\ge}{\geqslant}

\renewcommand{\le}{\leqslant}

\begin{document}
\title[Partition Algebras]
{Jucys--Murphy Elements and a Presentation for Partition Algebras}

\author{John Enyang}
\curraddr{School of Mathematics and Statistics F07, University of Sydney NSW 2006, Australia}
\address{Department of Mathematics and Statistics, University of Melbourne, Parkville, VIC 3010 Australia}
\email{jenyang@unimelb.edu.au}

\makeatletter
\renewcommand{\subjclassname}{$2010$ Mathematics Subject Classification}
\makeatother

\subjclass[JEL]{Primary 20C08; Secondary 15E10, 20C30.}


\keywords{Partition algebras; Jucys--Murphy elements; central elements; presentation.}

\maketitle

\begin{abstract}
We give a new presentation for the partition algebras. This presentation was discovered in the course of establishing an inductive formula for the partition algebra Jucys--Murphy elements defined by Halverson and Ram [European J. Combin. \textbf{26} (2005), 869--921]. Using Schur--Weyl duality we show that our recursive formula and the original definition of Jucys--Murphy elements given by Halverson and Ram are equivalent.  The new presentation and inductive formula for the partition algebra Jucys--Murphy elements given in this paper are used to construct the seminormal representations for the partition algebras in a separate paper. 
\end{abstract}


\section{Introduction}
The partition algebras $A_k(n)$, for $k,n\in\mathbb{Z}_{\ge0},$ are a family of algebras defined in the work of Martin and Jones in ~\cite{Mar:1991}, \cite{Mar:1994}, \cite{Jo:1994} in connection with the Potts model and higher dimensional statistical mechanics. By~\cite{Jo:1994}, the partition algebra $A_k(n)$ is in Schur--Weyl duality with the symmetric group $\mathfrak{S}_n$ acting diagonally on the $k$--fold tensor product $V^{\otimes k}$ of its $n$--dimensional permutation representation $V$. In~\cite{Mar:2000}, Martin defined the partition algebras $A_{k+\half}(n)$, for $k,n\in\mathbb{Z}_{\ge0},$ as the centralisers of the subgroup $\mathfrak{S}_{n-1}\subseteq\mathfrak{S}_n$ acting on $V^{\otimes k}$. Including the algebras $A_{k+\half}(n)$ in the tower 
\begin{align}\label{tower}
A_0(n)\subseteq A_\half(n)\subseteq A_{1}(n)\subseteq A_{1+\half}(n)\subseteq\cdots 
\end{align}
allowed for the simultaneous analysis of the whole tower of algebras using the Jones Basic construction by Martin~\cite{Mar:2000} and Halverson and Ram~\cite{HR:2005}. 

Halverson and Ram~\cite{HR:2005} and East~\cite{Ea:2007} have given a presentation for the partition algebras in terms of Coxeter generators for the symmetric group and certain contractions. Halverson and Ram~\cite{HR:2005} used Schur--Weyl duality to show that certain diagrammatically defined elements in the partition algebras $A_k(n)$ and $A_{k+\half}(n)$ play an analogous role to the classical Jucys--Murphy elements in the group algebra of the symmetric group $\mathfrak{S}_k$. 

The Jucys--Murphy elements in the group algebra of the symmetric group respect the inclusions $\mathfrak{S}_{k-1}\subseteq\mathfrak{S}_k$ and are simultaneously diagonalisable in any irreducible representation of the symmetric group. The seminormal representations of the symmetric group are the irreducible matrix representations with respect to a basis on which the Jucys--Murphy elements act diagonally. These may be constructed inductively (see~\cite{VO:2005}, for example). 

This paper provides a new presentation for the partition algebras. This presentation was discovered in the course of establishing an inductive formula for the Jucys--Murphy elements for partition algebras given by Halverson and Ram. In a separate paper, we use this new presentation to construct seminormal representations for the partition algebras~\cite{En:2011}, solving a problem highlighted in the introduction to the paper of Halverson and Ram~\cite{HR:2005}.

In~\S\ref{a-1-1} we recall the presentation of the partition algebras given by Halverson and Ram~\cite{HR:2005} and East~\cite{Ea:2007}. In \S\ref{a-1-2} we state a recursion defining a family of operators 
\begin{align}
(L_i,L_{i+\half}:i=0,1,\ldots)\label{a-1-a}
\end{align} 
and establish that the operators~\eqref{a-1-a} form a commuting family with properties analogous to the Jucys--Murphy elements that arise in the representation theory of the symmetric group. Simultaneously, we establish some basic commutativity results for certain operators denoted
\begin{align}
(\sigma_{i},\sigma_{i+\half}:i=1,2,\ldots)\label{a-1-b}
\end{align} 
which arose in the recursive definition of the Jucys--Murphy elements~\eqref{a-1-a}. 
In~\S\ref{a-1-3} we show that the elements of~\eqref{a-1-b} are involutions which are related to the Coxeter generators for the symmetric group in a precise way. Using the relation between the involutions~\eqref{a-1-b} and the Coxeter generators for the symmetric group, and the properties established in \S\ref{a-1-2}, we derive a new presentation for the partition algebras. In~\S\ref{a-1-5} we give formulae for the actions of the Jucys--Murphy elements~\eqref{a-1-a} and the involutions~\eqref{a-1-b} on tensor space. Using Schur--Weyl duality, we demonstrate that the recursive definition of Jucys--Murphy elements given in~\S\ref{a-1-2} is equivalent to the definition of Jucys--Murphy elements given by Halverson and Ram~\cite{HR:2005}.

\subsection*{Acknowledgements} 
The author would like to express his gratitude to Arun Ram for numerous stimulating conversations throughout the course of this work. The author is also indebted to Fred Goodman and Susanna Fishel for several helpful discussions related to this research, and to the two anonymous referees for their  suggestions. This research was supported by the Australian Research Council (grant ARC DP--0986774) at the University of Melbourne.

\section{The Partition Algebras} \label{a-1-1}
In this section we follow the exposition by Halverson and Ram in~\cite{HR:2005}. For $k=1,2,\ldots,$ let 
\begin{align*}
{A}_k&=\{\text{set partitions of $\{1,2,\dots,k,1',2',\dots,k'\}$}\},\qquad\text{and,}\\
{A}_{k-\half}&=\{d\in A_k\,|\,\text{$k$ and $k'$ are in the same block of $d$}\}.
\end{align*}
Any element $\rho\in{A}_k$ may be represented as a graph with $k$ vertices in the top row, labelled from left to right, by $1,2,\dots,k$ and $k$ vertices in the bottom row, labelled, from left to right by $1',2',\dots,k'$, with vertex $i$ joined to vertex $j$ if $i$ and $j$ belong to the same block of $\rho$. The representation of a partition by a diagram is not unique; for example the partition 
\begin{align*}
\rho=\{\{1,1',3,4',5',6\},\{2,2',3',4,5,6'\}\}
\end{align*}
can be represented by the diagrams:
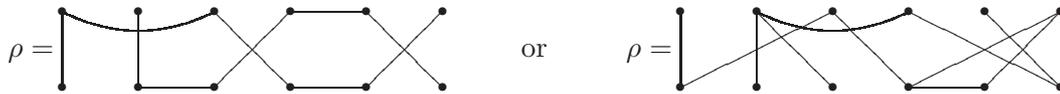
\begin{align*}
\rho&=\xy
(-25,5)*=0{\text{\tiny\textbullet}}="+";
(-15,5)*=0{\text{\tiny\textbullet}}="+";
(-5,5)*=0{\text{\tiny\textbullet}}="+";
(5,5)*=0{\text{\tiny\textbullet}}="+";
(15,5)*=0{\text{\tiny\textbullet}}="+";
(25,5)*=0{\text{\tiny\textbullet}}="+";
(-25,-5)*=0{\text{\tiny\textbullet}}="+";
(-15,-5)*=0{\text{\tiny\textbullet}}="+";
(-5,-5)*=0{\text{\tiny\textbullet}}="+";
(5,-5)*=0{\text{\tiny\textbullet}}="+";
(15,-5)*=0{\text{\tiny\textbullet}}="+";
(25,-5)*=0{\text{\tiny\textbullet}}="+";
(-25,5)*{}; (-25,-5)*{} **\dir{-};
(-15,5)*{}; (-15,-5)*{} **\dir{-};
(15,5)*{}; (25,-5)*{} **\dir{-};
(15,-5)*{}; (25,5)*{} **\dir{-};
(5,5)*{}; (15,5)*{} **\dir{-};
(5,-5)*{}; (15,-5)*{} **\dir{-};
(-5,-5)*{}; (5,5)*{} **\dir{-};
(-5,5)*{}; (5,-5)*{} **\dir{-};
(-15,-5)*{}; (-5,-5)*{} **\dir{-};
(-25,5)*{}; (-5,5)*{} **\crv{(-15,0)};
\endxy
&&\text{or}&&
\rho=\xy
(-25,5)*=0{\text{\tiny\textbullet}}="+";
(-15,5)*=0{\text{\tiny\textbullet}}="+";
(-5,5)*=0{\text{\tiny\textbullet}}="+";
(5,5)*=0{\text{\tiny\textbullet}}="+";
(15,5)*=0{\text{\tiny\textbullet}}="+";
(25,5)*=0{\text{\tiny\textbullet}}="+";
(-25,-5)*=0{\text{\tiny\textbullet}}="+";
(-15,-5)*=0{\text{\tiny\textbullet}}="+";
(-5,-5)*=0{\text{\tiny\textbullet}}="+";
(5,-5)*=0{\text{\tiny\textbullet}}="+";
(15,-5)*=0{\text{\tiny\textbullet}}="+";
(25,-5)*=0{\text{\tiny\textbullet}}="+";
(-25,5)*{}; (-25,-5)*{} **\dir{-};
(-15,5)*{}; (-15,-5)*{} **\dir{-};
(15,5)*{}; (25,-5)*{} **\dir{-};
(15,-5)*{}; (25,5)*{} **\dir{-};
(5,5)*{}; (25,-5)*{} **\dir{-};
(5,-5)*{}; (15,-5)*{} **\dir{-};
(5,-5)*{}; (25,5)*{} **\dir{-};
(-15,5)*{}; (-5,-5)*{} **\dir{-};
(-5,5)*{}; (5,-5)*{} **\dir{-};
(-25,-5)*{}; (-5,5)*{} **\dir{-};
(-15,5)*{}; (5,5)*{} **\crv{(-5,0)};
\endxy
\end{align*}
If $\rho_1,\rho_2\in{A}_k$, then the composition $\rho_1\circ\rho_2$ is the partition obtained by placing $\rho_1$ above $\rho_2$ and identifying each vertex in the bottom row of $\rho_1$ with the corresponding vertex in the top row of $\rho_2$ and deleting any components of the resulting diagram which contains only elements from the middle row. 
The composition product makes $A_k$ into an associative monoid with identity
\begin{align*}
1&=\thinspace\thinspace\xy
(-25,5)*=0{\text{\tiny\textbullet}}="+";
(-15,5)*=0{\text{\tiny\textbullet}}="+";
(-5,0)*=0{\cdots}="+";
(5,5)*=0{\text{\tiny\textbullet}}="+";
(-25,-5)*=0{\text{\tiny\textbullet}}="+";
(-15,-5)*=0{\text{\tiny\textbullet}}="+";
(5,-5)*=0{\text{\tiny\textbullet}}="+";
(-25,5)*{}; (-25,-5)*{} **\dir{-};
(-15,5)*{}; (-15,-5)*{} **\dir{-};
(5,5)*{}; (5,-5)*{} **\dir{-};
\endxy\,.
\end{align*}

Let $z$ be an indeterminant and $R=\mathbb{Z}[z]$. The partition algebra $\mathcal{A}_k(z)$ is the $R$--module freely generated by $A_k$, equipped with the product 
\begin{align*}
\rho_1\rho_2=z^{\ell}\rho_1\circ\rho_2,&&\text{for $\rho_1,\rho_2\in{A}_k$,}
\end{align*}
where $\ell$ is the number of blocks removed from the middle row in constructing the composition $\rho_1\circ\rho_2$. Let $\mathcal{A}_{k-\half}(z)$ denote the subalgebra of $\mathcal{A}_k(z)$ generated by $A_{k-\half}$. A presentation for $\mathcal{A}_k(z)$ has been given by Halverson and Ram~\cite{HR:2005} and East~\cite{Ea:2007}. 
\begin{theorem}[Theorem~1.11 of~\cite{HR:2005}]\label{hr-presentation}
If $k=1,2,\ldots,$ then the partition algebra $\mathcal{A}_{k}(z)$ is the unital associative $R$--algebra presented by the generators 
\begin{align*}
p_1,p_{1+\half},p_2,p_{2+\half},\ldots,p_{k},s_1,s_2,\ldots,s_{k-1},
\end{align*}
and the relations
\begin{enumerate}
\item (Coxeter relations) \label{cox-0}
\begin{enumerate}[label=(\roman{*}), ref=\roman{*}]
\item $s_i^2=1$, for $i=1,\ldots, k-1$.\label{cox-i}
\item $s_is_j=s_js_i$, if $j\ne i+1$.\label{cox-ii}
\item $s_is_{i+1}s_i=s_{i+1}s_is_{i+1}$, for $i=1,\ldots,k-2$. \label{cox-iii}
\end{enumerate}
\item (Idempotent relations)\label{ide-0}
\begin{enumerate}[label=(\roman{*}), ref=\roman{*}]
\item $p_i^2=z p_i$, for $i=1,\ldots,k$. \label{ide-i}
\item $p_{i+\half}^2=p_{i+\half}$, for $i=1,\ldots,k-1$. \label{ide-ii}
\item $s_ip_{i+\half}=p_{i+\half}s_i=p_{i+\half}$, for $i=1,\ldots,k-1$. \label{ide-iii}
\item $s_ip_ip_{i+1}=p_ip_{i+1}s_i=p_ip_{i+1}$, for $i=1,\ldots,k-1$. \label{ide-iv}
\end{enumerate}
\item (Commutation relations)\label{com-0}
\begin{enumerate}[label=(\roman{*}), ref=\roman{*}]
\item $p_ip_j=p_jp_i$, for $i=1,\ldots k$ and $j=1,\ldots,k$. \label{com-i}
\item $p_{i+\half}p_{j+\half}=p_{j+\half}p_{i+\half}$, for $i=1,\ldots k-1$ and $j=1,\ldots,k-1$.\label{com-ii}
\item $p_ip_{j+\half}=p_{j+\half}p_i$, for $j\ne i,i+1$. \label{com-iii}
\item $s_i p_j=p_js_i$, for $j\ne i,i+1$.  \label{com-iv}
\item $s_ip_{j+\half}=p_{j+\half}s_i$, for $j\ne i-1,i+1$.\label{com-v}
\item $s_ip_is_i=p_{i+1}$, for $i=1,\ldots,k-1$. \label{com-vi}
\item $s_{i}p_{i-\half}s_{i}=s_{i-1}p_{i+\half}s_{i-1}$, for $i=2,\ldots,k-1$. \label{com-vii}
\end{enumerate}
\item (Contraction relations)\label{con-0}
\begin{enumerate}[label=(\roman{*}), ref=\roman{*}]
\item $p_{i+\half}p_{j}p_{i+\half}=p_{i+\half}$, for $j=i,i+1$.  \label{con-i} 
\item $p_ip_{j-\half}p_i=p_i$, for $j=i,i+1$. \label{con-ii}
\end{enumerate}
\end{enumerate}
\end{theorem}
The above relations also imply that:
\begin{align*}
&p_{i+\half}s_{i\pm1}p_{i+\half}=p_{i+\half}p_{i\pm1+\half},\\
&p_is_ip_i=p_ip_{i+1}=p_{i+1}s_ip_{i+1},\\
&p_{i}p_{i+\half}p_{i+1}=p_is_i,\\
&p_{i+1}p_{i+\half}p_i=p_{i+1}s_i.
\end{align*}
In Theorem~\ref{hr-presentation}, the following identifications have been made:
\begin{align*}
s_i&=\thinspace\thinspace\xy
(-5,9)*=0{i}="+";
(5,9)*=0{i+1}="+";
(-25,5)*=0{\text{\tiny\textbullet}}="+";
(-25,-5)*=0{\text{\tiny\textbullet}}="+";
(-20,0)*=0{\cdots}="+";
(20,0)*=0{\cdots}="+";
(-15,5)*=0{\text{\tiny\textbullet}}="+";
(-15,-5)*=0{\text{\tiny\textbullet}}="+";
(-5,5)*=0{\text{\tiny\textbullet}}="+";
(-5,-5)*=0{\text{\tiny\textbullet}}="+";
(5,5)*=0{\text{\tiny\textbullet}}="+";
(5,-5)*=0{\text{\tiny\textbullet}}="+";
(15,5)*=0{\text{\tiny\textbullet}}="+";
(15,-5)*=0{\text{\tiny\textbullet}}="+";
(25,5)*=0{\text{\tiny\textbullet}}="+";
(25,-5)*=0{\text{\tiny\textbullet}}="+";
(-25,5)*{}; (-25,-5)*{} **\dir{-};
(-15,5)*{}; (-15,-5)*{} **\dir{-};
(-5,5)*{}; (5,-5)*{} **\dir{-};
(5,5)*{}; (-5,-5)*{} **\dir{-};
(15,5)*{}; (15,-5)*{} **\dir{-};
(25,5)*{}; (25,-5)*{} **\dir{-};
\endxy
&&\text{and}&&
p_j=\thinspace\thinspace\xy
(-5,9)*=0{j}="+";
(-25,5)*=0{\text{\tiny\textbullet}}="+";
(-25,-5)*=0{\text{\tiny\textbullet}}="+";
(-20,0)*=0{\cdots}="+";
(10,0)*=0{\cdots}="+";
(-15,5)*=0{\text{\tiny\textbullet}}="+";
(-15,-5)*=0{\text{\tiny\textbullet}}="+";
(-5,5)*=0{\text{\tiny\textbullet}}="+";
(-5,-5)*=0{\text{\tiny\textbullet}}="+";
(5,5)*=0{\text{\tiny\textbullet}}="+";
(5,-5)*=0{\text{\tiny\textbullet}}="+";
(15,5)*=0{\text{\tiny\textbullet}}="+";
(15,-5)*=0{\text{\tiny\textbullet}}="+";
(-25,5)*{}; (-25,-5)*{} **\dir{-};
(-15,5)*{}; (-15,-5)*{} **\dir{-};
(5,5)*{}; (5,-5)*{} **\dir{-};
(15,5)*{}; (15,-5)*{} **\dir{-};
\endxy
\end{align*}
and
\begin{align*}
p_{i+\half}&=\thinspace\thinspace\xy
(-5,9)*=0{i}="+";
(5,9)*=0{i+1}="+";
(-25,5)*=0{\text{\tiny\textbullet}}="+";
(-25,-5)*=0{\text{\tiny\textbullet}}="+";
(-20,0)*=0{\cdots}="+";
(20,0)*=0{\cdots}="+";
(-15,5)*=0{\text{\tiny\textbullet}}="+";
(-15,-5)*=0{\text{\tiny\textbullet}}="+";
(-5,5)*=0{\text{\tiny\textbullet}}="+";
(-5,-5)*=0{\text{\tiny\textbullet}}="+";
(5,5)*=0{\text{\tiny\textbullet}}="+";
(5,-5)*=0{\text{\tiny\textbullet}}="+";
(15,5)*=0{\text{\tiny\textbullet}}="+";
(15,-5)*=0{\text{\tiny\textbullet}}="+";
(25,5)*=0{\text{\tiny\textbullet}}="+";
(25,-5)*=0{\text{\tiny\textbullet}}="+";
(-25,5)*{}; (-25,-5)*{} **\dir{-};
(-15,5)*{}; (-15,-5)*{} **\dir{-};
(-5,5)*{}; (-5,-5)*{} **\dir{-};
(5,5)*{}; (5,-5)*{} **\dir{-};
(15,5)*{}; (15,-5)*{} **\dir{-};
(-5,5)*{}; (5,5)*{} **\dir{-};
(-5,-5)*{}; (5,-5)*{} **\dir{-};
(25,5)*{}; (25,-5)*{} **\dir{-};
\endxy\,.
\end{align*}
We also let $\mathfrak{S}_k$ denote the symmetric group on $k$ letters which is generated by $s_1,\ldots,s_{k-1}$. If $u\in\mathfrak{S}_k\subset\mathcal{A}_k(z)$, and $\rho\in\mathcal{A}_k(z)$, we will sometimes write $\rho^u=u\rho u^{-1}$. Let $*:\mathcal{A}_k(z)\to \mathcal{A}_k(z)$ denote the algebra anti--involution which, given $\rho\in A_k$, for $i=1,\ldots k$, interchanges $i$ and $i'$ in $\rho$. Then $*$ reflects each element of the diagram basis for $\mathcal{A}_k(z)$ in the horizontal axis, and satisfies
\begin{align*}
u^*=u^{-1}&&\text{(for $u\in\mathfrak{S}_k$)}
\end{align*}
and 
\begin{align*}
p_i^*=p_i&&\text{(for $i=1,\ldots,k$)}&&\text{and}&&p_{j+\half}^*=p_{j+\half}&&\text{(for $j=1,\ldots,k-1$).}
\end{align*}
Restricting the map $*$ from $\mathcal{A}_{k}(z)$ to $\mathcal{A}_{k-\half}(z)$, gives an algebra anti involution of $\mathcal{A}_{k-\half}(z)$ which we also denote by $*$. 
\section{Jucys--Murphy Elements}\label{a-1-2}
In this section we recursively define a family of Jucys--Murphy elements in $\mathcal{A}_k(z)$ and $\mathcal{A}_{k+\half}(z)$. It will be shown in \S\ref{a-1-5} that the recursive formula given below is equivalent to the combinatorial definition of Jucys--Murphy elements given by Halverson and Ram~\cite{HR:2005}. 

Let $(\sigma_i:i=1,2,\dots)$ and $(L_i:i=0,1,\dots)$ be given by 
\begin{align*}
L_0=0, &&L_1=p_1, &&\sigma_1=1,&&\text{and,}&&\sigma_2=s_1,
\end{align*} 
and, for $i=1,2,\dots,$ 
\begin{align}\label{jm-i-1}
L_{i+1}=-s_iL_ip_{i+\half}-p_{i+\half}L_is_i+p_{i+\half}L_ip_{i+1}p_{i+\half}+s_{i}L_is_{i}+\sigma_{i+1},
\end{align}
where, for $i=2,3,\dots,$
\begin{multline}\label{sigma-i-1}
\sigma_{i+1}=s_{i-1}s_i\sigma_is_is_{i-1}+s_ip_{i-\half}L_{i-1}s_ip_{i-\half}s_i+p_{i-\half}L_{i-1}s_ip_{i-\half}\\
-s_ip_{i-\half}L_{i-1}s_{i-1}p_{i+\half}p_ip_{i-\half}-p_{i-\half}p_ip_{i+\half}s_{i-1}L_{i-1}p_{i-\half}s_i.
\end{multline}
Define $(\sigma_{i+\half}:i=1,2,\dots)$ and $(L_{i+\half}:i=0,1,\dots)$ by 
\begin{align*}
L_{\half}=0,&&\sigma_{\half}=1,&&\text{and,}&&\sigma_{1+\half}=1,
\end{align*}
and, for $i=1,2,\dots,$
\begin{align}\label{jm-i-2}
L_{i+\half}=-L_ip_{i+\half}-p_{i+\half}L_i+p_{i+\half}L_ip_{i}p_{i+\half}+s_iL_{i-\half}s_i+\sigma_{i+\half},
\end{align}
where, for $i=2,3,\dots,$
\begin{multline}\label{sigma-i-2}
\sigma_{i+\half}=s_{i-1}s_i\sigma_{i-\half}s_is_{i-1}+p_{i-\half}L_{i-1}s_ip_{i-\half}s_i+s_ip_{i-\half}L_{i-1}s_ip_{i-\half}\\
-p_{i-\half}L_{i-1}s_{i-1}p_{i+\half}p_ip_{i-\half}-s_ip_{i-\half}p_ip_{i+\half}s_{i-1}L_{i-1}p_{i-\half}s_i.
\end{multline}
Rewriting the last summand in~\eqref{sigma-i-2} as 
\begin{align*}
s_ip_{i-\half}p_ip_{i+\half}s_{i-1}L_{i-1}p_{i-\half}s_i
&=s_{i-1}p_{i+\half}s_{i-1}s_ip_ip_{i+\half}s_{i-1}L_{i-1}s_is_{i-1}p_{i+\half}s_{i-1}\\
&=s_{i-1}p_{i+\half}s_{i-1}s_ip_{i+1}p_{i+\half}s_{i-1}s_iL_{i-1}s_{i-1}p_{i+\half}s_{i-1}\\
&=s_{i-1}p_{i+\half}s_{i-1}p_{i+1}s_{i-1}s_ip_{i-\half}L_{i-1}s_{i-1}p_{i+\half}s_{i-1}\\
&=s_{i-1}p_{i+\half}p_{i}p_{i-\half}L_{i-1}s_{i-1}p_{i+\half}s_{i-1},
\end{align*}
the expression~\eqref{sigma-i-2} becomes 
 \begin{multline}\label{sigma-i-3}
\sigma_{i+\half}=s_{i-1}s_i\sigma_{i-\half}s_is_{i-1}+p_{i-\half}L_{i-1}s_{i-1}p_{i+\half}s_{i-1}+s_{i-1}p_{i+\half}s_{i-1}L_{i-1}p_{i-\half}\\
-p_{i-\half}L_{i-1}s_{i-1}p_{i+\half}p_ip_{i-\half}-s_{i-1}p_{i+\half}p_{i}p_{i-\half}L_{i-1}s_{i-1}p_{i+\half}s_{i-1}.
\end{multline}
Using induction, it follows that if $i=0,1,\ldots,$ then $\sigma_{i+\half}\in\mathcal{A}_{i+\half}(z)$, and $L_{i+\half}\in\mathcal{A}_{i+\half}(z)$. Observe that if $i=0,1,\dots,$ then  $(L_i)^*=L_i$ and $(\sigma_{i+1})^*=\sigma_{i+1}$.  The fact that $(L_{i+\half})^*=L_{i+\half}$ and $(\sigma_{i+\half})^*=\sigma_{i+\half}$ will be shown in Proposition~\ref{invariant}.
\begin{example}
In terms of the diagram presentation for $\mathcal{A}_3(z)$, we have:
\begin{align*}
&L_{2}=\,\xy
(-56,4)*=0{\text{\tiny\textbullet}}="+";
(-48,4)*=0{\text{\tiny\textbullet}}="+";
(-40,4)*=0{\text{\tiny\textbullet}}="+";
(-32,4)*=0{\text{\tiny\textbullet}}="+";
(-24,4)*=0{\text{\tiny\textbullet}}="+";
(-16,4)*=0{\text{\tiny\textbullet}}="+";
(-8,4)*=0{\text{\tiny\textbullet}}="+";
(0,4)*=0{\text{\tiny\textbullet}}="+";
(8,4)*=0{\text{\tiny\textbullet}}="+";
(16,4)*=0{\text{\tiny\textbullet}}="+";
(24,4)*=0{\text{\tiny\textbullet}}="+";
(32,4)*=0{\text{\tiny\textbullet}}="+";
(40,4)*=0{\text{\tiny\textbullet}}="+";
(48,4)*=0{\text{\tiny\textbullet}}="+";
(56,4)*=0{\text{\tiny\textbullet}}="+";
(-56,-4)*=0{\text{\tiny\textbullet}}="+";
(-48,-4)*=0{\text{\tiny\textbullet}}="+";
(-40,-4)*=0{\text{\tiny\textbullet}}="+";
(-32,-4)*=0{\text{\tiny\textbullet}}="+";
(-24,-4)*=0{\text{\tiny\textbullet}}="+";
(-16,-4)*=0{\text{\tiny\textbullet}}="+";
(-8,-4)*=0{\text{\tiny\textbullet}}="+";
(0,-4)*=0{\text{\tiny\textbullet}}="+";
(8,-4)*=0{\text{\tiny\textbullet}}="+";
(16,-4)*=0{\text{\tiny\textbullet}}="+";
(24,-4)*=0{\text{\tiny\textbullet}}="+";
(32,-4)*=0{\text{\tiny\textbullet}}="+";
(40,-4)*=0{\text{\tiny\textbullet}}="+";
(48,-4)*=0{\text{\tiny\textbullet}}="+";
(56,-4)*=0{\text{\tiny\textbullet}}="+";
(-56,4)*{}; (-56,-4)*{} **\dir{-};
(-40,4)*{}; (-40,-4)*{} **\dir{-};
(-32,4)*{}; (-32,-4)*{} **\dir{-};
(-16,4)*{}; (-16,-4)*{} **\dir{-};
(8,4)*{}; (8,-4)*{} **\dir{-};
(16,4)*{}; (16,-4)*{} **\dir{-};
(32,4)*{}; (32,-4)*{} **\dir{-};
(56,4)*{}; (56,-4)*{} **\dir{-};
(-56,-4)*{}; (-48,-4)*{} **\dir{-};
(-32,4)*{}; (-24,4)*{} **\dir{-};
(-8,4)*{}; (0,4)*{} **\dir{-};
(-8,-4)*{}; (0,-4)*{} **\dir{-};
(40,4)*{}; (48,-4)*{} **\dir{-};
(40,-4)*{}; (48,4)*{} **\dir{-};
(-59.5,0)*=0{-}="+";
(-36,0)*=0{-}="+";
(-12,0)*=0{+}="+";
(12,0)*=0{+}="+";
(36,0)*=0{+}="+";
\endxy
\intertext{and} 
&\sigma_{3}=\quad\,\xy
(-56,4)*=0{\text{\tiny\textbullet}}="+";
(-48,4)*=0{\text{\tiny\textbullet}}="+";
(-40,4)*=0{\text{\tiny\textbullet}}="+";
(-32,4)*=0{\text{\tiny\textbullet}}="+";
(-24,4)*=0{\text{\tiny\textbullet}}="+";
(-16,4)*=0{\text{\tiny\textbullet}}="+";
(-8,4)*=0{\text{\tiny\textbullet}}="+";
(0,4)*=0{\text{\tiny\textbullet}}="+";
(8,4)*=0{\text{\tiny\textbullet}}="+";
(16,4)*=0{\text{\tiny\textbullet}}="+";
(24,4)*=0{\text{\tiny\textbullet}}="+";
(32,4)*=0{\text{\tiny\textbullet}}="+";
(40,4)*=0{\text{\tiny\textbullet}}="+";
(48,4)*=0{\text{\tiny\textbullet}}="+";
(56,4)*=0{\text{\tiny\textbullet}}="+";
(-56,-4)*=0{\text{\tiny\textbullet}}="+";
(-48,-4)*=0{\text{\tiny\textbullet}}="+";
(-40,-4)*=0{\text{\tiny\textbullet}}="+";
(-32,-4)*=0{\text{\tiny\textbullet}}="+";
(-24,-4)*=0{\text{\tiny\textbullet}}="+";
(-16,-4)*=0{\text{\tiny\textbullet}}="+";
(-8,-4)*=0{\text{\tiny\textbullet}}="+";
(0,-4)*=0{\text{\tiny\textbullet}}="+";
(8,-4)*=0{\text{\tiny\textbullet}}="+";
(16,-4)*=0{\text{\tiny\textbullet}}="+";
(24,-4)*=0{\text{\tiny\textbullet}}="+";
(32,-4)*=0{\text{\tiny\textbullet}}="+";
(40,-4)*=0{\text{\tiny\textbullet}}="+";
(48,-4)*=0{\text{\tiny\textbullet}}="+";
(56,-4)*=0{\text{\tiny\textbullet}}="+";
(-56,4)*{}; (-56,-4)*{} **\dir{-};
(16,4)*{}; (16,-4)*{} **\dir{-};
(40,4)*{}; (40,-4)*{} **\dir{-};
(-48,4)*{}; (-40,-4)*{} **\dir{-};
(-48,-4)*{}; (-40,4)*{} **\dir{-};
(-32,4)*{}; (-24,-4)*{} **\dir{-};
(-32,-4)*{}; (-24,4)*{} **\dir{-};
(-24,4)*{}; (-16,-4)*{} **\dir{-};
(-24,-4)*{}; (-16,4)*{} **\dir{-};
(-8,4)*{}; (0,4)*{} **\dir{-};
(-8,-4)*{}; (0,-4)*{} **\dir{-};
(0,4)*{}; (8,-4)*{} **\dir{-};
(0,-4)*{}; (8,4)*{} **\dir{-};
(16,-4)*{}; (24,-4)*{} **\dir{-};
(24,4)*{}; (32,-4)*{} **\dir{-};
(24,-4)*{}; (32,4)*{} **\dir{-};
(40,4)*{}; (48,4)*{} **\dir{-};
(48,4)*{}; (56,-4)*{} **\dir{-};
(48,-4)*{}; (56,4)*{} **\dir{-};
(-36,0)*=0{+}="+";
(-12,0)*=0{+}="+";
(12,0)*=0{-}="+";
(36,0)*=0{-}="+";
\endxy\,\,.
\intertext{Using the above expressions for $L_2$ and $\sigma_3$, we can write $L_3$ as}            
&\quad\xy
(-66,24)*=0{L_3=}="+";
(-56,28)*=0{\text{\tiny\textbullet}}="+";
(-48,28)*=0{\text{\tiny\textbullet}}="+";
(-40,28)*=0{\text{\tiny\textbullet}}="+";
(-32,28)*=0{\text{\tiny\textbullet}}="+";
(-24,28)*=0{\text{\tiny\textbullet}}="+";
(-16,28)*=0{\text{\tiny\textbullet}}="+";
(-8,28)*=0{\text{\tiny\textbullet}}="+";
(-0,28)*=0{\text{\tiny\textbullet}}="+";
(8,28)*=0{\text{\tiny\textbullet}}="+";
(16,28)*=0{\text{\tiny\textbullet}}="+";
(24,28)*=0{\text{\tiny\textbullet}}="+";
(32,28)*=0{\text{\tiny\textbullet}}="+";
(40,28)*=0{\text{\tiny\textbullet}}="+";
(48,28)*=0{\text{\tiny\textbullet}}="+";
(56,28)*=0{\text{\tiny\textbullet}}="+";
(-56,20)*=0{\text{\tiny\textbullet}}="+";
(-48,20)*=0{\text{\tiny\textbullet}}="+";
(-40,20)*=0{\text{\tiny\textbullet}}="+";
(-32,20)*=0{\text{\tiny\textbullet}}="+";
(-24,20)*=0{\text{\tiny\textbullet}}="+";
(-16,20)*=0{\text{\tiny\textbullet}}="+";
(-8,20)*=0{\text{\tiny\textbullet}}="+";
(-0,20)*=0{\text{\tiny\textbullet}}="+";
(8,20)*=0{\text{\tiny\textbullet}}="+";
(16,20)*=0{\text{\tiny\textbullet}}="+";
(24,20)*=0{\text{\tiny\textbullet}}="+";
(32,20)*=0{\text{\tiny\textbullet}}="+";
(40,20)*=0{\text{\tiny\textbullet}}="+";
(48,20)*=0{\text{\tiny\textbullet}}="+";
(56,20)*=0{\text{\tiny\textbullet}}="+";
(-56,16)*=0{\text{\tiny\textbullet}}="+";
(-48,16)*=0{\text{\tiny\textbullet}}="+";
(-40,16)*=0{\text{\tiny\textbullet}}="+";
(-32,16)*=0{\text{\tiny\textbullet}}="+";
(-24,16)*=0{\text{\tiny\textbullet}}="+";
(-16,16)*=0{\text{\tiny\textbullet}}="+";
(-8,16)*=0{\text{\tiny\textbullet}}="+";
(-0,16)*=0{\text{\tiny\textbullet}}="+";
(8,16)*=0{\text{\tiny\textbullet}}="+";
(16,16)*=0{\text{\tiny\textbullet}}="+";
(24,16)*=0{\text{\tiny\textbullet}}="+";
(32,16)*=0{\text{\tiny\textbullet}}="+";
(40,16)*=0{\text{\tiny\textbullet}}="+";
(48,16)*=0{\text{\tiny\textbullet}}="+";
(56,16)*=0{\text{\tiny\textbullet}}="+";
(-56,8)*=0{\text{\tiny\textbullet}}="+";
(-48,8)*=0{\text{\tiny\textbullet}}="+";
(-40,8)*=0{\text{\tiny\textbullet}}="+";
(-32,8)*=0{\text{\tiny\textbullet}}="+";
(-24,8)*=0{\text{\tiny\textbullet}}="+";
(-16,8)*=0{\text{\tiny\textbullet}}="+";
(-8,8)*=0{\text{\tiny\textbullet}}="+";
(-0,8)*=0{\text{\tiny\textbullet}}="+";
(8,8)*=0{\text{\tiny\textbullet}}="+";
(16,8)*=0{\text{\tiny\textbullet}}="+";
(24,8)*=0{\text{\tiny\textbullet}}="+";
(32,8)*=0{\text{\tiny\textbullet}}="+";
(40,8)*=0{\text{\tiny\textbullet}}="+";
(48,8)*=0{\text{\tiny\textbullet}}="+";
(56,8)*=0{\text{\tiny\textbullet}}="+";
(-56,4)*=0{\text{\tiny\textbullet}}="+";
(-48,4)*=0{\text{\tiny\textbullet}}="+";
(-40,4)*=0{\text{\tiny\textbullet}}="+";
(-32,4)*=0{\text{\tiny\textbullet}}="+";
(-24,4)*=0{\text{\tiny\textbullet}}="+";
(-16,4)*=0{\text{\tiny\textbullet}}="+";
(-8,4)*=0{\text{\tiny\textbullet}}="+";
(-0,4)*=0{\text{\tiny\textbullet}}="+";
(8,4)*=0{\text{\tiny\textbullet}}="+";
(16,4)*=0{\text{\tiny\textbullet}}="+";
(24,4)*=0{\text{\tiny\textbullet}}="+";
(32,4)*=0{\text{\tiny\textbullet}}="+";
(40,4)*=0{\text{\tiny\textbullet}}="+";
(48,4)*=0{\text{\tiny\textbullet}}="+";
(56,4)*=0{\text{\tiny\textbullet}}="+";
(-56,-4)*=0{\text{\tiny\textbullet}}="+";
(-48,-4)*=0{\text{\tiny\textbullet}}="+";
(-40,-4)*=0{\text{\tiny\textbullet}}="+";
(-32,-4)*=0{\text{\tiny\textbullet}}="+";
(-24,-4)*=0{\text{\tiny\textbullet}}="+";
(-16,-4)*=0{\text{\tiny\textbullet}}="+";
(-8,-4)*=0{\text{\tiny\textbullet}}="+";
(-0,-4)*=0{\text{\tiny\textbullet}}="+";
(8,-4)*=0{\text{\tiny\textbullet}}="+";
(16,-4)*=0{\text{\tiny\textbullet}}="+";
(24,-4)*=0{\text{\tiny\textbullet}}="+";
(32,-4)*=0{\text{\tiny\textbullet}}="+";
(40,-4)*=0{\text{\tiny\textbullet}}="+";
(48,-4)*=0{\text{\tiny\textbullet}}="+";
(56,-4)*=0{\text{\tiny\textbullet}}="+";
(-56,-16)*=0{\text{\tiny\textbullet}}="+";
(-48,-16)*=0{\text{\tiny\textbullet}}="+";
(-40,-16)*=0{\text{\tiny\textbullet}}="+";
(-32,-16)*=0{\text{\tiny\textbullet}}="+";
(-24,-16)*=0{\text{\tiny\textbullet}}="+";
(-16,-16)*=0{\text{\tiny\textbullet}}="+";
(-8,-16)*=0{\text{\tiny\textbullet}}="+";
(-0,-16)*=0{\text{\tiny\textbullet}}="+";
(8,-16)*=0{\text{\tiny\textbullet}}="+";
(16,-16)*=0{\text{\tiny\textbullet}}="+";
(24,-16)*=0{\text{\tiny\textbullet}}="+";
(32,-16)*=0{\text{\tiny\textbullet}}="+";
(40,-16)*=0{\text{\tiny\textbullet}}="+";
(48,-16)*=0{\text{\tiny\textbullet}}="+";
(56,-16)*=0{\text{\tiny\textbullet}}="+";
(-56,-8)*=0{\text{\tiny\textbullet}}="+";
(-48,-8)*=0{\text{\tiny\textbullet}}="+";
(-40,-8)*=0{\text{\tiny\textbullet}}="+";
(-32,-8)*=0{\text{\tiny\textbullet}}="+";
(-24,-8)*=0{\text{\tiny\textbullet}}="+";
(-16,-8)*=0{\text{\tiny\textbullet}}="+";
(-8,-8)*=0{\text{\tiny\textbullet}}="+";
(-0,-8)*=0{\text{\tiny\textbullet}}="+";
(8,-8)*=0{\text{\tiny\textbullet}}="+";
(16,-8)*=0{\text{\tiny\textbullet}}="+";
(24,-8)*=0{\text{\tiny\textbullet}}="+";
(32,-8)*=0{\text{\tiny\textbullet}}="+";
(40,-8)*=0{\text{\tiny\textbullet}}="+";
(48,-8)*=0{\text{\tiny\textbullet}}="+";
(56,-8)*=0{\text{\tiny\textbullet}}="+";
(-56,-28)*=0{\text{\tiny\textbullet}}="+";
(-48,-28)*=0{\text{\tiny\textbullet}}="+";
(-40,-28)*=0{\text{\tiny\textbullet}}="+";
(-32,-28)*=0{\text{\tiny\textbullet}}="+";
(-24,-28)*=0{\text{\tiny\textbullet}}="+";
(-16,-28)*=0{\text{\tiny\textbullet}}="+";
(-8,-28)*=0{\text{\tiny\textbullet}}="+";
(-0,-28)*=0{\text{\tiny\textbullet}}="+";
(8,-28)*=0{\text{\tiny\textbullet}}="+";
(16,-28)*=0{\text{\tiny\textbullet}}="+";
(24,-28)*=0{\text{\tiny\textbullet}}="+";
(32,-28)*=0{\text{\tiny\textbullet}}="+";
(40,-28)*=0{\text{\tiny\textbullet}}="+";
(48,-28)*=0{\text{\tiny\textbullet}}="+";
(56,-28)*=0{\text{\tiny\textbullet}}="+";
(-56,-20)*=0{\text{\tiny\textbullet}}="+";
(-48,-20)*=0{\text{\tiny\textbullet}}="+";
(-40,-20)*=0{\text{\tiny\textbullet}}="+";
(-32,-20)*=0{\text{\tiny\textbullet}}="+";
(-24,-20)*=0{\text{\tiny\textbullet}}="+";
(-16,-20)*=0{\text{\tiny\textbullet}}="+";
(-8,-20)*=0{\text{\tiny\textbullet}}="+";
(-0,-20)*=0{\text{\tiny\textbullet}}="+";
(8,-20)*=0{\text{\tiny\textbullet}}="+";
(16,-20)*=0{\text{\tiny\textbullet}}="+";
(24,-20)*=0{\text{\tiny\textbullet}}="+";
(32,-20)*=0{\text{\tiny\textbullet}}="+";
(40,-20)*=0{\text{\tiny\textbullet}}="+";
(48,-20)*=0{\text{\tiny\textbullet}}="+";
(56,-20)*=0{\text{\tiny\textbullet}}="+";
(-56,28)*{}; (-48,28)*{} **\dir{-};
(40,28)*{}; (48,28)*{} **\dir{-};
(-56,20)*{}; (-40,20)*{} **\dir{-};
(-24,20)*{}; (-16,20)*{} **\dir{-};
(-8,20)*{}; (8,20)*{} **\dir{-};
(24,20)*{}; (32,20)*{} **\dir{-};
(48,20)*{}; (56,20)*{} **\dir{-};
(-48,16)*{}; (-40,16)*{} **\dir{-};
(-32,16)*{}; (-16,16)*{} **\dir{-};
(-8,16)*{}; (8,16)*{} **\dir{-};
(24,16)*{}; (32,16)*{} **\dir{-};
(48,16)*{}; (56,16)*{} **\dir{-};
(-32,8)*{}; (-24,8)*{} **\dir{-};
(40,8)*{}; (48,8)*{} **\dir{-};
(-48,4)*{}; (-40,4)*{} **\dir{-};
(-32,4)*{}; (-16,4)*{} **\dir{-};
(-8,4)*{}; (8,4)*{} **\dir{-};
(-8,4)*{}; (8,4)*{} **\dir{-};
(24,4)*{}; (32,4)*{} **\dir{-};
(48,4)*{}; (56,4)*{} **\dir{-};
(-56,-4)*{}; (-40,-4)*{} **\dir{-};
(-24,-4)*{}; (-16,-4)*{} **\dir{-};
(-8,-4)*{}; (8,-4)*{} **\dir{-};
(24,-4)*{}; (32,-4)*{} **\dir{-};
(48,-4)*{}; (56,-4)*{} **\dir{-};
(-8,-20)*{}; (0,-20)*{} **\dir{-};
(-8,-28)*{}; (0,-28)*{} **\dir{-};
(16,-28)*{}; (24,-28)*{} **\dir{-};
(40,-20)*{}; (48,-20)*{} **\dir{-};
(-56,28)*{}; (-56,20)*{} **\dir{-};
(-48,28)*{}; (-48,20)*{} **\dir{-};
(-32,28)*{}; (-32,20)*{} **\dir{-};
(-24,28)*{}; (-24,20)*{} **\dir{-};
(0,28)*{}; (0,20)*{} **\dir{-};
(16,28)*{}; (16,20)*{} **\dir{-};
(24,28)*{}; (24,20)*{} **\dir{-};
(48,28)*{}; (48,20)*{} **\dir{-};
(-56,16)*{}; (-56,8)*{} **\dir{-};
(-48,16)*{}; (-48,8)*{} **\dir{-};
(-24,16)*{}; (-24,8)*{} **\dir{-};
(-32,16)*{}; (-32,8)*{} **\dir{-};
(0,16)*{}; (0,8)*{} **\dir{-};
(16,16)*{}; (16,8)*{} **\dir{-};
(24,16)*{}; (24,8)*{} **\dir{-};
(48,16)*{}; (48,8)*{} **\dir{-};
(-56,4)*{}; (-56,-4)*{} **\dir{-};
(-32,4)*{}; (-32,-4)*{} **\dir{-};
(16,4)*{}; (16,-4)*{} **\dir{-};
(-56,-8)*{}; (-56,-16)*{} **\dir{-};
(-48,-8)*{}; (-48,-16)*{} **\dir{-};
(-32,-8)*{}; (-32,-16)*{} **\dir{-};
(-24,-8)*{}; (-24,-16)*{} **\dir{-};
(0,-8)*{}; (0,-16)*{} **\dir{-};
(16,-8)*{}; (16,-16)*{} **\dir{-};
(24,-8)*{}; (24,-16)*{} **\dir{-};
(48,-8)*{}; (48,-16)*{} **\dir{-};
(-56,-20)*{}; (-56,-28)*{} **\dir{-};
(16,-20)*{}; (16,-28)*{} **\dir{-};
(40,-20)*{}; (40,-28)*{} **\dir{-};
(-48,-20)*{}; (-40,-28)*{} **\dir{-};
(-40,-20)*{}; (-48,-28)*{} **\dir{-};
(-32,-20)*{}; (-24,-28)*{} **\dir{-};
(-32,-28)*{}; (-24,-20)*{} **\dir{-};
(-24,-20)*{}; (-16,-28)*{} **\dir{-};
(-16,-20)*{}; (-24,-28)*{} **\dir{-};
(0,-20)*{}; (8,-28)*{} **\dir{-};
(8,-20)*{}; (0,-28)*{} **\dir{-};
(24,-20)*{}; (32,-28)*{} **\dir{-};
(32,-20)*{}; (24,-28)*{} **\dir{-};
(48,-20)*{}; (56,-28)*{} **\dir{-};
(56,-20)*{}; (48,-28)*{} **\dir{-};
(40,-4)*{}; (48,4)*{} **\dir{-};
(48,-4)*{}; (40,4)*{} **\dir{-};
(40,20)*{}; (56,28)*{} **\dir{-};
(40,16)*{}; (56,8)*{} **\dir{-};
(-32,28)*{}; (-16,28)*{} **\crv{(-24,24)};
(-8,28)*{}; (8,28)*{} **\crv{(0,24)};
(-56,8)*{}; (-40,8)*{} **\crv{(-48,12)};
(-8,8)*{}; (8,8)*{} **\crv{(0,12)};
(-56,-16)*{}; (-40,-16)*{} **\crv{(-48,-12)};
(-32,-8)*{}; (-16,-8)*{} **\crv{(-24,-12)};
(-8,-8)*{}; (8,-8)*{} **\crv{(0,-12)};
(-8,-16)*{}; (8,-16)*{} **\crv{(0,-12)};
(40,-8)*{}; (56,-16)*{} **\crv{(52,-12)};
(40,-16)*{}; (56,-8)*{} **\crv{(52,-12)};
(-36,24)*=0{+}="+";
(-12,24)*=0{-}="+";
(12,24)*=0{-}="+";
(36,24)*=0{-}="+";
(-59.5,12)*=0{+}="+";
(-36,12)*=0{+}="+";
(-12,12)*=0{-}="+";
(12,12)*=0{-}="+";
(36,12)*=0{-}="+";
(-59.5,0)*=0{-}="+";
(-36,0)*=0{-}="+";
(-12,0)*=0{+}="+";
(12,0)*=0{+}="+";
(36,0)*=0{+}="+";
(-59.5,-12)*=0{-}="+";
(-36,-12)*=0{-}="+";
(-12,-12)*=0{+}="+";
(12,-12)*=0{+}="+";
(36,-12)*=0{+}="+";
(-59.5,-24)*=0{+}="+";
(-36,-24)*=0{+}="+";
(-12,-24)*=0{-}="+";
(12,-24)*=0{-}="+";
(36,-24)*=0{-}="+";
\endxy
\end{align*}
\end{example} 
The relations given in the next proposition are fundamental for subsequent calculations. 
\begin{proposition}\label{prel:a}
For $i=1,2,\dots,$ the following statements hold: 
\begin{enumerate}[label=(\arabic{*}), ref=\arabic{*},leftmargin=0pt,itemindent=1.5em]
\item $\sigma_{i+1}p_{i+\half}=p_{i+\half}$,\label{prel:a0}
\item $s_{i+1}\sigma_{i+1}p_{i+\thalf}=p_{i+\half}s_{i+1}\sigma_{i+1}$,\label{prel:a1}
\item $\sigma_{i+1}p_ip_{i+\half}=s_iL_ip_{i+\half}$,\label{prel:a3}
\item $\sigma_{i+1}p_{i+1}p_{i+\half}=L_ip_{i+\half}$,\label{prel:a2}
\item $p_{i+\half}L_ip_{i+\half}=p_{i+\half}$.\label{jm:comm:a0}
\end{enumerate}
\end{proposition}
\begin{proof}
(\ref{prel:a0}) We consider each of the terms on the right hand side of the definition 
\begin{multline*}
\sigma_{i+1}p_{i+\half}=s_{i-1}s_i\sigma_is_is_{i-1}p_{i+\half}+s_ip_{i-\half}L_{i-1}s_ip_{i-\half}s_ip_{i+\half}+p_{i-\half}L_{i-1}s_ip_{i-\half}p_{i+\half}\\
-s_ip_{i-\half}L_{i-1}s_{i-1}p_{i+\half}p_ip_{i-\half}p_{i+\half}-p_{i-\half}p_ip_{i+\half}s_{i-1}L_{i-1}p_{i-\half}s_ip_{i+\half},
\end{multline*}
beginning with the argument by induction:
\begin{align*}
s_{i-1}s_i\sigma_is_is_{i-1}p_{i+\half}=s_{i-1}s_i\sigma_ip_{i-\half}s_is_{i-1}=s_{i-1}s_ip_{i-\half}s_is_{i-1}=p_{i+\half}.
\end{align*}
From the fact that $p_{i+\half}$ commutes with $L_{i-1}$, $s_ip_{i-\half}L_{i-1}s_ip_{i-\half}s_ip_{i+\half}=p_{i-\half}L_{i-1}p_{i-\half}p_{i+\half}$ and $p_{i-\half}L_{i-1}s_ip_{i-\half}p_{i+\half}=p_{i-\half}L_{i-1}p_{i-\half}p_{i+\half}$, while
\begin{align*}
s_ip_{i-\half}L_{i-1}s_{i-1}p_{i+\half}p_ip_{i-\half}p_{i+\half}&=s_ip_{i-\half}L_{i-1}s_{i-1}p_{i+\half}p_{i-\half}=p_{i-\half}L_{i-1}p_{i-\half}p_{i+\half}
\end{align*}
and, from the relation $p_{i-\half}p_ip_{i+\half}s_{i-1}=p_{i-\half}p_is_{i-1}s_ip_{i-\half}s_{i}$, we obtain
\begin{align*}
p_{i-\half}p_ip_{i+\half}s_{i-1}L_{i-1}p_{i-\half}s_ip_{i+\half}&=p_{i-\half}p_is_{i-1}s_ip_{i-\half}s_{i}L_{i-1}p_{i-\half}s_ip_{i+\half}\\
&=p_{i-\half}s_{i-1}p_{i-1}s_ip_{i-\half}s_{i}L_{i-1}p_{i-\half}s_ip_{i+\half}\\
&=p_{i-\half}p_{i-1}s_ip_{i-\half}s_{i}L_{i-1}p_{i-\half}s_ip_{i+\half}\\
&=p_{i-\half}p_{i-1}p_{i-\half}L_{i-1}p_{i-\half}p_{i+\half}\\
&=p_{i-\half}L_{i-1}p_{i-\half}p_{i+\half}.
\end{align*}
Substituting the terms obtained above into the definition of $\sigma_{i+1}p_{i+\half}$, we observe that all terms vanish except for $p_{i+\half}$, which completes the proof of~(\ref{prel:a0}). \newline 
(\ref{prel:a1}) The definition~\eqref{sigma-i-1} gives
\begin{multline*}
s_{i+1}\sigma_{i+1}p_{i+\thalf}=s_{i+1}s_{i-1}s_i\sigma_is_is_{i-1}p_{i+\thalf}+s_{i+1}s_ip_{i-\half}L_{i-1}s_ip_{i-\half}s_ip_{i+\thalf}\\+s_{i+1}p_{i-\half}L_{i-1}s_ip_{i-\half}p_{i+\thalf}-s_{i+1}s_ip_{i-\half}L_{i-1}s_{i-1}p_{i+\half}p_ip_{i-\half}p_{i+\thalf}\\-s_{i+1}p_{i-\half}p_ip_{i+\half}s_{i-1}L_{i-1}p_{i-\half}s_ip_{i+\thalf}.
\end{multline*}
Now consider each of the terms on the right hand side of the above equality. Firstly,
\begin{align*}
s_{i+1}s_{i-1}s_i\sigma_is_is_{i-1}p_{i+\thalf}&=s_{i+1}s_{i-1}s_is_{i+1}\sigma_is_{i+1}s_is_{i-1}p_{i+\thalf}\\
&=s_{i-1}s_{i+1}s_is_{i+1}\sigma_is_{i+1}s_ip_{i+\thalf}s_{i-1}\\
&=s_{i-1}s_{i}s_{i+1}s_{i}\sigma_ip_{i+\half}s_{i+1}s_is_{i-1}\\
&=s_{i-1}s_{i}s_{i+1}p_{i-\half}s_{i}\sigma_is_{i+1}s_is_{i-1}&&\text{(by induction)}\\
&=s_{i-1}s_{i}p_{i-\half}s_{i+1}s_{i}\sigma_is_{i+1}s_is_{i-1}\\
&=p_{i+\half}s_{i-1}s_{i}s_{i+1}s_{i}\sigma_is_{i+1}s_is_{i-1}\\
&=p_{i+\half}s_{i-1}s_{i+1}s_{i}s_{i+1}\sigma_is_{i+1}s_is_{i-1}\\
&=p_{i+\half}s_{i+1}s_{i-1}s_i\sigma_is_is_{i-1}.
\end{align*}
Next, use the fact that $s_ip_{i+\thalf}s_i$ and $p_{i-\half}$ commute to observe that
\begin{align*}
s_{i+1}s_ip_{i-\half}L_{i-1}s_ip_{i-\half}s_ip_{i+\thalf}&=s_{i+1}s_ip_{i-\half}L_{i-1}s_ip_{i-\half}(s_ip_{i+\thalf}s_i)s_i\\
&=s_{i+1}s_ip_{i-\half}L_{i-1}s_i(s_ip_{i+\thalf}s_i)p_{i-\half}s_i\\
&=s_{i+1}s_ip_{i-\half}L_{i-1}p_{i+\thalf}s_ip_{i-\half}s_i\\
&=s_{i+1}s_ip_{i+\thalf}p_{i-\half}L_{i-1}s_ip_{i-\half}s_i\\
&=p_{i+\half}s_{i+1}s_ip_{i-\half}L_{i-1}s_ip_{i-\half}s_i,
\end{align*}
and that
\begin{align*}
s_{i+1}p_{i-\half}L_{i-1}s_ip_{i-\half}p_{i+\thalf}&=s_{i+1}p_{i-\half}L_{i-1}s_ip_{i+\thalf}p_{i-\half}\\
&=s_{i+1}p_{i-\half}L_{i-1}(s_ip_{i+\thalf}s_i)s_ip_{i-\half}\\
&=s_{i+1}p_{i-\half}(s_ip_{i+\thalf}s_i)L_{i-1}s_ip_{i-\half}\\
&=s_{i+1}(s_ip_{i+\thalf}s_i)p_{i-\half}L_{i-1}s_ip_{i-\half}\\
&=p_{i+\half}s_{i+1}s_is_ip_{i-\half}L_{i-1}s_ip_{i-\half}\\
&=p_{i+\half}s_{i+1}p_{i-\half}L_{i-1}s_ip_{i-\half}.
\end{align*}
Since $p_{i+\thalf}$ commutes with $\mathcal{A}_{i+\half}(z)$, we see that 
\begin{align*}
s_{i+1}s_ip_{i-\half}L_{i-1}s_{i-1}p_{i+\half}p_ip_{i-\half}p_{i+\thalf}&=s_{i+1}s_ip_{i+\thalf} p_{i-\half}L_{i-1}s_{i-1}p_{i+\half}p_ip_{i-\half}\\
&=p_{i+\half} s_{i+1}s_ip_{i-\half}L_{i-1}s_{i-1}p_{i+\half}p_ip_{i-\half}.
\end{align*}
For the last term, we use the fact that $(s_ip_{i+\thalf}s_i)$ and $p_{i-\half}$ commute to see that
\begin{align*}
s_{i+1}p_{i-\half}p_ip_{i+\half}s_{i-1}L_{i-1}p_{i-\half}s_ip_{i+\thalf}&=s_{i+1}p_{i-\half}p_ip_{i+\half}s_{i-1}L_{i-1}p_{i-\half}(s_ip_{i+\thalf}s_i)s_i\\
&=s_{i+1}p_{i-\half}p_ip_{i+\half}s_{i-1}L_{i-1}(s_ip_{i+\thalf}s_i)p_{i-\half}s_i\\
&=s_{i+1}p_{i-\half}p_ip_{i+\half}s_{i-1}(s_ip_{i+\thalf}s_i)L_{i-1}p_{i-\half}s_i\\
&=s_{i+1}p_{i-\half}p_ip_{i+\half}(s_{i-1}s_ip_{i+\thalf}s_is_{i-1})s_{i-1}L_{i-1}p_{i-\half}s_i\\
&=s_{i+1}p_{i-\half}(s_{i-1}s_ip_{i-1}p_{i-\half}p_{i+\thalf}s_is_{i-1})s_{i-1}L_{i-1}p_{i-\half}s_i\\
&=s_{i+1}p_{i-\half}(s_{i-1}s_ip_{i+\thalf}s_is_{i-1})p_ip_{i+\half}s_{i-1}L_{i-1}p_{i-\half}s_i\\
&=s_{i+1}p_{i-\half}s_ip_{i+\thalf}s_is_{i-1}p_ip_{i+\half}s_{i-1}L_{i-1}p_{i-\half}s_i\\
&=p_{i-\half}s_{i+1}s_ip_{i+\thalf}s_is_{i-1}p_ip_{i+\half}s_{i-1}L_{i-1}p_{i-\half}s_i\\
&=p_{i-\half}p_{i+\half}s_{i+1}s_i^2s_{i-1}p_ip_{i+\half}s_{i-1}L_{i-1}p_{i-\half}s_i\\
&=p_{i+\half}s_{i+1}p_{i-\half}s_{i-1}p_ip_{i+\half}s_{i-1}L_{i-1}p_{i-\half}s_i.
\end{align*}
Putting the above together, we have
\begin{multline*}
s_{i+1}\sigma_{i+1}p_{i+\thalf}=p_{i+\half}s_{i+1}s_{i-1}s_i\sigma_is_is_{i-1}
+p_{i+\half}s_{i+1}s_ip_{i-\half}L_{i-1}s_ip_{i-\half}s_i\\
+p_{i+\half}s_{i+1}p_{i-\half}L_{i-1}s_ip_{i-\half}
-p_{i+\half} s_{i+1}s_ip_{i-\half}L_{i-1}s_{i-1}p_{i+\half}p_ip_{i-\half}\\
-p_{i+\half}s_{i+1}p_{i-\half}s_{i-1}p_ip_{i+\half}s_{i-1}L_{i-1}p_{i-\half}s_i
=p_{i+\half}s_{i+1}\sigma_{i+1},
\end{multline*}
which completes the proof of~(\ref{prel:a1}). \newline
(\ref{prel:a3}) We consider the terms on the right hand side of the equality
\begin{multline*}
\sigma_{i+1}p_{i}p_{i+\half}=
s_{i-1}s_i\sigma_is_is_{i-1}p_{i}p_{i+\half}
+s_ip_{i-\half}L_{i-1}s_ip_{i-\half}s_ip_{i}p_{i+\half}\\
+p_{i-\half}L_{i-1}s_ip_{i-\half}p_{i}p_{i+\half}
-s_ip_{i-\half}L_{i-1}s_{i-1}p_{i+\half}p_ip_{i-\half}p_{i}p_{i+\half}\\
-p_{i-\half}p_ip_{i+\half}s_{i-1}L_{i-1}p_{i-\half}s_ip_{i}p_{i+\half},
\end{multline*}
beginning with
\begin{align*}
s_{i-1}s_i\sigma_is_is_{i-1}p_{i}p_{i+\half}=&s_{i-1}s_i\sigma_ip_{i-1}s_is_{i-1}p_{i+\half}\\
&=s_{i-1}s_i\sigma_ip_{i-1}p_{i-\half}s_is_{i-1}\\
&=s_{i-1}s_is_{i-1}L_{i-1}p_{i-\half} s_is_{i-1}&&(\text{by induction})\\
&= s_{i}s_{i-1}s_{i}L_{i-1} s_is_{i-1}p_{i+\half}\\
&= s_{i}s_{i-1}L_{i-1} s_{i-1}p_{i+\half}.
\end{align*}
For the second term, we have 
\begin{align*}
s_ip_{i-\half}L_{i-1}s_ip_{i-\half}s_ip_ip_{i+\half}&=s_ip_{i-\half}L_{i-1}s_ip_{i-\half}p_{i+1}s_ip_{i+\half}\\
&=s_ip_{i-\half}L_{i-1}p_is_ip_{i-\half}s_ip_{i+\half}\\
&=s_ip_{i-\half}L_{i-1}p_ip_{i-\half}p_{i+\half},
\end{align*}
and for the third,
\begin{align*}
p_{i-\half}L_{i-1}s_ip_{i-\half}p_{i}p_{i+\half}&= p_{i-\half}s_iL_{i-1}p_{i-\half}p_{i}p_{i+\half}\\
&=p_{i-\half}s_{i}\sigma_ip_ip_{i-\half}p_ip_{i+\half} &&(\text{by induction})\\
&=p_{i-\half}s_{i}\sigma_ip_{i+\half}\\
&=s_{i}\sigma_ip_{i+\half} &&(\text{by item~\eqref{prel:a1}}).
\end{align*}
Using the relation $p_{i}p_{i-\half}p_i=p_i$, we see that the fourth term satisfies
\begin{align*}
s_ip_{i-\half}L_{i-1}s_{i-1}p_{i+\half}p_ip_{i-\half}p_ip_{i+\half}&=s_ip_{i-\half}L_{i-1}s_{i-1}p_{i+\half},
\end{align*}
while, for the fifth term, 
\begin{align*}
p_{i-\half}p_ip_{i+\half}s_{i-1}L_{i-1}p_{i-\half}s_ip_ip_{i+\half}&=p_{i-\half}p_ip_{i+\half}s_{i-1}L_{i-1}p_{i-\half}p_{i+1}s_ip_{i+\half}\\
&=p_{i-\half}p_ip_{i+\half}p_{i+1}s_{i-1}L_{i-1}p_{i-\half}s_ip_{i+\half}\\
&=p_{i-\half}s_i p_{i+1}s_{i-1}L_{i-1}p_{i-\half}p_{i+\half}\\
&=p_{i-\half}s_i s_{i-1}L_{i-1}p_{i-\half}p_{i+1}p_{i+\half}\\
&=s_i s_{i-1}p_{i+\half}L_{i-1}p_{i-\half}p_{i+1}p_{i+\half}\\
&=s_i s_{i-1}L_{i-1}p_{i-\half}p_{i+\half}.
\end{align*}
Combining the above terms, we obtain
\begin{multline*}
\sigma_{i+1}p_ip_{i+\half}=s_{i}s_{i-1}L_{i-1} s_{i-1}p_{i+\half}+s_ip_{i-\half}L_{i-1}p_ip_{i-\half}p_{i+\half} +s_i\sigma_ip_{i+\half}\\
-s_ip_{i-\half}L_{i-1}s_{i-1}p_{i+\half}-s_i s_{i-1}L_{i-1}p_{i-\half}p_{i+\half}=s_iL_ip_{i+\half},
\end{multline*}
which completes the proof of~(\ref{prel:a3}).\newline
(\ref{prel:a2}) We consider each of the terms on the right hand side of the expression
\begin{multline*}
\sigma_{i+1}p_{i+1}p_{i+\half}=s_{i-1}s_i\sigma_is_is_{i-1}p_{i+1}p_{i+\half}+s_ip_{i-\half}L_{i-1}s_ip_{i-\half}s_ip_{i+1}p_{i+\half}\\
+p_{i-\half}L_{i-1}s_ip_{i-\half}p_{i+1}p_{i+\half}-s_ip_{i-\half}L_{i-1}s_{i-1}p_{i+\half}p_ip_{i-\half}p_{i+1}p_{i+\half}\\
-p_{i-\half}p_ip_{i+\half}s_{i-1}L_{i-1}p_{i-\half}s_ip_{i+1}p_{i+\half}.
\end{multline*}
Firstly,
\begin{align*}
s_{i-1}s_i\sigma_is_is_{i-1}p_{i+1}p_{i+\half}&=s_{i-1}s_i\sigma_ip_is_is_{i-1}p_{i+\half}\\
&=s_{i-1}s_i\sigma_ip_ip_{i-\half}s_is_{i-1}\\
&=s_{i-1}s_iL_{i-1}p_{i-\half}s_is_{i-1}&&\text{(by induction)}\\
&=s_{i-1}s_iL_{i-1}s_is_{i-1}p_{i+\half}\\
&=s_{i-1}L_{i-1}s_{i-1}p_{i+\half}.
\end{align*}
For the second term, 
\begin{align*}
s_ip_{i-\half}s_iL_{i-1}p_{i-\half}s_ip_{i+1}p_{i+\half}&=s_ip_{i-\half}s_i\sigma_ip_ip_{i-\half}s_ip_{i+1}p_{i+\half}&&\text{(by induction)}\\
&=s_is_i\sigma_ip_{i+\half}p_ip_{i-\half}s_ip_{i+1}p_{i+\half}&&\text{(by item~(\ref{prel:a1}))}\\
&=\sigma_ip_{i+\half}p_ip_{i-\half}p_is_ip_{i+\half}\\
&=\sigma_ip_{i+\half}p_is_ip_{i+\half}\\
&=\sigma_ip_{i+\half}.
\end{align*}
For the third term,
\begin{align*}
p_{i-\half}L_{i-1}s_ip_{i-\half}p_{i+1}p_{i+\half}&=p_{i-\half}L_{i-1}p_{i}s_ip_{i-\half}p_{i+\half}\\
&=p_{i-\half}L_{i-1}p_{i}p_{i-\half}p_{i+\half}.
\end{align*}
For the fourth term, we use the relation $s_{i-1}p_{i+\half}p_ip_{i-\half}p_{i+1}p_{i+\half}=s_ip_{i-\half}p_{i-1}p_ip_{i-\half}p_{i+\half}$ in
\begin{align*}
s_ip_{i-\half}L_{i-1}s_{i-1}p_{i+\half}p_ip_{i-\half}p_{i+1}p_{i+\half}&=
s_ip_{i-\half}L_{i-1}s_ip_{i-\half}p_{i-1}p_ip_{i-\half}p_{i+\half}\\
&=s_ip_{i-\half}s_iL_{i-1}p_{i-\half}p_{i-1}p_ip_{i-\half}p_{i+\half}\\
&=s_ip_{i-\half}s_i\sigma_ip_ip_{i-\half}p_{i-1}p_ip_{i-\half}p_{i+\half}&&\text{(by induction)}\\
&=s_ip_{i-\half}s_i\sigma_ip_{i-1}p_ip_{i-\half}p_{i+\half}\\
&=s_is_i\sigma_ip_{i+\half}p_{i-1}p_ip_{i-\half}p_{i+\half}&&\text{(by item~(\ref{prel:a1}))}\\
&=\sigma_ip_{i-1}p_{i-\half}p_{i+\half}\\
&=s_{i-1}L_{i-1}p_{i-\half}p_{i+\half}&&\text{(by item (\ref{prel:a3})).}
\end{align*}
For the final term, we use the relation $p_{i-\half}p_ip_{i+\half}s_{i-1}=p_{i-\half}s_ip_{i-1}p_{i-\half}s_i$ in
\begin{align*}
p_{i-\half}p_ip_{i+\half}s_{i-1}L_{i-1}p_{i-\half}s_ip_{i+1}p_{i+\half}&=
p_{i-\half}s_ip_{i-1}p_{i-\half}s_iL_{i-1}p_{i-\half}s_ip_{i+1}p_{i+\half}\\
&=p_{i-\half}s_ip_{i-1}p_{i-\half}s_i\sigma_i p_ip_{i-\half}s_ip_{i+1}p_{i+\half}&&\text{(by induction)}\\
&=p_{i-\half}s_ip_{i-1}p_{i-\half}s_i\sigma_i p_ip_{i-\half}p_is_ip_{i+\half}\\
&=p_{i-\half}s_ip_{i-1}p_{i-\half}s_i\sigma_i p_ip_{i+\half}\\
&=p_{i-\half}s_ip_{i-1}s_i\sigma_ip_{i+\half} p_ip_{i+\half}&&\text{(by item~(\ref{prel:a1}))}\\
&=p_{i-\half}s_ip_{i-1}s_i\sigma_ip_{i+\half}\\
&=p_{i-\half}p_{i-1}\sigma_ip_{i+\half}\\
&=p_{i-\half}L_{i-1}s_{i-1}p_{i+\half}&&\text{(by item~(\ref{prel:a3})).}
\end{align*}
Putting the above together,  
\begin{multline*}
\sigma_{i+1}p_{i+1}p_{i+\half}=s_{i-1}L_{i-1}s_{i-1}p_{i+\half}+\sigma_ip_{i+\half}
+p_{i-\half}L_{i-1}p_ip_{i-\half}p_{i+\half}\\
-s_{i-1}L_{i-1}p_{i-\half}p_{i+\half}-p_{i-\half}L_{i-1}s_{i-1}p_{i+\half}=L_ip_{i+\half},
\end{multline*}
which proves~\eqref{prel:a2}.\newline
\eqref{jm:comm:a0} Parts~\eqref{prel:a0} and~\eqref{prel:a3} give
\begin{align*}
p_{i+\half}L_ip_{i+\half}
=p_{i+\half}s_iL_ip_{i+\half}
=p_{i+\half}\sigma_{i+1}p_{i}p_{i+\half}
=p_{i+\half}p_{i}p_{i+\half}=p_{i+\half},
\end{align*}
as required.  
\end{proof}
\begin{proposition}\label{invariant}
If $i=1,2,\ldots,$ then 
\begin{enumerate}[label=(\arabic{*}), ref=\arabic{*},leftmargin=0pt,itemindent=1.5em]
\item $(\sigma_{i+\half})^*=\sigma_{i+\half}$,\label{invol-1}
\item $(L_{i+\half})^*=L_{i+\half}$.\label{invol-2}
\end{enumerate}
\end{proposition}
\begin{proof}
\eqref{invol-1} We show that the summand $p_{i-\half}L_{i-1}s_{i-1}p_{i+\half}p_ip_{i-\half}$ in~\eqref{sigma-i-2} is fixed under the $*$ anti-involution on $\mathcal{A}_{i+\half}(z)$, using Proposition~\ref{prel:a}, as follows: 
\begin{align*}
&p_{i-\half}L_{i-1}s_{i-1}p_{i+\half}p_ip_{i-\half}
=p_{i-\half}p_i\sigma_is_{i-1}p_{i+\half}p_ip_{i-\half}
=p_{i-\half}p_i\sigma_is_ip_{i-\half}s_is_{i-1}p_ip_{i-\half}\\
&\quad=p_{i-\half}p_ip_{i+\half}\sigma_is_{i-1}p_{i}p_{i-\half}
=p_{i-\half}p_ip_{i+\half}\sigma_ip_{i-1}p_{i-\half}
=p_{i-\half}p_ip_{i+\half}s_{i-1}L_{i-1}p_{i-\half}.
\end{align*}
\eqref{invol-2} Given that $(\sigma_{i+\half})^*=\sigma_{i+\half}$, it suffices to show that the summand $p_{i+\half}L_ip_ip_{i+\half}$ in~\eqref{jm-i-2} is fixed under the $*$ anti-involution on $\mathcal{A}_{i+\half}(z)$. Using Proposition~\ref{prel:a}, 
\begin{align*}
p_{i+\half}L_ip_ip_{i+\half}=p_{i+\half}p_{i+1}\sigma_{i+1}p_ip_{i+\half}=p_{i+\half}p_{i+1}s_iL_ip_{i+\half}=p_{i+\half}s_ip_{i}L_ip_{i+\half}=p_{i+\half}p_{i}L_ip_{i+\half},
\end{align*}
gives the required result. 
\end{proof}
\begin{proposition}\label{z-0}
If $i=1,2,\ldots,$ then $\sigma_{i+\half}s_i=s_i\sigma_{i+\half}=\sigma_{i+1}$.
\end{proposition}
\begin{proof}
After checking the case $i=1$, the statement follows from induction and the equality 
\begin{align*}
p_{i-\half}p_ip_{i+\half}s_{i-1}L_{i-1}p_{i-\half}=p_{i-\half}L_{i-1}s_{i-1}p_{i+\half}p_ip_{i-\half},
\end{align*}
which was established in the proof of Proposition~\ref{invariant}.
\end{proof}
The following observation is made for later reference.
\begin{lemma}\label{f-1-0}
If $i=1,2,\ldots,$ then 
\begin{align*}
s_{i}s_{i+1}\sigma_{i+1}s_{i+1}s_{i}p_{i+\half}
=s_{i+1}p_{i+\half}L_is_ip_{i+\thalf}p_{i+1}p_{i+\half}=\sigma_{i+\half}s_{i+1}p_{i+\half}.
\end{align*}
\end{lemma}
\begin{proof}
On the one hand,
\begin{align*}
s_is_{i+1}\sigma_{i+1}s_{i+1}s_ip_{i+\half}&=s_is_{i+1}\sigma_{i+1}s_{i+1}p_{i+\half}=s_is_{i+1}p_{i+\thalf}\sigma_{i+1}s_{i+1}\\
&=s_ip_{i+\thalf}\sigma_{i+1}s_{i+1}=s_i\sigma_{i+1}s_{i+1}p_{i+\half}=\sigma_{i+\half}s_{i+1}p_{i+\half},
\end{align*}
and on the other,
\begin{align*}
&s_{i+1}p_{i+\half}L_is_ip_{i+\thalf}p_{i+1}p_{i+\half}=
s_{i+1}p_{i+\half}p_{i+1}\sigma_{i+1}s_ip_{i+\thalf}p_{i+1}p_{i+\half}\\
&=s_{i+1}p_{i+\half}p_{i+1}\sigma_{i+\half}p_{i+\thalf}p_{i+1}p_{i+\half}
=s_{i+1}p_{i+\half}p_{i+1}p_{i+\thalf}\sigma_{i+\half}p_{i+1}p_{i+\half}\\
&=s_{i}p_{i+\thalf}s_is_{i+1}p_{i+1}p_{i+\thalf}\sigma_{i+\half}p_{i+1}p_{i+\half}
=s_{i}p_{i+\thalf}s_ip_{i+2}p_{i+\thalf}\sigma_{i+\half}p_{i+1}p_{i+\half}\\
&=s_{i}p_{i+\thalf}s_i\sigma_{i+\half}p_{i+2}p_{i+\thalf}p_{i+1}p_{i+\half}
=s_{i}p_{i+\thalf}\sigma_{i+1}p_{i+2}p_{i+\thalf}p_{i+1}p_{i+\half}\\
&=s_{i}p_{i+\thalf}\sigma_{i+1}s_{i+1}p_{i+1}p_{i+\half}
=s_{i}\sigma_{i+1}s_{i+1}p_{i+\half}p_{i+1}p_{i+\half}\\
&=s_{i}\sigma_{i+1}s_{i+1}p_{i+\half}=\sigma_{i+\half}s_{i+1}p_{i+\half},
\end{align*}
as required. 
\end{proof}
We are now in a position to prove the first commutativity result of this paper.
\begin{theorem}\label{a-b}
The elements $\sigma_{i+1}$ and $\sigma_{i+\half}$ satisfy the following commutativity relations:
\begin{enumerate}[label=(\arabic{*}), ref=\arabic{*},leftmargin=0pt,itemindent=1.5em]
\item $\sigma_{i+1}p_{i-\half}=p_{i-\half}\sigma_{i+1}=p_{i-\half}L_{i-1}s_ip_{i-\half}$ for $i=2,3,\ldots.$\label{a-b-1}
\item $\sigma_{i+1}p_{i-1}=p_{i-1}\sigma_{i+1}=s_{i-1}\sigma_ip_{i+1}s_is_{i-1}$ for $i=2,3,\ldots.$\label{a-b-2}
\item $\sigma_{i+1}p_{i-\thalf}=p_{i-\thalf}\sigma_{i+1}$ for $i=3,4,\ldots.$\label{a-b-3}
\item $\sigma_{i+1}s_{i-2}=s_{i-2}\sigma_{i+1}$ for $i=3,4,\ldots.$\label{a-b-4}
\item $\sigma_{i+1}p_{i-2}=p_{i-2}\sigma_{i+1}$ for $i=3,4,\ldots.$\label{a-b-5}
\item $\sigma_{i+\half}p_{i-1}=p_{i-1}\sigma_{i+\half}$ for $i=2,3,\ldots.$\label{a-b-6}
\item $\sigma_{i+\half}p_{i-\thalf}=p_{i-\thalf}\sigma_{i+\half}$ for $i=3,4,\ldots.$\label{a-b-7}
\item $\sigma_{i+\half}s_{i-2}=s_{i-2}\sigma_{i+\half}$ for $i=3,4,\ldots.$\label{a-b-8}
\item $\sigma_{i+\half}p_{i-2}=p_{i-2}\sigma_{i+\half}$ for $i=3,4,\ldots.$\label{a-b-9}
\end{enumerate}
\end{theorem}
\begin{proof}
\eqref{a-b-1} Using Lemma~\ref{f-1-0}, 
\begin{align*}
\sigma_{i+1}p_{i-\half}
&=-s_ip_{i-\half}L_{i-1}s_{i-1}p_{i+\half}p_ip_{i-\half}
-p_{i-\half}p_ip_{i+\half}s_{i-1}L_{i-1}p_{i-\half}s_ip_{i-\half}\\
&\quad+p_{i-\half}L_{i-1}s_ip_{i-\half}+s_{i-1}s_i\sigma_is_is_{i-1}p_{i-\half}
+s_ip_{i-\half}L_{i-1}s_ip_{i-\half}s_ip_{i-\half}\\
&=-p_{i-\half}p_ip_{i+\half}s_{i-1}L_{i-1}p_{i-\half}s_ip_{i-\half}+
p_{i-\half}L_{i-1}s_ip_{i-\half}
+s_ip_{i-\half}L_{i-1}s_ip_{i-\half}s_ip_{i-\half}\\
&=-p_{i-\half}p_ip_{i+\half}s_{i-1}L_{i-1}p_{i-\half}p_{i+\half}+
p_{i-\half}L_{i-1}s_ip_{i-\half}
+s_ip_{i-\half}L_{i-1}s_ip_{i-\half}p_{i+\half}\\
&=-p_{i-\half}p_ip_{i+\half}s_{i-1}p_{i+\half}L_{i-1}p_{i-\half}+
p_{i-\half}L_{i-1}s_ip_{i-\half}
+s_ip_{i-\half}L_{i-1}s_ip_{i-\half}p_{i+\half}\\
&=-p_{i-\half}p_ip_{i-\half}p_{i+\half}L_{i-1}p_{i-\half}+
p_{i-\half}L_{i-1}s_ip_{i-\half}
+s_ip_{i-\half}L_{i-1}p_{i-\half}p_{i+\half}\\
&=-p_{i-\half}p_{i+\half}L_{i-1}p_{i-\half}+
p_{i-\half}L_{i-1}s_ip_{i-\half}
+p_{i-\half}L_{i-1}p_{i-\half}p_{i+\half}\\
&=p_{i-\half}L_{i-1}s_ip_{i-\half}=p_{i-\half}s_iL_{i-1}p_{i-\half}=p_{i-\half}\sigma_{i+1},
\end{align*}
as required. \newline
\eqref{a-b-2} We first show that 
\begin{enumerate}[label=(\roman{*}), ref=\roman{*},leftmargin=0pt,itemindent=1.5em]
\item $s_ip_{i-\half}L_{i-1}s_{i-1}p_{i+\half}p_ip_{i-\half}p_{i-1}
=s_ip_{i-\half}L_{i-1}s_ip_{i-\half}s_ip_{i-1}$,\label{h-1-i}
\item $p_{i-\half}p_ip_{i+\half}s_{i-1}L_{i-1}p_{i-\half}s_ip_{i-1}
=p_{i-\half}L_{i-1}s_ip_{i-\half}p_{i-1}$,\label{h-1-ii}
\item $s_{i-1}s_i\sigma_is_is_{i-1}p_{i-1}
=s_{i-1}s_i\sigma_ip_{i+1}s_is_{i-1}
=p_{i-1}s_{i-1}s_i\sigma_is_is_{i-1}$. \label{h-1-iii}
\end{enumerate}
The left hand side of~\eqref{h-1-i} gives
\begin{align*}
s_ip_{i-\half}L_{i-1}s_{i-1}p_{i+\half}p_ip_{i-\half}p_{i-1}
&=s_ip_{i-\half}L_{i-1}s_{i-1}p_{i+\half}s_{i-1}p_{i-1}=s_ip_{i-\half}L_{i-1}s_{i}p_{i-\half}s_{i}p_{i-1},
\end{align*}
which is the right hand side of~\eqref{h-1-i}. Using the relation $p_{i+\half}s_{i-1}s_i=s_{i-1}s_ip_{i-\half}$, the left hand side of~\eqref{h-1-ii} gives 
\begin{align*}
p_{i-\half}p_ip_{i+\half}s_{i-1}L_{i-1}p_{i-\half}s_ip_{i-1}
&=p_{i-\half}p_ip_{i+\half}\sigma_{i-\half}p_{i}p_{i-\half}p_{i-1}s_i\\
&=p_{i-\half}p_ip_{i+\half}\sigma_{i-\half}p_{i}s_{i-1}s_i\\
&=p_{i-\half}p_i\sigma_{i-\half}p_{i+\half}s_{i-1}s_ip_{i-1}\\
&=p_{i-\half}p_i\sigma_{i-\half}s_{i-1}s_ip_{i-\half}p_{i-1}\\
&=p_{i-\half}p_i\sigma_{i}s_ip_{i-\half}p_{i-1}\\
&=p_{i-\half}L_{i-1}s_ip_{i-\half}p_{i-1},
\end{align*}
as required. The item~\eqref{h-1-iii} follows from the relation $s_{i-1}s_ip_{i+1}=p_{i-1}s_{i-1}s_i$.  Next, using the definition~\eqref{sigma-i-1}, 
\begin{align*}
\sigma_{i+1}p_{i-1}&=s_{i-1}s_i\sigma_is_is_{i-1}p_{i-1}
+s_ip_{i-\half}L_{i-1}s_ip_{i-\half}s_ip_{i-1}
+p_{i-\half}L_{i-1}s_ip_{i-\half}p_{i-1}\\
&\quad-s_ip_{i-\half}L_{i-1}s_{i-1}p_{i+\half}p_ip_{i-\half}p_{i-1}
-p_{i-\half}p_ip_{i+\half}s_{i-1}L_{i-1}p_{i-\half}s_ip_{i-1}\\
&=s_{i-1}s_i\sigma_is_is_{i-1}p_{i-1}.
\end{align*}
Since the right hand side of the last expression is manifestly fixed under the $*$ anti-involution on $\mathcal{A}_{i+1}(z)$, the proof of~\eqref{a-b-2} is complete.
\newline
\eqref{a-b-3} We first show that
\begin{enumerate}[label=(\roman{*}), ref=\roman{*},leftmargin=0pt,itemindent=1.5em]
\item $s_ip_{i-\half}p_{i-\thalf}L_{i-2}s_{i-2}s_{i-1}p_{i+\half}p_ip_{i-\half}p_{i-\thalf}
=s_ip_{i-\half}s_{i-2}L_{i-2}s_{i-2}s_{i-1}p_{i+\half}p_ip_{i-\half}p_{i-\thalf}$,\label{k-1-i}
\item $s_ip_{i-\half}s_{i-2}L_{i-2}p_{i-\thalf}s_{i-1}p_{i+\half}p_ip_{i-\half}p_{i-\thalf}=
s_{i}p_{i-\half}p_{i-\thalf}L_{i-2}p_{i-1}p_{i-\thalf}s_{i-1}p_{i+\half}p_ip_{i-\half}p_{i-\thalf}$,\label{k-1-ii}
\item $s_ip_{i-\half}\sigma_{i-1}s_{i-1}p_{i+\half}p_ip_{i-\half}p_{i-\thalf}
=s_{i-1}s_ip_{i-\thalf}L_{i-2}s_{i-1}p_{i-\thalf}s_is_{i-1}p_{i-\thalf}$,\label{k-1-iii}
\item $p_{i-\half}p_ip_{i+\half}s_{i-1}p_{i-\thalf}L_{i-2}s_{i-2}p_{i-\half}s_ip_{i-\thalf}
=p_{i-\half}p_ip_{i+\half}s_{i-1}\sigma_{i-1}p_{i-\half}s_ip_{i-\thalf}$,\label{k-1-iv}
\item $p_{i-\half}p_ip_{i+\half}s_{i-1}s_{i-2}L_{i-2}p_{i-\thalf}p_{i-\half}s_ip_{i-\thalf}
=p_{i-\half}p_{i}p_{i+\half}s_{i-1}s_{i-2}L_{i-2}s_{i-2}p_{i-\half}s_{i}p_{i-\thalf}$,\label{k-1-v}
\item $p_{i-\half}p_ip_{i+\half}s_{i-1}p_{i-\thalf}L_{i-2}p_{i-1}p_{i-\thalf}p_{i-\half}s_ip_{i-\thalf}
=s_{i-1}s_is_{i-1}p_{i-\thalf}L_{i-2}s_{i-1}p_{i-\thalf}s_{i-1}s_is_{i-1}p_{i-\thalf}$,\label{k-1-vi}
\item $p_{i-\half}s_{i-2}L_{i-2}p_{i-\thalf}s_ip_{i-\half}p_{i-\thalf}
=p_{i-\half}s_{i-2}L_{i-2}s_{i-2}s_ip_{i-\half}p_{i-\thalf}$,\label{k-1-vii}
\item $p_{i-\half}p_{i-\thalf}L_{i-2}s_{i-2}s_ip_{i-\half}p_{i-\thalf}
=p_{i-\half}\sigma_{i-1}s_ip_{i-\half}p_{i-\thalf}$, \label{k-1-viii}
\item $s_ip_{i-\half}s_{i-2}L_{i-2}s_{i-2}s_ip_{i-\half}s_ip_{i-\thalf}
=s_ip_{i-\half}s_{i-2}L_{i-2}p_{i-\thalf}s_ip_{i-\half}s_ip_{i-\thalf}$,\label{k-1-ix}
\item $s_ip_{i-\half}\sigma_{i-1}s_ip_{i-\half}s_ip_{i-\thalf}
=s_ip_{i-\half}p_{i-\thalf}L_{i-2}s_{i-2}s_ip_{i-\half}s_ip_{i-\thalf}$.\label{k-1-x}
\item \label{k-1-xi} $s_{i-1}s_is_{i-1}p_{i-\thalf}L_{i-2}s_{i-2}p_{i-\half}p_{i-1}p_{i-\thalf} s_is_{i-1}p_{i-\thalf}= s_{i-1}\sigma_{i-\thalf}s_{i-1}s_ip_{i-\thalf}p_{i-\half}$.
\item \label{k-1-xii} $s_{i-1}s_ip_{i-\thalf}p_{i-1}p_{i-\half}s_{i-2}L_{i-2}p_{i-\thalf} s_{i-1}s_{i}s_{i-1}p_{i-\thalf}=p_{i-\thalf}p_{i-\half}s_is_{i-1}\sigma_{i-\thalf}s_{i-1}$.
\end{enumerate}
For the left hand side of~\eqref{k-1-i},
\begin{align*}
&s_ip_{i-\half}p_{i-\thalf}L_{i-2}s_{i-2}s_{i-1}p_{i+\half}p_ip_{i-\half}p_{i-\thalf}
=s_ip_{i-\half}p_{i-\thalf}L_{i-2}s_{i-2}s_{i-1}p_{i-\thalf}p_{i+\half}p_ip_{i-\half}\\
&\quad=s_ip_{i-\half}p_{i-\thalf}L_{i-2}p_{i-\half}s_{i-2}s_{i-1}p_{i+\half}p_ip_{i-\half}
=s_ip_{i-\half}p_{i-\thalf}L_{i-2}s_{i-2}s_{i-1}p_{i+\half}p_ip_{i-\half},
\end{align*}
and for the right hand side of~\eqref{k-1-i},
\begin{align*}
&s_ip_{i-\half}s_{i-2}L_{i-2}s_{i-2}s_{i-1}p_{i+\half}p_ip_{i-\half}p_{i-\thalf}=
s_ip_{i-\half}s_{i-2}L_{i-2}s_{i-2}s_{i-1}p_{i-\thalf}p_{i+\half}p_ip_{i-\half}\\
&\quad=s_ip_{i-\half}s_{i-2}L_{i-2}p_{i-\half}s_{i-2}s_{i-1}p_{i+\half}p_ip_{i-\half}
=s_ip_{i-\half}s_{i-2}p_{i-\half}L_{i-2}s_{i-2}s_{i-1}p_{i+\half}p_ip_{i-\half}\\
&\quad=s_ip_{i-\half}p_{i-\thalf}L_{i-2}s_{i-2}s_{i-1}p_{i+\half}p_ip_{i-\half}.
\end{align*}
For the left hand side of~\eqref{k-1-ii}, 
\begin{align*}
&s_ip_{i-\half}s_{i-2}L_{i-2}p_{i-\thalf}s_{i-1}p_{i+\half}p_ip_{i-\half}p_{i-\thalf}=
s_ip_{i-\half}s_{i-2}L_{i-2}p_{i-\thalf}s_{i-1}p_{i-\thalf}p_{i+\half}p_ip_{i-\half}\\
&\quad=s_ip_{i-\half}s_{i-2}L_{i-2}p_{i-\half}p_{i-\thalf}p_{i+\half}p_ip_{i-\half}
=s_ip_{i-\half}s_{i-2}p_{i-\half}L_{i-2}p_{i-\thalf}p_{i+\half}p_ip_{i-\half}\\
&\quad=s_ip_{i-\half}p_{i-\thalf}L_{i-2}p_{i-\thalf}p_{i+\half}p_ip_{i-\half}
=s_ip_{i-\half}p_{i-\thalf}p_{i+\half}p_ip_{i-\half}=p_{i-\thalf}p_{i-\half}p_{i+\half},
\end{align*}
and for the right hand side of~\eqref{k-1-ii},
\begin{align*}
&s_{i}p_{i-\half}p_{i-\thalf}L_{i-2}p_{i-1}p_{i-\thalf}s_{i-1}p_{i+\half}p_ip_{i-\half}p_{i-\thalf}=
s_{i}p_{i-\half}p_{i-\thalf}L_{i-2}p_{i-1}p_{i-\thalf}s_{i-1}p_{i-\thalf}p_{i+\half}p_ip_{i-\half}\\
&=s_{i}p_{i-\half}p_{i-\thalf}L_{i-2}p_{i-1}p_{i-\thalf}p_{i-\half}p_{i+\half}p_ip_{i-\half}=
s_{i}p_{i-\half}p_{i-\thalf}L_{i-2}p_{i-1}p_{i-\thalf}p_{i+\half}p_{i-\half}p_ip_{i-\half}\\
&=s_{i}p_{i-\half}p_{i-\thalf}L_{i-2}p_{i-1}p_{i-\thalf}p_{i+\half}p_{i-\half}
=s_{i}p_{i-\thalf}L_{i-2}p_{i-\half}p_{i-1}p_{i-\half}p_{i-\thalf}p_{i+\half}\\
&=s_{i}p_{i-\thalf}L_{i-2}p_{i-\half}p_{i-\thalf}p_{i+\half}=p_{i-\thalf}p_{i-\half}p_{i+\half}.
\end{align*}
For the left hand side of~\eqref{k-1-iii}, 
\begin{align*}
&s_ip_{i-\half}\sigma_{i-1}s_{i-1}p_{i+\half}p_ip_{i-\half}p_{i-\thalf}=
s_i\sigma_{i-1}s_{i-1}p_{i-\thalf}p_{i+\half}p_ip_{i-\half}p_{i-\thalf}\\
&\quad=s_i\sigma_{i-1}s_{i-1}p_{i+\half}p_ip_{i-\half}p_{i-\thalf}
=\sigma_{i-1}s_is_{i-1}p_{i+\half}p_ip_{i-\half}p_{i-\thalf}\\
&\quad=\sigma_{i-1}p_{i-\half}s_is_{i-1}p_ip_{i-\half}p_{i-\thalf}
=s_{i-1}p_{i-\thalf}s_{i-1}\sigma_{i-1}s_ip_{i-1}p_{i-\half}p_{i-\thalf}\\
&\quad=s_{i-1}p_{i-\thalf}s_{i-1}\sigma_{i-1}p_{i-1}p_{i-\thalf}s_ip_{i-\half}
=s_{i-1}p_{i-\thalf}L_{i-2}s_{i-1}s_ip_{i-\thalf}p_{i-\half},
\end{align*}
and for the right hand side of~\eqref{k-1-iii}, 
\begin{align*}
&s_{i-1}s_ip_{i-\thalf}L_{i-2}s_{i-1}p_{i-\thalf}s_is_{i-1}p_{i-\thalf}
=s_{i-1}p_{i-\thalf}L_{i-2}s_is_{i-1}s_ip_{i-\thalf}s_{i-1}p_{i-\thalf}\\
&\quad=s_{i-1}p_{i-\thalf}L_{i-2}s_{i-1}s_ip_{i-\thalf}s_{i-1}p_{i-\thalf}
=s_{i-1}p_{i-\thalf}L_{i-2}s_{i-1}s_ip_{i-\thalf}p_{i-\half}.
\end{align*}
For the left hand side of~\eqref{k-1-iv}, 
\begin{align*}
&p_{i-\half}p_ip_{i+\half}s_{i-1}p_{i-\thalf}L_{i-2}s_{i-2}p_{i-\half}s_ip_{i-\thalf}
=p_{i-\half}p_ip_{i+\half}s_{i-1}p_{i-\thalf}L_{i-2}s_{i-2}p_{i-\thalf}p_{i-\half}s_i\\
&\quad=p_{i-\half}p_ip_{i+\half}s_{i-1}p_{i-\thalf}p_{i-\half}s_i=p_{i-\thalf}p_{i-\half}p_ip_{i+\half},
\end{align*}
and, for the right hand side of~\eqref{k-1-iv}, 
\begin{align*}
&p_{i-\half}p_ip_{i+\half}s_{i-1}\sigma_{i-1}p_{i-\half}s_ip_{i-\thalf}=
p_{i-\half}p_ip_{i+\half}s_{i-1}\sigma_{i-1}p_{i-\thalf}p_{i-\half}s_i\\
&=p_{i-\half}p_ip_{i+\half}s_{i-1}p_{i-\thalf}p_{i-\half}s_i=p_{i-\thalf}p_{i-\half}p_ip_{i+\half}.
\end{align*}
For the left hand side of~\eqref{k-1-v}, 
\begin{align*}
&p_{i-\half}p_ip_{i+\half}s_{i-1}s_{i-2}L_{i-2}p_{i-\thalf}p_{i-\half}s_ip_{i-\thalf}=
p_{i-\half}p_ip_{i+\half}s_{i-1}s_{i-2}L_{i-2}p_{i-\thalf}p_{i-\thalf}p_{i-\half}s_i, 
\end{align*}
and, for the right hand side of~\eqref{k-1-v}, 
\begin{align*}
p_{i-\half}p_{i}p_{i+\half}s_{i-1}s_{i-2}L_{i-2}s_{i-2}p_{i-\half}s_{i}p_{i-\thalf}&=
p_{i-\half}p_{i}p_{i+\half}s_{i-1}s_{i-2}L_{i-2}s_{i-2}p_{i-\thalf}p_{i-\half}s_{i}\\
&=p_{i-\half}p_{i}p_{i+\half}s_{i-1}s_{i-2}L_{i-2}p_{i-\thalf}p_{i-\half}s_{i}. 
\end{align*}
For the left hand side of~\eqref{k-1-vi}, 
\begin{align*}
p_{i-\half}p_ip_{i+\half}s_{i-1}p_{i-\thalf}L_{i-2}p_{i-1}p_{i-\thalf}p_{i-\half}s_ip_{i-\thalf}
&=p_{i-\half}p_ip_{i+\half}s_{i-1}p_{i-\thalf}L_{i-2}p_{i-1}p_{i-\thalf}p_{i-\half}s_i\\
&=p_{i-\half}p_is_{i-1}s_ip_{i-\half}s_ip_{i-\thalf}L_{i-2}p_{i-1}p_{i-\thalf}p_{i-\half}s_i\\
&=p_{i-\half}s_ip_{i-1}p_{i-\half}p_{i-\thalf}p_{i-1}s_iL_{i-2}p_{i-\thalf}p_{i-\half}s_i\\
&=p_{i-\half}s_is_{i-1}p_ip_{i-\thalf}s_{i-1}s_iL_{i-2}p_{i-\thalf}p_{i-\half}s_i\\
&=s_{i}s_{i-1}p_{i+\half}p_ip_{i+\half}p_{i-\thalf}s_{i-1}s_iL_{i-2}p_{i-\thalf}s_i\\
&=s_is_{i-1}p_{i+\half}p_{i-\thalf}s_{i-1}L_{i-2}p_{i-\thalf},
\end{align*}
and, for the right hand side of~\eqref{k-1-vi}, 
\begin{align*}
&s_{i-1}s_is_{i-1}p_{i-\thalf}L_{i-2}s_{i-1}p_{i-\thalf}s_{i-1}s_is_{i-1}p_{i-\thalf}
=s_{i-1}s_is_{i-1}p_{i-\thalf}L_{i-2}s_{i-1}p_{i-\thalf}s_{i}s_{i-1}s_{i}p_{i-\thalf}\\
&\quad=s_{i-1}s_is_{i-1}p_{i-\thalf}L_{i-2}s_{i-1}s_{i}p_{i-\thalf}s_{i-1}p_{i-\thalf}s_{i}
=s_{i-1}s_is_{i-1}p_{i-\thalf}L_{i-2}s_{i-1}s_{i}p_{i-\half}p_{i-\thalf}s_{i}\\
&\quad=s_{i-1}s_is_{i-1}p_{i+\half}p_{i-\thalf}L_{i-2}s_{i-1}s_{i}p_{i-\thalf}s_{i}
=s_is_{i-1}p_{i+\half}p_{i-\thalf}L_{i-2}s_{i-1}p_{i-\thalf}. 
\end{align*}
For the left hand side of~\eqref{k-1-vii}
\begin{align*}
p_{i-\half}s_{i-2}L_{i-2}p_{i-\thalf}s_ip_{i-\half}p_{i-\thalf}
=p_{i-\half}s_{i-2}L_{i-2}p_{i-\thalf}s_ip_{i-\half}
\end{align*}
and for the right hand side of~\eqref{k-1-vii}
\begin{align*}
p_{i-\half}s_{i-2}L_{i-2}s_{i-2}s_ip_{i-\half}p_{i-\thalf}
=p_{i-\half}s_{i-2}L_{i-2}s_{i-2}p_{i-\thalf}s_ip_{i-\half}
=p_{i-\half}s_{i-2}L_{i-2}p_{i-\thalf}s_ip_{i-\half}.
\end{align*}
For the left hand side of~\eqref{k-1-viii},
\begin{align*}
p_{i-\half}p_{i-\thalf}L_{i-2}s_{i-2}s_ip_{i-\half}p_{i-\thalf}
&=p_{i-\half}p_{i-\thalf}L_{i-2}s_{i-2}p_{i-\thalf}s_ip_{i-\half}\\
&=p_{i-\half}p_{i-\thalf}s_ip_{i-\half}=p_{i-\thalf}p_{i-\half}p_{i+\half},
\end{align*}
and for the right hand side of~\eqref{k-1-viii},
\begin{align*}
p_{i-\half}\sigma_{i-1}s_ip_{i-\half}p_{i-\thalf}=
p_{i-\half}\sigma_{i-1}p_{i-\thalf}s_ip_{i-\half}=
p_{i-\half}p_{i-\thalf}s_ip_{i-\half}=
p_{i-\thalf}p_{i-\half}p_{i+\half}. 
\end{align*}
For the left hand side of~\eqref{k-1-ix},
\begin{align*}
s_ip_{i-\half}s_{i-2}L_{i-2}s_{i-2}s_ip_{i-\half}s_ip_{i-\thalf}
&=s_ip_{i-\half}s_{i-2}L_{i-2}s_{i-2}p_{i-\thalf}s_ip_{i-\half}s_i\\
&=s_ip_{i-\half}s_{i-2}L_{i-2}p_{i-\thalf}s_ip_{i-\half}s_i
\end{align*}
and for the right hand side of~\eqref{k-1-ix},
\begin{align*}
&s_ip_{i-\half}s_{i-2}L_{i-2}p_{i-\thalf}s_ip_{i-\half}s_ip_{i-\thalf}
=s_ip_{i-\half}s_{i-2}L_{i-2}p_{i-\thalf}s_ip_{i-\half}s_i.
\end{align*}
For the left hand side of~\eqref{k-1-x},
\begin{align*}
s_ip_{i-\half}\sigma_{i-1}s_ip_{i-\half}s_ip_{i-\thalf}
=s_ip_{i-\half}\sigma_{i-1}p_{i-\thalf}s_ip_{i-\half}s_i
=p_{i-\thalf}p_{i-\half}p_{i+\half},
\end{align*}
and for the right hand side of~\eqref{k-1-x},
\begin{align*}
s_ip_{i-\half}p_{i-\thalf}L_{i-2}s_{i-2}s_ip_{i-\half}s_ip_{i-\thalf}&=
s_ip_{i-\half}p_{i-\thalf}L_{i-2}s_{i-2}p_{i-\thalf}s_ip_{i-\half}s_i\\
&= s_ip_{i-\half}p_{i-\thalf}s_ip_{i-\half}s_i=p_{i-\thalf}p_{i-\half}p_{i+\half}.
\end{align*}
Starting with the left hand side of~\eqref{k-1-xi},
\begin{align*}
&s_{i-1}s_is_{i-1}p_{i-\thalf}L_{i-2}s_{i-2}p_{i-\half}p_{i-1}p_{i-\thalf} s_is_{i-1}p_{i-\thalf}\\
&\qquad=s_{i-1}s_is_{i-1}p_{i-\thalf}L_{i-2}s_{i-1}p_{i-\thalf}s_{i-1}s_{i-2}p_{i-1}p_{i-\thalf} s_is_{i-1}p_{i-\thalf}\\
&\qquad=s_{i-1}s_is_{i-1}p_{i-\thalf}s_{i-1}L_{i-2}p_{i-\thalf}s_{i-1}p_{i-2}p_{i-\thalf} s_is_{i-1}p_{i-\thalf}\\
&\qquad=s_{i-1}s_is_{i-1}p_{i-\thalf}s_{i-1}\sigma_{i-1}p_{i-1}p_{i-\thalf}p_{i-2}s_{i-1}p_{i-\thalf} s_is_{i-1}p_{i-\thalf}\\
&\qquad=s_{i-1}s_is_{i-1}p_{i-\thalf}s_{i-1}\sigma_{i-1}p_{i-1}s_{i-2}s_{i-1}p_{i-\thalf} s_is_{i-1}p_{i-\thalf}\\
&\qquad=s_{i-1}s_is_{i-1}s_{i-1}\sigma_{i-1}p_{i-\half}p_{i-1}p_{i-\half}s_{i-2}s_{i-1} s_is_{i-1}p_{i-\thalf}\\
&\qquad=s_{i-1}s_i\sigma_{i-1}p_{i-\half}s_{i-2}s_{i-1} s_is_{i-1}p_{i-\thalf}\\
&\qquad=s_{i-1}s_{i}\sigma_{i-1}s_{i-2}s_{i-1}s_ip_{i-\thalf}s_{i-1}p_{i-\thalf}\\
&\qquad=s_{i-1}\sigma_{i-\thalf}s_{i-1}s_ip_{i-\thalf}p_{i-\half},
\end{align*}
as required. Considering the left hand side of~\eqref{k-1-xii}, 
\begin{align*}
&s_{i-1}s_ip_{i-\thalf}p_{i-1}p_{i-\half}s_{i-2}L_{i-2}p_{i-\thalf} s_{i-1}s_{i}s_{i-1}p_{i-\thalf}\\
&\qquad=s_{i-1}s_ip_{i-\thalf}p_{i-1}s_{i-2}s_{i-1}p_{i-\thalf}s_{i-1}L_{i-2}p_{i-\thalf} s_{i-1}s_{i}s_{i-1}p_{i-\thalf}\\
&\qquad=s_{i-1}s_ip_{i-\thalf}s_{i-1}p_{i-2}p_{i-\thalf}L_{i-2}s_{i-1}p_{i-\thalf} s_{i-1}s_{i}s_{i-1}p_{i-\thalf}\\
&\qquad=s_{i-1}s_ip_{i-\thalf}s_{i-1}p_{i-2}p_{i-\thalf}p_{i-1}\sigma_{i-1}s_{i-1}p_{i-\thalf} s_{i-1}s_{i}s_{i-1}p_{i-\thalf}\\
&\qquad=s_{i-1}s_ip_{i-\thalf}s_{i-1}s_{i-2}p_{i-1}\sigma_{i-1}s_{i-1}p_{i-\thalf} s_{i-1}s_{i}s_{i-1}p_{i-\thalf}\\
&\qquad=s_{i-1}s_is_{i-1}s_{i-2}p_{i-\half}p_{i-1}p_{i-\half}\sigma_{i-1}s_{i-1} s_{i-1}s_{i}s_{i-1}p_{i-\thalf}\\
&\qquad=s_{i-1}s_is_{i-1}s_{i-2}p_{i-\half}\sigma_{i-1}s_{i-1} s_{i-1}s_{i}s_{i-1}p_{i-\thalf}\\
&\qquad=s_{i-1}s_is_{i-1}\sigma_{i-\thalf}s_{i-1}p_{i-\thalf}s_{i-1}s_{i}s_{i-1}p_{i-\thalf}\\
&\qquad=s_{i-1}s_is_{i-1}\sigma_{i-\thalf}s_{i-1}s_ip_{i-\thalf}p_{i-\half}s_i\\
&\qquad=p_{i-\half}s_is_{i-1}\sigma_{i-\thalf}s_{i-1}p_{i-\thalf}\\
&\qquad=p_{i-\half}s_is_{i-1}s_{i-2}p_{i-\half}\sigma_{i-1}s_{i-1}\\
&\qquad=p_{i-\thalf}p_{i-\half}s_is_{i-1}\sigma_{i-\thalf}s_{i-1},
\end{align*}
as required. Now, after substituting the expressions
\begin{align*}
L_{i-1}&=-p_{i-\thalf}L_{i-2}s_{i-2}-s_{i-2}L_{i-2}p_{i-\thalf}
+p_{i-\thalf}L_{i-2}p_{i-1}p_{i-\thalf}+s_{i-2}L_{i-2}s_{i-2}+\sigma_{i-1}
\end{align*}
and
\begin{multline*}
\sigma_i=s_{i-2}s_{i-1}\sigma_{i-1}s_{i-1}s_{i-2}+ s_{i-1}p_{i-\thalf}L_{i-2}s_{i-1}p_{i-\thalf}s_{i-1}+ p_{i-\thalf}L_{i-2}s_{i-1}p_{i-\thalf} \\-s_{i-1}p_{i-\thalf}L_{i-2}s_{i-2}p_{i-\half}p_{i-1}p_{i-\thalf} -p_{i-\thalf}p_{i-1}p_{i-\half}s_{i-2}L_{i-2}p_{i-\thalf}s_{i-1}
\end{multline*}
into the definition~\eqref{sigma-i-1}, and simplifying the resulting expression using items~\eqref{k-1-i}--\eqref{k-1-x}, we obtain the equality
\begin{multline*}
\sigma_{i+1}p_{i-\thalf}=s_{i-1}s_is_{i-2}s_{i-1}\sigma_{i-1}s_{i-1}s_{i-2}s_is_{i-1}p_{i-\thalf}+
p_{i-\half}p_{i-\thalf}L_{i-2}p_{i-1}p_{i-\thalf}s_ip_{i-\half}\\
+s_ip_{i-\half}p_{i-\thalf}L_{i-2}p_{i-1}p_{i-\thalf}s_ip_{i-\half}s_i-s_{i-1}s_is_{i-1}p_{i-\thalf}L_{i-2}s_{i-2}p_{i-\half}p_{i-1}p_{i-\thalf}s_is_{i-1}p_{i-\thalf}\\
-s_{i-1}s_ip_{i-\thalf}p_{i-1}p_{i-\half}s_{i-2}L_{i-2}p_{i-\thalf}s_{i-1}s_is_{i-1}p_{i-\thalf},
\end{multline*}
in which the two terms with negative coefficients survive from the expansion of $s_{i-1}s_i\sigma_is_{i}s_{i-1}p_{i-\thalf}$. By items~\eqref{k-1-xi} and~\eqref{k-1-xii}, the two terms with negative coefficients in the last expression are interchanged by the $*$ map on $\mathcal{A}_{i+1}(z)$.  Since 
\begin{align*}
s_{i-1}s_is_{i-2}s_{i-1}\sigma_{i-1}s_{i-1}s_{i-2}s_is_{i-1}p_{i-\thalf}=
p_{i-\thalf}s_{i-1}s_is_{i-2}s_{i-1}\sigma_{i-1}s_{i-1}s_{i-2}s_is_{i-1},
\end{align*}
the right hand side of the above expression for $\sigma_{i+1}p_{i-\thalf}$ is fixed under the $*$ anti-involution on $\mathcal{A}_{i+1}(z)$. This completes the proof of~\eqref{a-b-3}.\newline
\eqref{a-b-4} We show that after substituting the expression 
\begin{align*}
L_{i-1}=
-p_{i-\thalf}L_{i-2}s_{i-2}-s_{i-2}L_{i-2}p_{i-\thalf}+p_{i-\thalf}L_{i-2}p_{i-1}p_{i-\thalf}+s_{i-2}L_{i-2}s_{i-2}+\sigma_{i-1}
\end{align*}
into the definition~\eqref{sigma-i-1}, conjugation by $s_{i-2}$ permutes the summands of  $\sigma_{i+1}$ as follows:
\begin{enumerate}[label=(\roman{*}), ref=\roman{*},leftmargin=0pt,itemindent=1.5em]
\item $(s_ip_{i-\half}p_{i-\thalf}L_{i-2}s_{i-2}s_{i-1}p_{i+\half}p_ip_{i-\half})^{s_{i-2}}
=s_ip_{i-\half}p_{i-\thalf}L_{i-2}s_{i-2}s_{i-1}p_{i+\half}p_ip_{i-\half}$,\label{x-1-i}
\item $(s_ip_{i-\half}s_{i-2}L_{i-2}p_{i-\thalf}s_{i-1}p_{i+\half}p_ip_{i-\half})^{s_{i-2}}
=p_{i-\half}p_ip_{i+\half}s_{i-1}p_{i-\thalf}L_{i-2}s_{i-2}p_{i-\half}s_i$,\label{x-1-ii}
\item $(s_ip_{i-\half}p_{i-\thalf}L_{i-2}p_{i-1}p_{i-\thalf}s_{i-1}p_{i+\half}p_ip_{i-\half})^{s_{i-2}}
=s_ip_{i-\half}p_{i-\thalf}L_{i-2}s_{i-2}s_ip_{i-\half}s_i$,\label{x-1-iii}
\item $(s_ip_{i-\half}s_{i-2}L_{i-2}s_{i-2}s_{i-1}p_{i+\half}p_ip_{i-\half})^{s_{i-2}}
=s_{i-1}s_is_{i-1}p_{i-\thalf}L_{i-2}s_{i-2}p_{i-\half}p_{i-1}p_{i-\thalf}s_is_{i-1}$,\label{x-1-iv}
\item $(s_ip_{i-\half}\sigma_{i-1}s_{i-1}p_{i+\half}p_ip_{i-\half})^{s_{i-2}}=
p_{i-\half}s_{i-2}L_{i-2}p_{i-\thalf}s_ip_{i-\half}$, \label{x-1-v}
\item $(p_{i-\half}p_{i-\thalf}L_{i-2}p_{i-1}p_{i-\thalf}s_ip_{i-\half})^{s_{i-2}}
=p_{i-\half}p_{i-\thalf}L_{i-2}p_{i-1}p_{i-\thalf}s_ip_{i-\half}$,\label{x-1-vi}
\item $(s_ip_{i-\half}s_{i-2}L_{i-2}s_{i-2}s_ip_{i-\half}s_i)^{s_{i-2}}
=s_{i-1}s_is_{i-1}p_{i-\thalf}L_{i-2}s_{i-1}p_{i-\thalf}s_{i-1}s_is_{i-1}$,\label{x-1-vii}
\item $(s_ip_{i-\half}\sigma_{i-1}s_ip_{i-\half}s_i)^{s_{i-2}}
=p_{i-\half}\sigma_{i-1}s_ip_{i-\half}$,\label{x-1-viii}
\item $(s_{i-1}s_ip_{i-\thalf}L_{i-2}s_{i-1}p_{i-\thalf}s_is_{i-1})^{s_{i-2}}
=p_{i-\half}s_{i-2}L_{i-2}s_{i-2}s_ip_{i-\half}$,\label{x-1-ix}
\item $(s_{i-1}s_is_{i-2}s_{i-1}\sigma_{i-1}s_{i-1}s_{i-2}s_is_{i-1})^{s_{i-2}}
=s_{i-1}s_is_{i-2}s_{i-1}\sigma_{i-1}s_{i-1}s_{i-2}s_is_{i-1}$.\label{x-1-x}
\end{enumerate}
The item~\eqref{x-1-i} follows from the relation  $s_{i-2}p_{i-\thalf}=p_{i-\thalf}$ and 
\begin{align*}
s_{i}p_{i-\half}p_{i-\thalf}L_{i-2}s_{i-2} s_{i-1}p_{i+\half}p_ip_{i-\half}s_{i-2}&=
s_{i}p_{i-\half}p_{i-\thalf}L_{i-2}s_{i-1} s_{i-2}s_{i-1}p_{i+\half}p_ip_{i-\half}\\
&=s_{i}p_{i-\half}p_{i-\thalf}L_{i-2}s_{i-2} s_{i-1}p_{i+\half}p_ip_{i-\half}. 
\end{align*}
For the left hand side of~\eqref{x-1-ii},
\begin{align*}
(s_ip_{i-\half}s_{i-2}L_{i-2}p_{i-\thalf}s_{i-1}p_{i+\half}p_ip_{i-\half})^{s_{i-2}}&=s_{i-2}s_ip_{i-\half}s_{i-2}L_{i-2}p_{i-\thalf}s_{i-1}p_{i+\half}s_{i-1}p_{i-1}p_{i-\half}s_{i-2}\\
&=s_{i-2}s_ip_{i-\half}s_{i-2}s_{i-1}p_{i+\half}s_{i-1}L_{i-2}p_{i-\thalf}p_{i-1}p_{i-\half}s_{i-2}\\
&=s_{i-2}s_is_{i-2}s_{i-1}p_{i-\thalf}p_{i+\half}s_{i-1}L_{i-2}p_{i-\thalf}p_{i-1}p_{i-\half}s_{i-2}\\
&=s_is_{i-1}p_{i-\thalf}p_{i+\half}s_{i-1}L_{i-2}p_{i-\thalf}p_{i-1}p_{i-\half}s_{i-2}\\
&=p_{i-\half}s_is_{i-1}p_{i-\thalf}s_{i-1}L_{i-2}p_{i-\thalf}p_{i-1}p_{i-\half}s_{i-2}\\
&=p_{i-\half}s_is_{i-1}p_{i-\thalf}s_{i-1}\sigma_{i-1}p_{i-1}p_{i-\thalf}p_{i-1}p_{i-\half}s_{i-2}\\
&=p_{i-\half}s_is_{i-1}p_{i-\thalf}s_{i-1}\sigma_{i-1}p_{i-1}p_{i-\half}s_{i-2}\\
&=p_{i-\half}s_is_{i-1}s_{i-1}\sigma_{i-1}p_{i-\half}p_{i-1}p_{i-\half}s_{i-2}\\
&=p_{i-\half}s_i\sigma_{i-1}p_{i-\half}s_{i-2},
\end{align*}
and, for the right hand side of~\eqref{x-1-ii},
\begin{align*}
&p_{i-\half}p_ip_{i+\half}s_{i-1}p_{i-\thalf}L_{i-2}s_{i-2}p_{i-\half}s_i
=p_{i-\half}p_ip_{i+\half}s_{i-1}p_{i-\thalf}p_{i-2}\sigma_{i-1}p_{i-\half}s_i\\
&\quad=p_{i-\half}p_{i-1}s_{i-1}p_{i+\half}s_{i-1}p_{i-\thalf}p_{i-2}\sigma_{i-1}p_{i-\half}s_i
=p_{i-\half}p_{i-1}p_{i-\thalf}p_{i-2}s_{i-1}p_{i+\half}s_{i-1}\sigma_{i-1}p_{i-\half}s_i\\
&\quad=p_{i-\half}p_{i-1}p_{i-\thalf}p_{i-2}s_{i-1}p_{i+\half}s_{i-1}\sigma_{i-1}p_{i-\half}s_i
=p_{i-\half}p_{i-1}s_{i-2}s_{i-1}p_{i+\half}s_{i-1}\sigma_{i-1}p_{i-\half}s_i\\
&\quad=p_{i-\half}p_{i-1}s_{i-2}s_{i-1}p_{i+\half}s_{i-1}\sigma_{i-1}p_{i-\half}s_i
=p_{i-\half}p_{i-1}s_{i-2}s_{i-1}p_{i+\half}p_{i+\half}s_{i-1}\sigma_{i-1}s_i\\
&\quad=p_{i-\half}p_{i-1}p_{i-\half}s_{i-2}s_{i-1}p_{i+\half}s_{i-1}\sigma_{i-1}s_i
=p_{i-\half}s_{i-2}s_{i-1}p_{i+\half}s_{i-1}\sigma_{i-1}s_i\\
&\quad=p_{i-\half}s_{i-2}s_{i-1}s_{i-1}s_ip_{i-\half}\sigma_{i-1}
=p_{i-\half}s_{i-2}s_ip_{i-\half}\sigma_{i-1}=p_{i-\half}s_i\sigma_{i-1}p_{i-\half}s_{i-2},
\end{align*}
where the last equality follows from $s_{i-2}p_{i-\half}\sigma_{i-1}=\sigma_{i-1}p_{i-\half}s_{i-2}$. \newline
For the left hand side of~\eqref{x-1-iii},
\begin{align*}
(s_ip_{i-\half}p_{i-\thalf}L_{i-2}p_{i-1}p_{i-\thalf}s_{i-1}p_{i+\half}p_ip_{i-\half})^{s_{i-2}}&=s_ip_{i-\half}p_{i-\thalf}L_{i-2}p_{i-1}p_{i-\thalf}s_{i-1}p_{i+\half}p_ip_{i-\half}s_{i-2}\\
&=s_ip_{i-\half}p_{i-\thalf}L_{i-2}p_{i-1}p_{i-\thalf}s_{i}p_{i-\half}s_is_{i-1}p_ip_{i-\half}s_{i-2}\\
&=s_ip_{i-\half}p_{i-\thalf}L_{i-2}p_{i-1}p_{i-\thalf}s_{i}p_{i-\half}p_{i-1}s_ip_{i-\half}s_{i-2}\\
&=s_ip_{i-\half}p_{i-\thalf}L_{i-2}s_{i}p_{i-1}p_{i-\thalf}p_{i-\half}p_{i-1}s_ip_{i-\half}s_{i-2}\\
&=s_ip_{i-\half}p_{i-\thalf}L_{i-2}s_{i}s_{i-1}p_{i}p_{i-\thalf}s_{i-1}s_ip_{i-\half}s_{i-2}\\
&=s_ip_{i-\thalf}L_{i-2}p_{i-\half}s_{i}s_{i-1}p_{i}p_{i-\thalf}s_{i-1}s_ip_{i-\half}s_{i-2}\\
&=s_ip_{i-\thalf}L_{i-2}s_{i}s_{i-1}p_{i+\half}p_{i}p_{i+\half}p_{i-\thalf}s_{i-1}s_is_{i-2}\\
&=p_{i-\thalf}L_{i-2}s_{i-1}p_{i+\half}p_{i-\thalf}s_{i-1}s_is_{i-2}\\
&=p_{i-\thalf}L_{i-2}s_{i-1}p_{i-\thalf}s_{i-1}s_ip_{i-\half}s_{i-2}\\
&=p_{i-\thalf}p_{i-1}\sigma_{i-1}s_{i-1}p_{i-\thalf}s_{i-1}s_ip_{i-\half}s_{i-2}\\
&=p_{i-\thalf}p_{i-1}p_{i-\half}\sigma_{i-1}s_{i-1}s_{i-1}s_ip_{i-\half}s_{i-2}\\
&=p_{i-\thalf}p_{i-1}p_{i-\half}\sigma_{i-1}s_ip_{i-\half}s_{i-2},
\end{align*}
and for the right hand side of~\eqref{x-1-iii},
\begin{align*}
&s_ip_{i-\half}p_{i-\thalf}L_{i-2}s_{i-2}s_ip_{i-\half}s_i=
s_ip_{i-\half}s_ip_{i-\thalf}L_{i-2}s_{i-2}p_{i-\half}s_i\\
&\quad=s_{i-1}p_{i+\half}s_{i-1}p_{i-\thalf}L_{i-2}s_{i-2}p_{i-\half}s_i
=p_{i-\thalf}L_{i-2}s_{i-1}p_{i+\half}s_{i-1}s_{i-2}p_{i-\half}s_i\\
&\quad=p_{i-\thalf}L_{i-2}s_{i-1}p_{i+\half}p_{i-\thalf}s_{i-1}s_{i-2}s_i
=p_{i-\thalf}p_{i-1}\sigma_{i-1}s_{i-1}p_{i+\half}p_{i-\thalf}s_{i-1}s_{i-2}s_i\\
&\quad=p_{i-\thalf}p_{i-1}\sigma_{i-1}s_{i-1}p_{i+\half}p_{i-\thalf}s_{i-1}s_{i-2}s_i
=p_{i-\thalf}p_{i-1}p_{i-\half}\sigma_{i-1}s_{i-1}p_{i+\half}s_{i-1}s_{i-2}s_i\\
&\quad=p_{i-\thalf}p_{i-1}p_{i-\half}\sigma_{i-1}s_{i}p_{i-\half}s_{i}s_{i-2}s_i
=p_{i-\thalf}p_{i-1}p_{i-\half}\sigma_{i-1}s_{i}p_{i-\half}s_{i-2}.
\end{align*}
From the left hand side of~\eqref{x-1-iv}, we obtain
\begin{align*}
(s_ip_{i-\half}s_{i-2}L_{i-2}s_{i-2}s_{i-1}p_{i+\half}p_ip_{i-\half})^{s_{i-2}}
&=(s_ip_{i-\half}s_{i-2}s_{i-1}L_{i-2}s_{i-1}s_{i-2}s_{i-1}p_{i+\half}p_ip_{i-\half})^{s_{i-2}}\\
&=(s_is_{i-2}s_{i-1}p_{i-\thalf}L_{i-2}s_{i-1}s_{i-2}s_{i-1}p_{i+\half}p_ip_{i-\half})^{s_{i-2}}\\
&=s_is_{i-1}p_{i-\thalf}L_{i-2}s_{i-1}s_{i-2}s_{i-1}p_{i+\half}p_ip_{i-\half}s_{i-2}\\
&=s_is_{i-1}p_{i-\thalf}L_{i-2}s_{i-2}s_{i-1}s_{i-2}p_{i+\half}p_ip_{i-\half}s_{i-2}\\
&=s_is_{i-1}p_{i-\thalf}L_{i-2}s_{i-2}s_{i-1}p_{i+\half}p_is_{i-2}p_{i-\half}s_{i-2}\\
&=s_is_{i-1}p_{i-\thalf}L_{i-2}s_{i-2}s_{i-1}p_{i+\half}p_is_{i-1}p_{i-\thalf}s_{i-1}\\
&=s_is_{i-1}p_{i-\thalf}L_{i-2}s_{i-2}s_{i}p_{i-\half}s_is_{i-1}p_is_{i-1}p_{i-\thalf}s_{i-1}\\
&=s_is_{i-1}s_{i}p_{i-\thalf}L_{i-2}s_{i-2}p_{i-\half}s_ip_{i-1}p_{i-\thalf}s_{i-1}\\
&=s_{i-1}s_{i}s_{i-1}p_{i-\thalf}L_{i-2}s_{i-2}p_{i-\half}s_ip_{i-1}p_{i-\thalf}s_{i-1}\\
&=s_{i-1}s_{i}s_{i-1}p_{i-\thalf}L_{i-2}s_{i-2}p_{i-\half}p_{i-1}p_{i-\thalf}s_is_{i-1}
\end{align*}
which is identical to the right hand side of~\eqref{x-1-iv}.\newline 
From the left hand side of~\eqref{x-1-v}, we obtain  
\begin{align*}
&(s_ip_{i-\half}\sigma_{i-1}s_{i-1}p_{i+\half}p_ip_{i-\half})^{s_{i-2}}
=s_{i-2}s_i\sigma_{i-1}s_{i-1}p_{i-\thalf}p_{i+\half}p_ip_{i-\half}s_{i-2}\\
&\quad=s_i\sigma_{i-\thalf}s_{i-1}p_{i+\half}p_ip_{i-\half}p_{i-\thalf}
=p_{i-\half}\sigma_{i-\thalf}s_is_{i-1}p_ip_{i-\half}p_{i-\thalf}\\
&\quad=p_{i-\half}\sigma_{i-\thalf}s_ip_{i-1}p_{i-\half}p_{i-\thalf}
=p_{i-\half}\sigma_{i-1}s_ip_{i-2}p_{i-\half}p_{i-\thalf}\\
&\quad=p_{i-\half}\sigma_{i-1}p_{i-2}p_{i-\thalf}s_ip_{i-\half}
=p_{i-\half}s_{i-2}L_{i-2}p_{i-\thalf}s_ip_{i-\half},
\end{align*} 
which is identical to the right hand side of~\eqref{x-1-v}.\newline
The item~\eqref{x-1-vi} follows immediately from the relation $s_{i-2}p_{i-\thalf}=p_{i-\thalf}s_{i-2}=p_{i-\thalf}$. \newline
From the left hand side of~\eqref{x-1-vii}, we obtain
\begin{align*}
&(s_ip_{i-\half}s_{i-2}L_{i-2}s_{i-2}s_ip_{i-\half}s_i)^{s_{i-2}}
=s_{i-2}s_is_{i-2}s_{i-1}p_{i-\thalf}s_{i-1}L_{i-2}s_{i-2}s_ip_{i-\half}s_is_{i-2}\\
&\quad=s_is_{i-1}p_{i-\thalf}s_{i-1}L_{i-2}s_{i-2}s_{i}p_{i-\half}s_{i}s_{i-2}
=s_is_{i-1}p_{i-\thalf}s_{i-1}L_{i-2}s_{i-2}s_{i-1}p_{i+\half}s_{i-1}s_{i-2}\\
&\quad=s_is_{i-1}p_{i-\thalf}s_{i-1}L_{i-2}s_{i}s_{i-1}p_{i-\thalf}s_{i-1}s_{i}
=s_is_{i-1}p_{i-\thalf}s_{i}L_{i-2}s_{i-1}s_{i}p_{i-\thalf}s_{i-1}s_{i}\\
&\quad=s_is_{i-1}s_{i}p_{i-\thalf}L_{i-2}s_{i-1}p_{i-\thalf}s_{i}s_{i-1}s_{i}
=s_{i-1}s_{i}s_{i-1}p_{i-\thalf}L_{i-2}s_{i-1}p_{i-\thalf}s_{i-1}s_{i}s_{i-1},
\end{align*}
which is identical to the right hand side of~\eqref{x-1-vii}.\newline
From the left hand side of~\eqref{x-1-viii}, 
\begin{align*}
&(s_ip_{i-\half}\sigma_{i-1}s_ip_{i-\half}s_i)^{s_{i-2}}
=s_{i-2}s_ip_{i-\half}\sigma_{i-1}s_{i-1}p_{i+\half}s_{i-1}s_{i-2}\\
&\quad=s_{i-2}s_i\sigma_{i-1}s_{i-1}p_{i-\thalf}p_{i+\half}s_{i-1}s_{i-2}
=\sigma_{i-\thalf}s_is_{i-1}p_{i-\thalf}p_{i+\half}s_{i-1}s_{i-2}\\
&\quad=\sigma_{i-\thalf}s_is_{i-1}p_{i+\half}p_{i-\thalf}s_{i-1}s_{i-2}
=\sigma_{i-\thalf}p_{i-\half}s_is_{i-1}s_{i-1}s_{i-2}p_{i-\half}\\
&\quad=\sigma_{i-\thalf}p_{i-\half}s_is_{i-2}p_{i-\half}
=p_{i-\half}\sigma_{i-\thalf}s_{i-2}s_ip_{i-\half}
=p_{i-\half}\sigma_{i-1}s_ip_{i-\half},
\end{align*}
which is identical to the right hand side of~\eqref{x-1-viii}.\newline
From the left hand side of~\eqref{x-1-ix},
\begin{align*}
&(s_{i-1}s_ip_{i-\thalf}L_{i-2}s_{i-1}p_{i-\thalf}s_is_{i-1})^{s_{i-2}}
=s_{i-2}s_{i-1}p_{i-\thalf}L_{i-2}s_is_{i-1}s_ip_{i-\thalf}s_{i-1}s_{i-2}\\
&\quad=p_{i-\half}s_{i-2}s_{i-1}L_{i-2}s_is_{i-1}s_is_{i-1}s_{i-2}p_{i-\half}
=p_{i-\half}s_{i-2}L_{i-2}s_{i-1}s_is_{i-1}s_is_{i-1}s_{i-2}p_{i-\half}\\
&\quad=p_{i-\half}s_{i-2}L_{i-2}s_is_{i-2}p_{i-\half},
\end{align*}
which is identical to the right hand side of~\eqref{x-1-ix}. \newline
The equality~\eqref{x-1-x} follows from the fact that $s_{i-2}s_{i-1}s_is_{i-2}s_{i-1}=s_{i-1}s_is_{i-2}s_{i-1}s_{i}$. \newline
\eqref{a-b-5} By~\eqref{a-b-2} and~\eqref{a-b-4}, $\sigma_{i+1}$ commutes with $\langle s_{i-2},p_{i-1}\rangle$, and so with $p_{i-2}=s_{i-2}p_{i-1}s_{i-2}$. \newline
\eqref{a-b-6} By Proposition~\ref{z-0}, and~\eqref{a-b-2},
\begin{align*}
\sigma_{i+\half}p_{i-1}=s_{i}\sigma_{i+1}p_{i-1}=p_{i-1}s_{i}\sigma_{i+1}=p_{i-1}\sigma_{i+\half}. 
\end{align*}
\eqref{a-b-7}--\eqref{a-b-9} Can be proved using the same argument as part~\eqref{a-b-6}. 
\end{proof}
\begin{theorem}\label{a-c}
The elements $L_{i+1}$ and $L_{i+\half}$ satisfy the following commutation relations:
\begin{enumerate}[label=(\arabic{*}), ref=\arabic{*},leftmargin=0pt,itemindent=1.5em]
\item  $L_{i+1}p_{i+\half}=p_{i+\half}L_{i+1}=p_{i+\half}L_ip_{i+1}p_{i+\half}$, for $i=1,2,\ldots.$\label{a-c-1}
\item $L_{i+1}p_{i}=p_{i}L_{i+1}$, for $i=1,2,\ldots.$\label{a-c-2}
\item $L_{i+1}p_{i-\half}=p_{i-\half}L_{i+1}$, for $i=2,3,\ldots.$ \label{a-c-3}
\item $L_{i+1}s_{i-1}=s_{i-1}L_{i+1}$, for $i=2,3,\ldots.$\label{a-c-4}
\item $L_{i+1}p_{i-1}=p_{i-1}L_{i+1}$, for $i=2,3\ldots.$\label{a-c-5}
\item $L_{i+\half}p_{i}=p_{i}L_{i+\half}$, for $i=1,2,\ldots.$\label{a-c-6}
\item $L_{i+\half}p_{i-\half}=p_{i-\half}L_{i+\half}$, for $i=2,3,\ldots.$\label{a-c-7}
\item $L_{i+\half}s_{i-1}=s_{i-1}L_{i+\half}$, for $i=2,3,\ldots.$\label{a-c-8}
\item $L_{i+\half}p_{i-1}=p_{i-1}L_{i+\half}$, for $i=2,3\ldots.$\label{a-c-9}
\end{enumerate}
\end{theorem}
\begin{proof}
\eqref{a-c-1} Using Proposition~\ref{prel:a},
\begin{align*}
L_{i+1}p_{i+\half}&=-s_iL_ip_{i+\half}-p_{i+\half}L_is_ip_{i+\half}+p_{i+\half}L_ip_{i+1}p_{i+\half}+s_iL_is_ip_{i+\half}+\sigma_{i+1}p_{i+\half}\\
&=-s_iL_ip_{i+\half}-p_{i+\half}L_ip_{i+\half}+p_{i+\half}L_ip_{i+1}p_{i+\half}+s_iL_ip_{i+\half}+p_{i+\half}\\
&=-s_iL_ip_{i+\half}-p_{i+\half}+p_{i+\half}L_ip_{i+1}p_{i+\half}+s_iL_ip_{i+\half}+p_{i+\half}\\
&=p_{i+\half}L_ip_{i+1}p_{i+\half},
\end{align*}
as required.\newline
\eqref{a-c-2} From Proposition~\ref{prel:a}, we obtain $\sigma_{i+1}p_i=\sigma_{i+1}p_ip_{i+\half}p_i=s_iL_ip_{i+\half}p_i$. Thus
\begin{align*}
L_{i+1}p_{i}&=-s_iL_ip_{i+\half}p_{i}-p_{i+\half}L_is_ip_{i}+p_{i+\half}L_ip_{i+1}p_{i+\half}p_{i}+s_{i}L_is_{i}p_{i}+\sigma_{i+1}p_{i}\\
&=(\sigma_{i+1}p_i-s_iL_ip_{i+\half}p_{i})+(p_{i+\half}L_ip_{i+1}p_{i+\half}p_i-p_{i+\half}L_is_ip_i)+s_iL_{i}p_{i+1}s_i\\ 
&=(\sigma_{i+1}p_i-\sigma_{i+1}p_i)+(p_{i+\half}L_is_ip_i-p_{i+\half}L_is_ip_i)+s_iL_{i}p_{i+1}s_i,
\end{align*}
and the statement now follows from the fact that the right hand side of the above expression is fixed under the $*$ anti-involution on $\mathcal{A}_{i+1}(z)$.\newline
\eqref{a-c-3} We first show that 
\begin{enumerate}[label=(\roman{*}), ref=\roman{*},leftmargin=0pt,itemindent=1.5em]
\item $p_{i+\half}L_is_ip_{i-\half}=p_{i+\half}\sigma_is_i$, \label{u-3-i}
\item $s_iL_ip_{i+\half}p_{i-\half}=s_ip_{i-\half}L_{i-1}p_{i}p_{i-\half}p_{i+\half}$, \label{u-3-ii}
\item $p_{i+\half}L_ip_{i+1}p_{i+\half}p_{i-\half}=
p_{i+\half}p_{i-\half}L_{i-1}p_ip_{i-\half}p_{i+\half}$, \label{u-3-iii}
\item $\sigma_{i+1}p_{i-\half}=p_{i-\half}\sigma_{i+1}$.\label{u-3-iv}
\end{enumerate}
\eqref{u-3-i} The definition~\eqref{jm-i-1} gives 
\begin{align*}
p_{i+\half}L_is_ip_{i-\half}
&=-p_{i+\half}s_{i-1}L_{i-1}p_{i-\half}s_ip_{i-\half}-p_{i+\half}p_{i-\half}L_{i-1}s_{i-1}s_ip_{i-\half}\\
&\quad+p_{i+\half}p_{i-\half}L_{i-1}p_ip_{i-\half}s_ip_{i-\half}+p_{i+\half}s_{i-1}L_{i-1}s_{i-1}s_ip_{i-\half}
+p_{i+\half}\sigma_is_ip_{i-\half}\\
&=-p_{i+\half}s_{i-1}L_{i-1}p_{i-\half}p_{i+\half}
-p_{i+\half}p_{i-\half}L_{i-1}p_{i+\half}s_{i-1}s_i\\
&\quad+p_{i+\half}p_{i-\half}L_{i-1}p_ip_{i-\half}p_{i+\half}+p_{i+\half}s_{i-1}L_{i-1}p_{i+\half}s_{i-1}s_i
+p_{i+\half}\sigma_is_ip_{i-\half}\\
&=-p_{i+\half}p_{i-\half}L_{i-1}p_{i-\half}
-p_{i+\half}p_{i-\half}L_{i-1}s_{i-1}s_i\\
&\quad+p_{i+\half}p_{i-\half}L_{i-1}p_{i-\half}+p_{i+\half}p_{i-\half}L_{i-1}s_{i-1}s_i
+p_{i+\half}\sigma_is_ip_{i-\half}\\
&=p_{i+\half}\sigma_is_i,
\end{align*}
where the last equality follows from Proposition~\ref{prel:a}. \newline
\eqref{u-3-ii} The definition~\eqref{jm-i-1} gives 
\begin{align*}
s_iL_ip_{i+\half}p_{i-\half}&=
-s_is_{i-1}L_{i-1}p_{i-\half}p_{i+\half}
-s_ip_{i-\half}L_{i-1}s_{i-1}p_{i+\half}p_{i-\half}\\
&\quad+s_ip_{i-\half}L_{i-1}p_ip_{i-\half}p_{i+\half}
+s_is_{i-1}L_{i-1}s_{i-1}p_{i+\half}p_{i-\half}
+s_i\sigma_ip_{i+\half}p_{i-\half}\\
&=-s_is_{i-1}L_{i-1}p_{i-\half}p_{i+\half}
-s_ip_{i-\half}L_{i-1}p_{i+\half}p_{i-\half}\\
&\quad+s_ip_{i-\half}L_{i-1}p_ip_{i-\half}p_{i+\half}
+s_is_{i-1}L_{i-1}p_{i+\half}p_{i-\half}
+p_{i+\half}p_{i-\half}\\
&=s_ip_{i-\half}L_{i-1}p_ip_{i-\half}p_{i+\half}
\end{align*}
since $p_{i-\half}L_{i-1}p_{i-\half}=p_{i-\half}$. \newline 
\eqref{u-3-iii} The definition~\eqref{jm-i-1} gives 
\begin{align*}
p_{i+\half}L_ip_{i+1}p_{i+\half}p_{i-\half}
&=-p_{i+\half}s_{i-1}L_{i-1}p_{i-\half}p_{i+1}p_{i+\half}
-p_{i+\half}p_{i-\half}L_{i-1}s_{i-1}p_{i+1}p_{i+\half}p_{i-\half}\\
&\quad+p_{i+\half}p_{i-\half}L_{i-1}p_ip_{i-\half}p_{i+1}p_{i+\half}
+p_{i+\half}s_{i-1}L_{i-1}s_{i-1}p_{i+1}p_{i+\half}p_{i-\half}\\
&\quad+p_{i+\half}\sigma_ip_{i+1}p_{i+\half}p_{i-\half}\\
&=-p_{i+\half}s_{i-1}L_{i-1}p_{i-\half}p_{i+1}p_{i+\half}
-p_{i+\half}p_{i-\half}L_{i-1}p_{i-\half}p_{i+1}p_{i+\half}\\
&\quad+p_{i+\half}p_{i-\half}L_{i-1}p_ip_{i-\half}p_{i+1}p_{i+\half}
+p_{i+\half}s_{i-1}L_{i-1}p_{i-\half}p_{i+1}p_{i+\half}\\
&\quad+p_{i+\half}p_{i-\half}p_{i+1}p_{i+\half}\\
&=-p_{i+\half}p_{i-\half}p_{i+1}p_{i+\half}+p_{i+\half}p_{i-\half}L_{i-1}p_ip_{i-\half}p_{i+1}p_{i+\half}\\
&\quad+p_{i+\half}p_{i-\half}p_{i+1}p_{i+\half}\\
&=p_{i+\half}p_{i-\half}L_{i-1}p_ip_{i-\half}p_{i+1}p_{i+\half}.
\end{align*}
\eqref{u-3-iv} Was demonstrated in Theorem~\ref{a-b}. \newline
Now, using 
\begin{align*}
s_{i}L_{i-1}s_{i-1}p_{i-\half}
&=-s_is_{i-1}L_{i-1}p_{i-\half}
-p_{i-\half}L_{i-1}s_{i-1}s_{i}p_{i-\half}
+s_ip_{i-\half}L_{i-1}p_ip_{i-\half}s_ip_{i-\half}\\
&\quad+s_is_{i-1}L_{i-1}s_{i-1}s_{i}p_{i-\half}
+s_i\sigma_{i}s_ip_{i-\half}\\
&=-s_is_{i-1}L_{i-1}p_{i-\half}
-p_{i-\half}L_{i-1}s_{i-1}s_{i}
+s_ip_{i-\half}L_{i-1}p_ip_{i-\half}p_{i+\half}\\
&\quad+s_is_{i-1}L_{i-1}s_{i-1}s_{i}p_{i-\half}
+p_{i+\half}\sigma_{i}s_i,
\end{align*}
we obtain
\begin{align*}
L_{i+1}p_{i-\half}
&=-s_iL_{i}p_{i+\half}p_{i-\half}
-p_{i+\half}L_is_ip_{i-\half}
+p_{i+\half}L_{i}p_{i+1}p_{i+\half}p_{i-\half}
+s_iL_is_ip_{i-\half}+\sigma_{i+1}p_{i-\half}\\
&=-s_ip_{i-\half}L_{i-1}p_ip_{i-\half}p_{i+\half}
-p_{i+\half}\sigma_is_i
+p_{i+\half}p_{i-\half}L_{i-1}p_ip_{i-\half}p_{i+1}p_{i+\half}\\
&\quad+s_iL_is_ip_{i-\half}
+\sigma_{i+1}p_{i-\half}\\
&=p_{i+\half}p_{i-\half}L_{i-1}p_ip_{i-\half}p_{i+1}p_{i+\half}
+s_{i}s_{i-1}L_{i-1}s_{i-1}s_ip_{i-\half}+\sigma_{i+1}p_{i-\half}.
\end{align*}
Since the right hand side of the last expression is fixed under the $*$ anti-involution on $\mathcal{A}_{i+1}(z)$, the proof of~\eqref{a-c-3} is complete. \newline
\eqref{a-c-4} We show that after substituting the expression
\begin{align*}
L_i=-p_{i-\half}L_{i-1}s_{i-1}-s_{i-1}L_{i-1}p_{i-\half}+p_{i-\half}L_{i-1}p_{i}p_{i-\half}+s_{i-1}L_{i-1}s_{i-1}+\sigma_i
\end{align*}
into the definition~\eqref{jm-i-1}, conjugation by $s_{i-1}$ permutes the summands of  $L_{i+1}$ as follows:
\begin{enumerate}[label=(\roman{*}), ref=\roman{*},leftmargin=0pt,itemindent=1.5em]
\item $s_{i-1}p_{i+\half}p_{i-\half}L_{i-1}s_{i-1}s_is_{i-1}
=p_{i+\half}p_{i-\half}L_{i-1}s_{i-1}s_i$,\label{x-2-i}
\item $s_{i-1}p_{i+\half}s_{i-1}L_{i-1}p_{i-\half}s_is_{i-1}
=s_ip_{i-\half}L_{i-1}s_{i-1}p_{i+\half}$, \label{x-2-ii}
\item $s_{i-1}p_{i+\half}p_{i-\half}L_{i-1}p_ip_{i-\half}s_is_{i-1}
=p_{i+\half}p_{i-\half}L_{i-1}s_{i-1}p_{i+1}p_{i+\half}$,\label{x-2-iii}
\item $s_{i-1}p_{i+\half}s_{i-1}L_{i-1}s_{i-1}s_is_{i-1}
=s_ip_{i-\half}L_{i-1}s_{i-1}s_i$,\label{x-2-iv}
\item $s_{i-1}p_{i+\half}\sigma_is_is_{i-1}
=s_ip_{i-\half}L_{i-1}s_{i-1}p_{i+\half}p_ip_{i-\half}$,\label{x-2-v}
\item $s_{i-1}p_{i+\half}p_{i-\half}L_{i-1}p_ip_{i-\half}p_{i+1}p_{i+\half}s_{i-1}
=p_{i+\half}p_{i-\half}L_{i-1}p_ip_{i-\half}p_{i+1}p_{i+\half}$,\label{x-2-vi}
\item $s_{i-1}p_{i+\half}s_{i-1}L_{i-1}s_{i-1}p_{i+1}p_{i+\half}s_{i-1}
=s_ip_{i-\half}L_{i-1}p_ip_{i-\half}s_i$,\label{x-2-vii}
\item $s_{i-1}p_{i+\half}\sigma_ip_{i+1}p_{i+\half}s_{i-1}
=s_ip_{i-\half}L_{i-1}s_{i}p_{i-\half}s_i$, \label{x-2-viii}
\item $s_{i-1}s_is_{i-1}L_{i-1}s_{i-1}s_is_{i-1}
=s_is_{i-1}L_{i-1}s_{i-1}s_i$,\label{x-2-ix}
\item $s_{i-1}p_{i-\half}L_{i-1}s_ip_{i-\half}s_{i-1}
=p_{i-\half}L_{i-1}s_ip_{i-\half}$.\label{x-2-x}
\end{enumerate}
For item~\eqref{x-2-i},
\begin{align*}
s_{i-1}p_{i+\half}p_{i-\half}L_{i-1}s_{i-1}s_is_{i-1}
=p_{i+\half}p_{i-\half}L_{i-1}s_is_{i-1}s_{i}
=p_{i+\half}p_{i-\half}L_{i-1}s_{i-1}s_{i}.
\end{align*}
For item~\eqref{x-2-ii},
\begin{align*}
&s_{i-1}p_{i+\half}s_{i-1}L_{i-1}p_{i-\half}s_is_{i-1}
=s_{i}p_{i-\half}s_{i}L_{i-1}p_{i-\half}s_is_{i-1}\\
&\quad=s_{i}p_{i-\half}L_{i-1}s_{i}p_{i-\half}s_is_{i-1}
=s_{i}p_{i-\half}L_{i-1}s_{i-1}p_{i+\half}.
\end{align*}
For item~\eqref{x-2-iii},
\begin{align*}
&s_{i-1}p_{i+\half}p_{i-\half}L_{i-1}p_ip_{i-\half}s_is_{i-1}
=p_{i+\half}p_{i-\half}L_{i-1}p_is_is_{i-1}p_{i+\half}
=p_{i+\half}p_{i-\half}L_{i-1}s_{i-1}p_{i+1}p_{i+\half}.
\end{align*}
For item~\eqref{x-2-iv},
\begin{align*}
&s_{i-1}p_{i+\half}s_{i-1}L_{i-1}s_{i-1}s_is_{i-1}
=s_{i}p_{i-\half}s_{i}L_{i-1}s_{i-1}s_is_{i-1}\\
&\quad=s_{i}p_{i-\half}L_{i-1}s_{i}s_{i-1}s_is_{i-1}
=s_{i}p_{i-\half}L_{i-1}s_{i-1}s_i.
\end{align*}
For the left hand side of item~\eqref{x-2-v},
\begin{align*}
s_{i-1}p_{i+\half}\sigma_is_is_{i-1}=s_{i-1}\sigma_is_ip_{i-\half}s_{i-1}=\sigma_{i-\half}s_ip_{i-\half},
\end{align*}
and for the right hand side of item~\eqref{x-2-v},
\begin{align*}
&s_ip_{i-\half}L_{i-1}s_{i-1}p_{i+\half}p_ip_{i-\half}=
s_ip_{i-\half}L_{i-1}s_{i}p_{i-\half}s_is_{i-1}p_ip_{i-\half}
=s_ip_{i-\half}L_{i-1}s_{i}p_{i-\half}s_ip_{i-1}p_{i-\half}\\
&\quad=s_ip_{i-\half}s_{i}L_{i-1}p_{i-\half}p_{i-1}s_ip_{i-\half}
=s_ip_{i-\half}s_{i}\sigma_ip_ip_{i-\half}p_{i-1}s_ip_{i-\half}
=s_ip_{i-\half}s_{i}\sigma_is_{i-1}p_{i-1}s_ip_{i-\half}\\
&\quad=s_is_{i}\sigma_ip_{i+\half}s_{i-1}s_ip_{i-1}p_{i-\half}
=\sigma_is_{i-1}s_ip_{i-\half}p_{i-1}p_{i-\half}
=\sigma_{i-\half}s_ip_{i-\half}.
\end{align*}
The item~\eqref{x-2-vi} follows from the relation $s_{i-1}p_{i-\half}=p_{i-\half}s_{i-1}=p_{i-\half}$. For item~\eqref{x-2-vii},
\begin{align*}
&s_{i-1}p_{i+\half}s_{i-1}L_{i-1}s_{i-1}p_{i+1}p_{i+\half}s_{i-1}
=s_{i}p_{i-\half}s_{i}L_{i-1}p_{i+1}s_{i}p_{i-\half}s_{i}
=s_{i}p_{i-\half}L_{i-1}p_{i}p_{i-\half}s_{i}.
\end{align*}
For item~\eqref{x-2-viii},
\begin{align*}
&s_{i-1}p_{i+\half}\sigma_ip_{i+1}p_{i+\half}s_{i-1}=
s_{i-1}\sigma_is_ip_{i-\half}s_ip_{i+1}p_{i+\half}s_{i-1}
=\sigma_{i-\half}s_ip_{i-\half}p_{i}p_{i+\half}s_{i-1}\\
&\quad=\sigma_{i-\half}s_ip_{i-\half}p_{i}s_{i-1}s_ip_{i-\half}s_i
=\sigma_{i-\half}s_ip_{i-\half}p_{i-1}s_ip_{i-\half}s_i
=\sigma_{i-\half}s_ip_{i-\half}s_ip_{i-1}p_{i-\half}s_i\\
&\quad=\sigma_{i-\half}s_ip_{i-\half}s_ip_{i-1}p_{i-\half}s_i
=\sigma_{i-\half}s_{i-1}p_{i+\half}s_{i-1}p_{i-1}p_{i-\half}s_i
=s_{i-1}p_{i+\half}s_{i-1}\sigma_{i-\half}p_{i-1}p_{i-\half}s_i\\
&\quad=s_{i}p_{i-\half}s_{i}\sigma_{i-\half}p_{i-1}p_{i-\half}s_i
=s_{i}p_{i-\half}s_iL_{i-1}p_{i-\half}s_i.
\end{align*}
\eqref{x-2-ix} Follows from the Coxeter relations, and~\eqref{x-2-x} from the relation $s_{i-1}p_{i-\half}=p_{i-\half}s_{i-1}=p_{i-\half}$.\newline
\eqref{a-c-5} By parts~\eqref{a-c-2} and~\eqref{a-c-4}, $L_{i+1}$ commutes with $\langle p_{i},s_{i-1}\rangle$, and so with $p_{i-1}=s_{i-1}p_is_{i-1}$. \newline
\eqref{a-c-6} From item~\eqref{prel:a3} of Proposition~\ref{prel:a}, $\sigma_{i+\half}p_i=s_i\sigma_{i+1}p_ip_{i+\half}p_i=L_ip_{i+\half}p_i$, and 
\begin{align*}
L_{i+\half}p_i&=
-p_{i+\half}L_ip_i-L_ip_{i+\half}p_i+p_{i+\half}L_ip_ip_{i+\half}p_i+s_iL_{i-\half}s_ip_i+\sigma_{i+\half}p_i\\
&=-p_{i+\half}L_ip_i-L_ip_{i+\half}p_i+p_{i+\half}L_ip_i+s_iL_{i-\half}p_{i+1}s_i+L_ip_{i+\half}p_i\\
&=s_iL_{i-\half}p_{i+1}s_i.
\end{align*}
Since  $s_iL_{i-\half}p_{i+1}s_i=s_ip_{i+1}L_{i-\half}s_i$, this completes the proof of~\eqref{a-c-6}.\newline 
\eqref{a-c-7} We show that 
\begin{multline}\label{z-a-1}
L_{i+\half}p_{i-\half}=-p_{i-\half}L_{i-1}p_ip_{i-\half}p_{i+\half}-p_{i+\half}p_{i-\half}L_{i-1}p_ip_{i-\half}\\
+s_is_{i-1}L_{i-\thalf}p_{i+\half}s_{i-1}s_i+s_is_{i-1}\sigma_{i-\half}p_{i+\half}s_{i-1}s_i. 
\end{multline}
From the definition~\eqref{jm-i-2}, 
\begin{align}\label{z-a-2}
L_{i+\half}p_{i-\half}=(-L_ip_{i+\half}-p_{i+\half}L_i+p_{i+\half}L_ip_ip_{i+\half}+s_iL_{i-\half}s_i+\sigma_{i+\half})p_{i-\half}.
\end{align}
Using part~\eqref{a-c-1} the first three summands in the right hand side of~\eqref{z-a-2} are transformed as:
\begin{align*}
L_ip_{i+\half}p_{i-\half}=p_{i-\half}L_{i-1}p_ip_{i-\half}p_{i+\half}&&\text{and}
&&p_{i+\half}L_ip_{i-\half}=p_{i+\half}p_{i-\half}L_{i-1}p_ip_{i-\half},
\end{align*}
and
\begin{align*}
p_{i+\half}L_ip_ip_{i+\half}p_{i-\half}&
=p_{i+\half}p_iL_ip_{i-\half}p_{i+\half}
=p_{i+\half}p_ip_{i-\half}L_{i-1}p_ip_{i-\half}p_{i+\half}\\
&=p_{i+\half}p_iL_{i-1}p_{i-\half}p_{i+\half}
=L_{i-1}p_{i-\half}p_{i+\half}.
\end{align*}
Next,
\begin{align*}
s_iL_{i-\half}s_ip_{i-\half}&=
-s_iL_{i-1}p_{i-\half}s_ip_{i-\half}
-s_ip_{i-\half}L_{i-1}s_ip_{i-\half}
+s_ip_{i-\half}L_{i-1}p_{i-1}p_{i-\half}s_ip_{i-\half}\\
&\quad+s_is_{i-1}L_{i-\thalf}s_{i-1}s_ip_{i-\half}
+s_i\sigma_{i-\half}s_ip_{i-\half}\\
&=-s_iL_{i-1}p_{i-\half}p_{i+\half}
-s_ip_{i-\half}L_{i-1}s_ip_{i-\half}
+s_ip_{i-\half}L_{i-1}p_{i-1}p_{i-\half}p_{i+\half}\\
&\quad+s_is_{i-1}L_{i-\thalf}s_{i-1}s_ip_{i-\half}
+s_is_{i-1}\sigma_{i-\half}s_{i-1}s_ip_{i-\half}\\
&=-L_{i-1}p_{i-\half}p_{i+\half}
-s_ip_{i-\half}L_{i-1}s_ip_{i-\half}
+p_{i-\half}L_{i-1}p_{i-1}p_{i-\half}p_{i+\half}\\
&\quad+s_is_{i-1}L_{i-\thalf}p_{i+\half}s_{i-1}s_i
+s_is_{i-1}\sigma_{i-\half}p_{i+\half}s_{i-1}s_i.
\end{align*}
Substituting each of the above into \eqref{z-a-2}, and using part~\eqref{a-b-1} of Theorem~\ref{a-b} gives~\eqref{z-a-1}. The right hand side of \eqref{z-a-1} being fixed under the $*$ anti-involution on $\mathcal{A}_{i+\half}(z)$, the proof of~\eqref{a-c-7} is complete.\newline 
\eqref{a-c-8} We show that, after substituting the expression
\begin{align*}
L_{i-\half}=-p_{i-\half}L_{i-1}-L_{i-1}p_{i-\half}+p_{i-\half}L_{i-1}p_{i-1}p_{i-\half}+s_{i-1}L_{i-\half}s_{i-1}+\sigma_{i-\half}
\end{align*}
into the definition~\eqref{jm-i-2}, conjugation by $s_{i-1}$ permutes the summands of $L_{i+\half}$ as follows:
\begin{enumerate}[label=(\roman{*}), ref=\roman{*},leftmargin=0pt,itemindent=1.5em]
\item $s_{i-1}(p_{i+\half}p_{i-\half}L_{i-1}s_{i-1})s_{i-1}
=p_{i+\half}p_{i-\half}L_{i-1}p_{i}p_{i-\half}p_ip_{i+\half}$, \label{c-1-i}
\item $s_{i-1}p_{i+\half}s_{i-1}L_{i-1}p_{i-\half}s_{i-1}
=s_ip_{i-\half}L_{i-1}s_ip_{i-\half}$, \label{c-1-ii}
\item $s_{i-1}p_{i+\half}p_{i-\half}L_{i-1}p_ip_{i-\half}s_{i-1}
=p_{i+\half}p_{i-\half}L_{i-1}p_ip_{i-\half}$, \label{c-1-iii}
\item $s_{i-1}(p_{i+\half}s_{i-1}L_{i-1}s_{i-1})s_{i-1}
=s_ip_{i-\half}L_{i-1}s_{i}$,\label{c-1-iv}
\item $s_{i-1}p_{i+\half}\sigma_is_{i-1}=\sigma_ip_{i+\half}$,\label{c-1-v}
\item $s_{i-1}p_{i+\half}p_{i-\half}L_{i-1}s_{i-1}p_ip_{i+\half}s_{i-1}
=p_{i+\half}p_{i-\half}L_{i-1}s_{i-1}p_ip_{i+\half}$, \label{c-1-vi}
\item $s_{i-1}p_{i+\half}s_{i-1}L_{i-1}p_{i-\half}p_ip_{i+\half}s_{i-1}
=s_ip_{i-\half}p_ip_{i+\half}s_{i-1}L_{i-1}p_{i-\half}s_i$,\label{c-1-vii}
\item $s_{i-1}p_{i+\half}s_{i-1}L_{i-1}s_{i-1}p_ip_{i+\half}s_{i-1}
=s_ip_{i-\half}L_{i-1}p_{i-1}p_{i-\half}s_i$,\label{c-1-viii}
\item $s_{i-1}p_{i-\half}L_{i-1}s_{i-1}p_{i+\half}p_{i}p_{i-\half}s_{i-1}
=p_{i-\half}L_{i-1}s_{i-1}p_{i+\half}p_{i}p_{i-\half}$. \label{c-1-ix}
\end{enumerate}
The left hand side of~\eqref{c-1-i} is
\begin{align*}
s_{i-1}(p_{i+\half}p_{i-\half}L_{i-1}s_{i-1})s_{i-1}
&=p_{i+\half}p_{i-\half}L_{i-1}
\end{align*}
and the right hand side of~\eqref{c-1-i} is
\begin{align*}
p_{i+\half}p_{i-\half}L_{i-1}p_ip_{i-\half}p_ip_{i+\half}
=p_{i+\half}p_{i-\half}L_{i-1}p_ip_{i+\half}
=p_{i+\half}p_{i-\half}L_{i-1}.
\end{align*}
The left hand side of~\eqref{c-1-ii} is 
\begin{align*}
s_{i-1}p_{2+\half}s_{i-1}L_{i-1}p_{i-\half}s_{i-1}
=s_{i}p_{i-\half}s_{i}L_{i-1}p_{i-\half}=s_{i}p_{i-\half}L_{i-1}s_{i}p_{i-\half},
\end{align*}
which is the same as the right hand side of~\eqref{c-1-ii}. The left hand side of~\eqref{c-1-iii} is
\begin{align*}
s_{i-1}p_{i+\half}p_{i-\half}L_{i-1}p_ip_{i-\half}s_{i-1}=p_{i+\half}p_{i-\half}L_{i-1}p_ip_{i-\half},
\end{align*} 
which is the same as the right hand side of~\eqref{c-1-iii}. The left hand side of~\eqref{c-1-iv} is 
\begin{align*}
s_{i-1}(p_{i+\half}s_{i-1}L_{i-1}s_{i-1})s_{i-1}=s_{i}p_{i-\half}s_{i}L_{i-1}=s_{i}p_{i-\half}L_{i-1}s_{i},
\end{align*}
which is the same as the right hand side of~\eqref{c-1-iv}. The left hand side of~\eqref{c-1-v}  is 
\begin{align*}
s_{i-1}p_{i+\half}\sigma_is_{i-1}
=s_{i-1}p_{i+\half}\sigma_{i-\half}
=s_{i-1}\sigma_{i-\half}p_{i+\half}
=\sigma_{i}p_{i+\half},
\end{align*}
which is the same as the right hand side of~\eqref{c-1-v}. The left hand side of~\eqref{c-1-vi} is 
\begin{align*}
s_{i-1}p_{i+\half}p_{i-\half}L_{i-1}s_{i-1}p_ip_{i+\half}s_{i-1}
&=p_{i+\half}p_{i-\half}L_{i-1}s_{i-1}p_is_{i-1}s_ip_{i-\half}s_{i}\\
&=p_{i+\half}p_{i-\half}L_{i-1}p_{i-1}s_ip_{i-\half}s_{i}\\
&=p_{i+\half}p_{i-\half}L_{i-1}p_{i-1}p_{i-\half},
\end{align*}
which is the same as the right hand side of~\eqref{c-1-vi}.  The left hand side of~\eqref{c-1-vii} is 
\begin{align*}
&s_{i-1}p_{i+\half}s_{i-1}L_{i-1}p_{i-\half}p_ip_{i+\half}s_{i-1}
=s_{i}p_{i-\half}s_{i}L_{i-1}p_{i-\half}p_ip_{i+\half}s_{i-1}\\
&\quad=s_{i}p_{i-\half}L_{i-1}s_{i}p_{i-\half}p_ip_{i+\half}s_{i-1}
=s_{i}p_{i-\half}p_i\sigma_{i}s_{i-1}p_{i+\half}s_{i-1}s_ip_ip_{i+\half}s_{i-1}\\
&\quad=s_{i}p_{i-\half}p_i\sigma_{i-\half}p_{i+\half}s_{i-1}s_ip_is_{i-1}s_ip_{i-\half}s_{i}
=s_{i}p_{i-\half}p_i\sigma_{i-\half}p_{i+\half}p_ip_{i-\half}s_{i}\\
&\quad=s_{i}p_{i-\half}p_ip_{i+\half}\sigma_{i-\half}p_ip_{i-\half}s_{i}
=s_{i}p_{i-\half}p_ip_{i+\half}s_{i-1}L_{i-1}p_{i-\half}s_{i},
\end{align*}
which is the same as the right hand side of~\eqref{c-1-vii}. The left hand side of~\eqref{c-1-viii} is 
\begin{align*}
s_{i-1}p_{i+\half}s_{i-1}L_{i-1}s_{i-1}p_ip_{i+\half}s_{i-1}
&=s_{i}p_{i-\half}s_{i}L_{i-1}s_{i-1}p_is_{i-1}s_ip_{i-\half}s_{i}\\
&=s_{i}p_{i-\half}L_{i-1}s_{i}s_{i-1}p_is_{i-1}s_ip_{i-\half}s_{i}\\
&=s_{i}p_{i-\half}L_{i-1}p_{i-1}p_{i-\half}s_{i},
\end{align*}
which is the same as the right hand side of~\eqref{c-1-viii}. Since the statement~\eqref{c-1-ix} is evident from the relation $s_{i-1}p_{i-\half}=p_{i-\half}s_{i-1}=p_{i-\half}$, the proof of part~\eqref{a-c-8} is complete. \newline
\eqref{a-c-9} By parts~\eqref{a-c-6} and~\eqref{a-c-8}, $L_{i+\half}$ commutes with $\langle p_i,s_{i-1}\rangle$, and so with $p_{i-1}=s_{i-1}p_is_{i-1}$. 
\end{proof}
\begin{theorem}\label{s-1}
If $i=1,2,\ldots,$ then 
\begin{enumerate}[label=(\arabic{*}), ref=\arabic{*},leftmargin=0pt,itemindent=1.5em]
\item $L_{i+1}$ commutes with $\mathcal{A}_{i+\half}(z)$, and $\sigma_{i+1}$ commutes with $\mathcal{A}_{i-\half}(z)$,\label{s-1-1}
\item $L_{i+\half}$ commutes with $\mathcal{A}_{i}(z)$, and $\sigma_{i+\half}$ commutes with $\mathcal{A}_{i-1}(z)$.\label{s-1-2}
\end{enumerate}
Consequently, the family of elements $(L_i,L_{i+\half}\,:\,i=0,1,\ldots)$ is pairwise commutative.
\end{theorem}
\begin{proof}
\eqref{s-1-1} Observe that $L_1$ commutes with $\mathcal{A}_{1-\half}(z)$ and $L_{2}$ commutes with $\mathcal{A}_{1+\half}(z)$, while $\sigma_2$ commutes with $\mathcal{A}_{1-\half}(z)$ and $\sigma_3$ commutes with $\mathcal{A}_{2-\half}(z)$. Since
\begin{align*}
&\text{$L_{i+1}$ commutes with $\langle s_{i-1},p_{i-1},p_{i},p_{i-\half},p_{i+\half}\rangle$, if $i\ge2$,  and }\\
&\text{$\sigma_{i+1}$ commutes with $\langle s_{i-2},p_{i-2},p_{i-1},p_{i-\thalf},p_{i-\half}\rangle$, if $i\ge3$,}
\end{align*}
it suffices to show that 
\begin{enumerate}[label=(\roman{*}), ref=\roman{*}]
\item $L_{i+1}$ commutes with $\mathcal{A}_{i-\thalf}(z)$, if $i\ge2$, and \label{s-1-1-1}
\item $\sigma_{i+1}$ commutes with $\mathcal{A}_{i-\frac{5}{2}}(z)$, if $i\ge3$.\label{s-1-1-2}
\end{enumerate}
If $i\ge2$, then induction on $i$ shows that $L_{i+1}$ commutes with $\mathcal{A}_{i-2}(z)$, while, if $i\ge3$, the fact that $L_{i+1}$ commutes with $p_{i-\thalf}$ follows from induction on $i$, and the fact that $\sigma_{i+1}$ commutes with $p_{i-\thalf}$. Similarly, if $i\ge3$, then induction on $i$ shows that $\sigma_{i+1}$ commutes with $\mathcal{A}_{i-3}(z)$, and that, if $i\ge4$, then $\sigma_{i+1}$ commutes with $p_{i-\frac{5}{2}}$. \newline
\eqref{s-1-2}  Observe that $L_{0+\half}$ commutes with $\mathcal{A}_{0}(z)$ and $L_{1+\half}$ commutes with $\mathcal{A}_{1}(z)$, while $\sigma_{1+\half}$ commutes with $\mathcal{A}_{0}(z)$ and $\sigma_{2+\half}$ commutes with $\mathcal{A}_{1}(z)$. Since 
\begin{align*}
&\text{$L_{i+\half}$ commutes with $\langle s_{i-1}, p_{i-1}, p_i, p_{i-\half}\rangle$, if $i\ge2$,  and }\\
&\text{$\sigma_{i+\half}$ commutes with $\langle s_{i-2},p_{i-2},p_{i-1},p_{i-\thalf}\rangle$, if $i\ge3$,}
\end{align*}
it suffices to show that 
\begin{enumerate}[label=(\roman{*}), ref=\roman{*}]
\item $L_{i+\half}$ commutes with $\mathcal{A}_{i-\thalf}(z)$, if $i\ge2$, and \label{s-1-1-3}
\item $\sigma_{i+\half}$ commutes with $\mathcal{A}_{i-\frac{5}{2}}(z)$, if $i\ge3$.\label{s-1-1-4}
\end{enumerate}
If $i\ge2$, then induction on $i$ shows that $L_{i+\half}$ commutes with $\mathcal{A}_{i-2}(z)$, while, if $i\ge3$, the fact that $L_{i+\half}$ commutes with $p_{i-\thalf}$ follows from induction on $i$, and the fact that $\sigma_{i+\half}$ commutes with $p_{i-\thalf}$. Similarly, if $i\ge3$, then induction on $i$ shows that $\sigma_{i+\half}$ commutes with $\mathcal{A}_{i-3}(z)$, and that, if $i\ge4$, then $\sigma_{i+\half}$ commutes with $p_{i-\frac{5}{2}}$. 
\end{proof}
\begin{proposition}\label{jm:commute}
For $i=1,2,\dots,$ the following statements hold:
\begin{enumerate}[label=(\arabic{*}), ref=\arabic{*},leftmargin=0pt,itemindent=1.5em]
\item $(L_{i+\half}+L_{i+1})p_{i+1}=p_{i+1}(L_{i+\half}+L_{i+1})=z p_{i+1}$,\label{jm:commute:4.b}
\item $(L_{i}+L_{i+\half})p_{i+\half}=p_{i+\half}(L_{i}+L_{i+\half})=z p_{i+\half}$,\label{jm:commute:3}
\item $(L_{i-\half}+L_i+L_{i+\half}+L_{i+1})s_i=s_i(L_{i-\half}+L_i+L_{i+\half}+L_{i+1})$.\label{jm:commute:5}
\end{enumerate}
\end{proposition}
\begin{proof}
\eqref{jm:commute:4.b} It suffices to observe that:
\begin{enumerate}[label=(\roman{*}), ref=\roman{*},leftmargin=0pt,itemindent=1.5em]
\item $L_ip_{i+\half}p_{i+1}=\sigma_{i+1}p_{i+1}$,
\item $p_{i+\half}L_ip_{i+1}=p_{i+\half}L_ip_{i+1}p_{i+\half}p_{i+1}$,
\item $p_{i+\half}L_ip_{i}p_{i+\half}p_{i+1}=p_{i+\half}L_is_ip_{i+1}$,
\item $s_iL_{i-\half}s_ip_{i+1}+s_iL_is_ip_{i+1}=z p_{i+1}$,\label{j-m-c-iv}
\item $\sigma_{i+\half}p_{i+1}=s_iL_ip_{i+\half}p_{i+1}$.
\end{enumerate}
With the exception of~\eqref{j-m-c-iv}, each of the statements above is evident from the defining relations or from what we have already shown. Since $(L_\half+L_1)p_1=zp_1$,
\begin{align}\label{vin:1}
s_iL_{i-\half}s_ip_{i+1}+s_iL_is_ip_{i+1}=s_i(L_{i-\half}p_i+L_ip_i)s_i=z s_ip_is_i=z p_{i+1},
\end{align}
gives \eqref{j-m-c-iv} by induction.\newline 
\eqref{jm:commute:3} Using the expression~\eqref{vin:1}, we have 
\begin{align*}
zp_{i+\half}&=zp_{i+\half}p_{i+1}p_{i+\half}\\
&=p_{i+\half}(s_iL_{i-\half}s_ip_{i+1}+s_iL_is_ip_{i+1})p_{i+\half}\\
&=p_{i+\half}L_{i-\half}s_ip_{i+1}p_{i+\half}+p_{i+\half}L_is_ip_{i+1}p_{i+\half}\\
&=L_{i-\half}p_{i+\half}p_{i}p_{i+\half}+p_{i+\half}L_ip_{i}p_{i+\half}\\
&=L_{i-\half}p_{i+\half}+p_{i+\half}L_ip_{i}p_{i+\half}, 
\end{align*}
which yields
\begin{align*}
(L_i+L_{i+\half})p_{i+\half}&=L_ip_{i+\half}-p_{i+\half}L_ip_{i+\half}-L_ip_{i+\half}+p_{i+\half}L_ip_ip_{i+\half}\\
&\qquad+s_{i}L_{i-\half}p_{i+\half}+\sigma_{i+\half}p_{i+\half}\\
&=-p_{i+\half}+zp_{i+\half}-L_{i-\half}p_{i+\half}+L_{i-\half}p_{i+\half}+p_{i+\half}\\
&=zp_{i+\half}. 
\end{align*}
\eqref{jm:commute:5} Since $\sigma_{i+1}s_i=\sigma_{i+\half}=s_i\sigma_{i+1}$, we have
\begin{align*}
(L_{i-\half}+L_i+L_{i+\half}+L_{i+1})s_i&=L_{i-\half}s_i+L_is_i-p_{i+\half}L_is_i-L_ip_{i+\half}+p_{i+\half}L_ip_ip_{i+\half}\\
&\quad +s_iL_{i-\half}+\sigma_{i+\half}s_i-p_{i+\half}L_i-s_iL_ip_{i+\half}\\
&\quad+p_{i+\half}L_ip_{i+1}p_{i+\half}+s_iL_i+\sigma_{i+1}s_i\\
&\quad=s_i(L_{i-\half}+L_i+L_{i+\half}+L_{i+1}),
\end{align*}
as required.
\end{proof}
\begin{theorem}\label{cent}
If $k=0,1,\ldots,$ then 
\begin{enumerate}[label=(\arabic{*}), ref=\arabic{*},leftmargin=0pt,itemindent=1.5em]
\item the element $z_{k+\half}=L_{\half}+L_{1}+L_{1+\half}+\cdots+L_{k+\half}$ is central in $\mathcal{A}_{k+\half}(z)$,\label{cent-1}
\item the element $z_{k+1}=L_{\half}+L_{1}+L_{1+\half}+\cdots+L_{k+1}$ is central in $\mathcal{A}_{k+1}(z)$.\label{cent-2}
\end{enumerate}
\end{theorem}
\begin{proof}
\eqref{cent-1} We first show that $z_{k+\half}$ commutes with $p_1,\ldots p_{k}$ and $p_{1+\half},\ldots,p_{k+\half}$. If $i=1,\ldots,k$, then $p_i$ commutes with $z_{i-1}$ and with  $L_{i+\half}+L_{i+1}+\cdots+L_{k+\half}$. Since $p_i$ also commutes with $L_{i-\half}+L_{i}$, it follows that $p_{i}$ commutes with $z_{k+\half}$. Similarly, since $p_{i+\half}$ commutes with $z_{i-\half}$, $L_{i}+L_{i+\half}$, and with $L_{i+1}+L_{i+\thalf}+\ldots L_{k+\half}$, it follows that $p_{i+\half}$ commutes with $z_{k+\half}$. Next, suppose that $i=1,\ldots,k-1$. Since $s_i$ commutes with $z_{i-1}$, $L_{i-\half}+L_i+L_{i+\half}+L_{i+1}$, and with $L_{i+\thalf}+L_{i+2}+\cdots+L_{k+\half}$, it is evident that $s_i$ commutes with $z_{k+\half}$.  As $\mathcal{A}_{k+\half}(z)$ is generated by $p_1,\ldots,p_k$, $p_{1+\half},\ldots,p_{k+\half}$, and $s_1,\ldots,s_{k-1}$, we conclude that $z_{k+\half}$ is central in $\mathcal{A}_{k+\half}(z)$. \newline
\eqref{cent-2} Given that $z_{k+\half}$ is central in $\mathcal{A}_{k+\half}(z)$, it suffices to observe that $s_{k}$ commutes with $z_{k+1}$. 
\end{proof} 
\begin{proposition}\label{c-e}
For $i=1,2,\ldots,$ the following statements hold:
\begin{enumerate}[label=(\arabic{*}), ref=\arabic{*},leftmargin=0pt,itemindent=1.5em]
\item $(L_i+L_{i+\half}+L_{i+1})\sigma_{i+1}=\sigma_{i+1}(L_i+L_{i+\half}+L_{i+1})$,\label{c-e-1}
\item $(L_{i-\half}+L_i+L_{i+\half})\sigma_{i+\half}=\sigma_{i+\half}(L_{i-\half}+L_i+L_{i+\half})$.\label{c-e-2}
\end{enumerate}
\begin{proof}
\eqref{c-e-1} Since $\sigma_{i+1}$ commutes with 
\begin{align*}
z_{i+1}=z_{i-1}+L_{i-\half}+L_{i}+L_{i+\half}+L_{i+1},
\end{align*}
and $\sigma_{i+1}$ also commutes with $z_{i-1}+L_{i-\half}\in\mathcal{A}_{i-\half}(z)$, it follows that $\sigma_{i+1}$ commutes with $L_{i}+L_{i+\half}+L_{i+1}$. \newline
\eqref{c-e-2} Given that $\sigma_{1+\half}=1$, we may suppose that $i=2,3,\ldots.$ Since $\sigma_{i+\half}$ commutes with 
\begin{align*}
z_{i+\half}=z_{i-\thalf}+L_{i-1}+L_{i-\half}+L_i+L_{i+\half},
\end{align*}
and $\sigma_{i+\half}$ also commutes with $z_{i-\thalf}+L_{i-1}\in\mathcal{A}_{i-1}(z)$, it follows that $\sigma_{i+\half}$ commutes with $L_{i-\half}+L_{i}+L_{i+\half}$. 
\end{proof}
\end{proposition}
\section{A Presentation for Partition Algebras}\label{a-1-3}
In this section, we rewrite the presentation for $\mathcal{A}_{k}(z)$ given by~\cite{HR:2005} in Theorem~\ref{hr-presentation} in terms of the elements  $\sigma_{i},\sigma_{i+\half},p_{i},p_{i+\half}$.  
\begin{theorem}\label{n-pres}
If $k=1,2,\ldots,$ then $\mathcal{A}_{k}(z)$ is the unital associative algebra presented by the generators 
\begin{align*}
p_1,p_{1+\half},p_2,\ldots,p_{k-\half},p_k,\sigma_2,\sigma_{2+\half},\sigma_3,\ldots,\sigma_{k-\half},\sigma_k,
\end{align*}
and the relations:
\begin{enumerate}
\item (Involutions)\label{inv}
\begin{enumerate}
\item $\sigma_{i+\half}^2=1$, for $i=2,\ldots,k-1$.\label{inv-1}
\item $\sigma_{i+1}^2=1$, for $i=1,\ldots,k-1$.\label{inv-2}
\end{enumerate}
\item (Braid--like relations) \label{brd}
\begin{enumerate}
\item $\sigma_{i+1}\sigma_{j+\half}=\sigma_{j+\half}\sigma_{i+1}$, if $j\ne i+1$.\label{brd-1}
\item $\sigma_{i}\sigma_{j}=\sigma_{j}\sigma_{i}$, if $j\ne i+1$.\label{brd-2}
\item $\sigma_{i+\half}\sigma_{j+\half}=\sigma_{j+\half}\sigma_{i+\half}$, if $j\ne i+1$. \label{brd-3}
\item $s_is_{i+1}s_i=s_{i+1}s_is_{i+1}$, for $i=1,\ldots,k-2$, where\label{brd-4}
\begin{align*}
s_\ell=
\begin{cases}
\sigma_{\ell+1},&\text{if $\ell=1$,}\\
\sigma_{\ell+\half}\sigma_{\ell+1},&\text{if $\ell=2,\ldots,k-1$,}
\end{cases}
\end{align*}
are the Coxeter generators for $\mathfrak{S}_k$.
\end{enumerate}
\item (Idempotent relations)\label{ide}
\begin{enumerate}
\item $p_i^2=zp_i$, for $i=1,\ldots,k$.\label{ide-1}
\item $p_{i+\half}^2=p_{i+\half}^2$, for $i=1,\ldots,k-1$.\label{ide-2}
\item $\sigma_{i+1}p_{i+\half}=p_{i+\half}\sigma_{i+1}=p_{i+\half}$, for $i=1,\ldots,k-1$.\label{ide-3}
\item $\sigma_{i+\half}p_{i+\half}=p_{i+\half}\sigma_{i+\half}=p_{i+\half}$, for $i=1,\ldots,k-1$. \label{ide-4}
\item $\sigma_{i+\half}p_{i}p_{i+1}=\sigma_{i+1}p_{i}p_{i+1}$, for $i=1,\ldots,k-1$.\label{ide-5}
\item $p_{i}p_{i+1}\sigma_{i+\half}=p_{i}p_{i+1}\sigma_{i+1}$, for $i=1,\ldots,k-1$.\label{ide-6}
\end{enumerate}
\item (Commutation relations)\label{com}
\begin{enumerate}
\item $p_ip_j=p_jp_i$, for $i,j=1,\ldots,k$.\label{com-1}
\item $p_{i+\half}p_{j+\half}=p_{j+\half}p_{i+\half}$, for $i,j=1,\ldots,k-1$.\label{com-2}
\item $p_{i+\half}p_j=p_jp_{i+\half}$, for $j\ne i,i+1$. \label{com-3}
\item $\sigma_ip_j=p_j\sigma_i$ if $j\ne i-1,i$. \label{com-4}
\item $\sigma_ip_{j+\half}=p_{j+\half}\sigma_i$, if $j\ne i$. \label{com-5}
\item $\sigma_{i+\half}p_j=p_j\sigma_{i+\half}$, if $j\ne i,i+1$. \label{com-6}
\item $\sigma_{i+\half}p_{j+\half}=p_{j+\half}\sigma_{i+\half}$, if $j\ne i-1$. \label{com-7}
\item $\sigma_{i+\half}p_i\sigma_{i+\half}=\sigma_{i+1}p_{i+1}\sigma_{i+1}$, for $i=1,\ldots,k-1$.\label{com-8}
\item $\sigma_{i+\half}p_{i-\half}\sigma_{i+\half}=\sigma_{i}p_{i+\half}\sigma_{i}$, for $i=2,\ldots,k-1$. \label{com-9}
\end{enumerate}
\item (Contraction relations)
\begin{enumerate}
\item $p_{i+\half}p_{j}p_{i+\half}=p_{i+\half}$, for $j=i,i+1$.   \label{con-1}
\item $p_ip_{j-\half}p_i=p_i$, for $j=i,i+1$. \label{con-2}
\end{enumerate}
\end{enumerate}
\end{theorem}
\begin{proof}
We first show that the relations given in the statement above are a consequence the presentation given by by~\cite{HR:2005} in Theorem~\ref{hr-presentation}.\newline 
\eqref{inv-1}, \eqref{inv-2} Follow from Proposition~\ref{sigma-inv}.\newline
\eqref{brd-1} To see that $\sigma_{i+1}$ and $\sigma_{i+\half}$ commute, we use Proposition~\ref{z-0} and Proposition~\ref{sigma-inv} to obtain
\begin{align}
\sigma_{i+1}=\sigma_{i+\half}s_i=s_i\sigma_{i+\half}
&&\text{and}&&
\sigma_{i+\half}\sigma_{i+1}=\sigma_{i+1}\sigma_{i+\half}=s_i&&\text{(for $i=1,\ldots,k-1$)}\label{gen}
\end{align}
where $s_i$ is a Coxeter generator for $\mathfrak{S}_{k}\subset\mathcal{A}_{k}(z)$. Theorem~\ref{s-1} shows that $\sigma_{j}$ commutes with $\mathcal{A}_{j-\thalf}(z)$, and hence that $\sigma_j$ commutes with $\sigma_{1+\half},\ldots,\sigma_{j-\thalf}$. From Theorem~\ref{s-1}, $\sigma_{i+\half}$ commutes with $\mathcal{A}_{i-1}(z)$, and hence with $\sigma_{2},\ldots,\sigma_{i-1}$. In general, however, $\sigma_i$ and $\sigma_{i+\half}$ do not commute.\newline
\eqref{brd-2}, \eqref{brd-3}  Theorem~\ref{s-1} shows that $\sigma_{i+1}$ commutes with $\sigma_2,\ldots,\sigma_{i-1}$ and that $\sigma_{i+\half}$ commutes with $\sigma_{1+\half},\ldots,\sigma_{i-\thalf}$.\newline
\eqref{brd-4} Follows from~\eqref{gen}, whereby for $j=1,\ldots,k-1$, each product $\sigma_{j+\half}\sigma_{j+1}=s_j$ is a Coxeter generator for $\mathfrak{S}_k\subset\mathcal{A}_k(z)$.\newline
\eqref{ide-1}, \eqref{ide-2} Are included in the set of relations given by~\cite{HR:2005}. \newline
\eqref{ide-3}, \eqref{ide-4} That $\sigma_{i+1}p_{i+\half}=p_{i+\half}\sigma_{i+1}=p_{i+\half}$ is given in Proposition~\ref{prel:a}. Proposition~\ref{z-0} shows that $\sigma_{i+\half}p_{i+\half}=s_i\sigma_{i+1}p_{i+\half}=s_ip_{i+\half}=p_{i+\half}$.\newline 
\eqref{ide-5}, \eqref{ide-6} Proposition~\ref{z-0} and Proposition~\ref{sigma-inv} show that 
\begin{align*}
p_{i}p_{i+1}=s_ip_{i}p_{i+1}=\sigma_{i+\half}\sigma_{i+1}p_ip_{i+1}
&&\text{and}&&
\sigma_{i+\half}p_{i}p_{i+1}=\sigma_{i+1}p_{i}p_{i+1}.
\end{align*}
\eqref{com-1}--\eqref{com-3} Are included in the set of relations given by~\cite{HR:2005}. \newline
\eqref{com-4}, \eqref{com-5}  By Theorem~\ref{s-1}, $\sigma_i$ commutes with $\mathcal{A}_{i-\thalf}(z)$, and hence with $p_1,\ldots p_{i-2}$ and with $p_{1+\half},\ldots,p_{i-\thalf}$. Proposition~\ref{prel:a} shows that $\sigma_i$ commutes with $p_{i-\half}$.\newline
\eqref{com-6}, \eqref{com-7} By Theorem~\ref{s-1}, $\sigma_{i+\half}$ commutes with $\mathcal{A}_{i-1}(z)$, and hence with $p_1,\ldots p_{i-1}$ and with $p_{1+\half},\ldots,p_{i-\thalf}$. From~\eqref{ide-4}, it follows that $\sigma_{i+\half}$ commutes with $p_{i+\half}$. \newline 
\eqref{com-8} Proposition~\ref{z-0} and Proposition~\ref{sigma-inv} show that 
\begin{align*}
p_{i}=s_ip_{i+1}s_i=\sigma_{i+\half}\sigma_{i+1}p_{i+1}\sigma_{i+1}\sigma_{i+\half}
&&\text{and}&&
\sigma_{i+\half}p_{i}\sigma_{i+\half}=\sigma_{i+1}p_{i+1}\sigma_{i+1}.
\end{align*}
\eqref{com-9} Proposition~\ref{prel:a} shows that $p_{i-\half}s_i\sigma_i=s_i\sigma_ip_{i+\half}$. Proposition~\ref{z-0} and Proposition~\ref{sigma-inv}, together with the fact that $\sigma_{i+1}$ commutes with $p_{i-\half}$ give
\begin{align*}
p_{i-\half}\sigma_{i+1}\sigma_{i+\half}\sigma_i=\sigma_{i+1}\sigma_{i+\half}\sigma_ip_{i+\half}
&&\text{and}&&
p_{i-\half}\sigma_{i+\half}\sigma_i=\sigma_{i+\half}\sigma_ip_{i+\half}.
\end{align*}
Multiplying both sides of the last expression by $\sigma_i\sigma_{i+\half}$ on the left, and using Proposition~\ref{sigma-inv} once more shows that
\begin{align*}
\sigma_i\sigma_{i+\half}p_{i-\half}\sigma_{i+\half}\sigma_i=\sigma_i\sigma_{i+\half}\sigma_{i+\half}\sigma_ip_{i+\half}=p_{i+\half}
&&\text{and}&&
\sigma_{i+\half}p_{i-\half}\sigma_{i+\half}=\sigma_ip_{i+\half}\sigma_i,
\end{align*}
as required. \newline
\eqref{con-1}, \eqref{con-2} Are included in the set of relations given by~\cite{HR:2005}. 

Next, we derive the relations given by~\cite{HR:2005} in Theorem~\ref{hr-presentation} from the relations~\eqref{inv-1}--\eqref{con-2} above. \newline
(\ref{cox-0}\ref{cox-i}) By the relations \eqref{inv-1}, \eqref{inv-2} and \eqref{brd-1},  
\begin{align*}
\sigma_{i+1}\sigma_{i+\half}=\sigma_{i+\half}\sigma_{i+\half}
&&\text{and}&&
(\sigma_{i+\half}\sigma_{i+1})^2=1, &&\text{for $i=1,\ldots,k-1$.}
\end{align*}
Thus, writing $s_i=\sigma_{i+\half}\sigma_{i+1}$, for $i=1,\ldots,k-1$, we recover (\ref{cox-0}\ref{cox-i}).\newline
(\ref{cox-0}\ref{cox-ii}) If $j\ne i+1$, then, by~\eqref{brd-2} and~\eqref{brd-3},
\begin{align*}
s_is_j=\sigma_{i+\half}\sigma_{i+1}\sigma_{j+\half}\sigma_{j+1}=\sigma_{j+\half}\sigma_{j+1}\sigma_{i+\half}\sigma_{i+1}=s_js_i,
\end{align*}
as required. \newline 
(\ref{cox-0}\ref{cox-iii}) Is equivalent to~\eqref{brd-4} with $s_i=\sigma_{i+\half}\sigma_{i+1}$, for $i=1,\ldots,k-1$. \newline 
(\ref{ide-0}\ref{ide-i}), (\ref{ide-0}\ref{ide-ii}) Are identical to the relations~\eqref{ide-1} and~\eqref{ide-2}. \newline
(\ref{ide-0}\ref{ide-iii}) With $s_i=\sigma_{i+1}\sigma_{i+\half}$, the relations~\eqref{ide-3} and~\eqref{ide-4} give
\begin{align*}
s_ip_{i+\half}&=\sigma_{i+\half}\sigma_{i+1}p_{i+\half}=\sigma_{i+\half}p_{i+\half}=p_{i+\half}
\intertext{and}
p_{i+\half}s_i&=p_{i+\half}\sigma_{i+1}\sigma_{i+\half}=p_{i+\half}\sigma_{i+\half}=p_{i+\half},
\end{align*}
as required.\newline
(\ref{ide-0}\ref{ide-iv}) With $s_i=\sigma_{i+1}\sigma_{i+\half}$, the relations~\eqref{brd-1},~\eqref{ide-5} and~\eqref{ide-6} give
\begin{align*}
s_ip_ip_{i+1}=\sigma_{i+1}\sigma_{i+\half}p_ip_{i+1}=p_ip_{i+1}
&&\text{and}&&
p_{i}p_{i+1}s_i=p_{i}p_{i+1}\sigma_{i+1}\sigma_{i+\half}=p_{i}p_{i+1},
\end{align*}
as required. \newline
(\ref{com-0}\ref{com-i})-(\ref{com-0}\ref{com-iii}) Are identical to the relations~\eqref{com-1}-\eqref{com-3}. \newline
(\ref{com-0}\ref{com-iv}) If $j\ne i,i+1$, then the relations~\eqref{com-4} and~\eqref{com-6} give
\begin{align*}
s_ip_j=\sigma_{i+\half}\sigma_{i+1}p_j=\sigma_{i+\half}p_j\sigma_{i+1}=p_j\sigma_{i+\half}\sigma_{i+1}=p_js_i,
\end{align*}
as required.\newline
(\ref{com-0}\ref{com-v}) If $j\ne i-1,i+1$, then the relations~\eqref{com-5} and~\eqref{com-6} give 
\begin{align*}
s_ip_{j+\half}
=\sigma_{i+\half}\sigma_{i+1}p_{j+\half}
=\sigma_{i+\half}p_{j+\half}\sigma_{i+1}
=p_{j+\half}\sigma_{i+\half}\sigma_{i+1}
=p_{j+\half}s_i,
\end{align*}
as required.\newline
(\ref{com-0}\ref{com-vi}) From the relations~\eqref{inv-2} and\eqref{com-8}, 
\begin{align*}
s_ip_is_i=\sigma_{i+1}\sigma_{i+\half}p_i\sigma_{i+\half}\sigma_{i+1}=\sigma_{i+1}^2p_{i+1}\sigma_{i+1}^2=p_{i+1},
\end{align*}
as required. \newline
(\ref{com-0}\ref{com-vii}) From the relations~\eqref{com-5},~\eqref{com-7} and~\eqref{inv-1},~\eqref{inv-2},
\begin{align*}
s_ip_{i-\half}s_i
=\sigma_{i+\half}\sigma_{i+1}p_{i-\half}\sigma_{i+1}\sigma_{i+\half}
=\sigma_i\sigma_{i-\half}p_{i+\half}\sigma_{i-\half}\sigma_i=s_{i-1}p_{i+\half}s_{i-1},
\end{align*}
as required. \newline
(\ref{con-0}\ref{con-i}), (\ref{con-0}\ref{con-ii}) Are identical to the relations~\eqref{con-1} and~\eqref{con-2}. 
\end{proof}
\begin{proposition}\label{sigma-inv}
If $i=0,1,\dots$,  then $(\sigma_{i+\half})^2=1$ and $(\sigma_{i+1})^2=1$.
\end{proposition}
\begin{proof}
Given that $\sigma_{i+\half}s_{i}=s_i\sigma_{i+\half}=\sigma_{i+1}$, we obtain 
\begin{align*}
\sigma_{i+\half}^2 =(s_i\sigma_{i+1})^2=\sigma_{i+1}^2. 
\end{align*}
It therefore suffices to show that $\sigma_{i+1}^2=1$. By definition $\sigma_1=1$, so we proceed by induction. After taking the square of the right hand side of the definition~\eqref{sigma-i-1}, the proposition will follow from the relations:
\begin{align}\label{state:8}
(p_{i-\half}p_ip_{i+\half}s_{i-1}L_{i-1}p_{i-\half}s_i)^2=(p_{i-\half}p_ip_{i+\half}s_{i-1}L_{i-1}p_{i-\half}s_i)(p_{i-\half}L_{i-1}s_ip_{i-\half});
\end{align}
\begin{multline}\label{state:3}
(p_{i-\half}p_ip_{i+\half}s_{i-1}L_{i-1}p_{i-\half}s_i)(s_ip_{i-\half}L_{i-1}s_{i-1}p_{i+\half}p_ip_{i-\half})\\=(p_{i-\half}p_ip_{i+\half}s_{i-1}L_{i-1}p_{i-\half}s_i)(s_{i-1}s_i\sigma_is_is_{i-1});
\end{multline}
\begin{multline}\label{state:6}
(p_{i-\half}p_ip_{i+\half}s_{i-1}L_{i-1}p_{i-\half}s_i)(s_ip_{i-\half}L_{i-1}s_ip_{i-\half}s_i)\\=(s_{i-1}s_i\sigma_is_is_{i-1})(s_ip_{i-\half}L_{i-1}s_ip_{i-\half}s_i);
\end{multline}
\begin{multline}\label{state:9}
(s_ip_{i-\half}L_{i-1}s_{i-1}p_{i+\half}p_{i}p_{i-\half})(p_{i-\half}p_ip_{i+\half}s_{i-1}L_{i-1}p_{i-\half}s_i)\\
=(s_ip_{i-\half}L_{i-1}s_{i-1}p_{i+\half}p_{i}p_{i-\half})(s_{i-1}s_i\sigma_is_is_{i-1});
\end{multline}
\begin{align}\label{state:4}
(s_ip_{i-\half}L_{i-1}s_{i-1}p_{i+\half}p_{i}p_{i-\half})^2=(p_{i-\half}L_{i-1}s_{i}p_{i-\half})(s_ip_{i-\half}L_{i-1}s_{i-1}p_{i+\half}p_{i}p_{i-\half});
\end{align}
\begin{align}\label{state:7}
(s_ip_{i-\half}L_{i-1}s_{i-1}p_{i+\half}p_{i}p_{i-\half})(p_{i-\half}L_{i-1}s_ip_{i-\half})=(s_{i-1}s_i\sigma_is_is_{i-1})(p_{i-\half}L_{i-1}s_ip_{i-\half});
\end{align}
\begin{align}\label{state:10}
(p_{i-\half}L_{i-1}s_ip_{i-\half})(p_{i-\half}p_ip_{i+\half}s_{i-1}L_{i-1}p_{i-\half}s_i)=(p_{i-\half}L_{i-1}s_ip_{i-\frac{1}{{2}}})(s_{i-1}s_i\sigma_{i}s_is_{i-1});
\end{align}
\begin{multline}\label{state:5}
(s_ip_{i-\half}L_{i-1}s_{i-1}p_{i+\half}p_ip_{i-\half})(s_ip_{i-\half}L_{i-1}s_{i}p_{i-\half}s_i)\\=(p_{i-\half}L_{i-1}s_ip_{i-\half})(s_ip_{i-\half}L_{i-1}s_ip_{i-\half}s_i);
\end{multline}
\begin{align}\label{state:11}
(p_{i-\half}L_{i-1}s_ip_{i-\half})^2=(s_{i-1}s_i\sigma_{i}s_is_{i-1})(s_ip_{i-\half}L_{i-1}s_{i-1}p_{i+\half}p_ip_{i-\half});
\end{align}
\begin{align}\label{state:12}
(s_{i-1}s_i\sigma_is_is_{i-1})(p_{i-\half}p_ip_{i+\half}s_{i-1}L_{i-1}p_{i-\half}s_i)=(s_ip_{i-\half}L_{i-1}s_ip_{i-\half}s_i)^2;
\end{align}
\begin{align}\label{state:0}
(s_{i-1}s_i\sigma_is_is_{i-1})^2=1;
\end{align}
\begin{multline}\label{state:1}
(s_ip_{i-\half}L_{i-1}s_ip_{i-\half}s_i)(p_{i-\half}p_ip_{i+\half}s_{i-1}L_{i-1}p_{i-\half}s_i)\\
=(s_ip_{i-\half}L_{i-1}s_ip_{i-\half}s_i)(p_{i-\half}L_{i-1}s_ip_{i-\half});
\end{multline}
\begin{multline}\label{state:2}
(s_ip_{i-\half}L_{i-1}s_ip_{i-\half}s_i)(s_ip_{i-\half}L_{i-1}s_{i-1}p_{i+\half}p_ip_{i-\half})\\
=(s_ip_{i-\half}L_{i-1}s_ip_{i-\half}s_i)(s_{i-1}s_i\sigma_is_is_{i-1}).
\end{multline}
From the left hand side of~\eqref{state:8}, using the fact that $p_{i-\half}L_{i-1}p_{i-\half}=p_{i-\half}$, together with $p_{i+\half}s_{i-1}p_{i+\half}=p_{i-\half}p_{i+\half}$, we obtain
\begin{align*}
(p_{i-\half}p_ip_{i+\half}s_{i-1}L_{i-1}p_{i-\half}s_i)^2&=p_{i-\half}p_ip_{i+\half}s_{i-1}L_{i-1}p_{i-\half}p_{i+\half}s_{i-1}L_{i-1}p_{i-\half}s_i\\
&=p_{i-\half}p_ip_{i+\half}s_{i-1}p_{i+\half}L_{i-1}p_{i-\half}L_{i-1}p_{i-\half}\\
&=p_{i+\half}p_{i-\half}L_{i-1}p_{i-\half}L_{i-1}p_{i-\half}\\
&=p_{i-\half}p_{i+\half}.
\end{align*}
Similarly, from the right hand side of~\eqref{state:8}, we obtain
\begin{align*}
(p_{i-\half}p_ip_{i+\half}s_{i-1}L_{i-1}p_{i-\half}s_i)(p_{i-\half}L_{i-1}s_ip_{i-\half})&=p_{i-\half}p_ip_{i+\half}s_{i-1}L_{i-1}p_{i+\half}p_{i-\half}L_{i-1}s_ip_{i-\half}\\
&=p_{i-\half}p_ip_{i+\half}s_{i-1}L_{i-1}p_{i+\half}p_{i-\half}L_{i-1}s_ip_{i-\half}\\
&=p_{i-\half}p_ip_{i+\half}s_{i-1}p_{i+\half}L_{i-1}p_{i-\half}L_{i-1}s_ip_{i-\half}\\
&=p_{i-\half}p_ip_{i-\half}p_{i+\half}L_{i-1}p_{i-\half}L_{i-1}s_ip_{i-\half}\\
&=p_{i-\half}p_{i+\half},
\end{align*}
which demonstrates~\eqref{state:8}. Now consider the left hand side of~\eqref{state:3}
\begin{align*}
(p_{i-\half}p_ip_{i+\half}s_{i-1}L_{i-1}p_{i-\half}s_i)(s_ip_{i-\half}L_{i-1}s_{i-1}&p_{i+\half}p_ip_{i-\half})\\
&=p_{i-\half}p_ip_{i+\half}s_{i-1}L_{i-1}p_{i-\half}L_{i-1}s_{i-1}p_{i+\half}p_ip_{i-\half}\\
&=p_{i-\half}p_ip_{i+\half}\sigma_ip_{i-1}\sigma_ip_{i+\half}p_ip_{i-\half}\\
&=p_{i-\half}p_ip_{i+\half}\sigma_is_ip_{i-1}s_i\sigma_ip_{i+\half}p_ip_{i-\half}\\
&=p_{i-\half}p_i\sigma_is_ip_{i-\half}s_i\sigma_ip_ip_{i-\half}\\
&=p_{i-\half}p_i\sigma_is_{i-1}p_{i+\half}s_{i-1}\sigma_ip_ip_{i-\half}\\
&=p_{i-\half}p_i\sigma_{i-\half}p_{i+\half}\sigma_{i-\half}p_ip_{i-\half}\\
&=p_{i-\half}p_i(\sigma_{i-\half})^2p_{i+\half}p_ip_{i-\half}\\
&=p_{i-\half}.
\end{align*}
From the right hand side of~\eqref{state:3}, we obtain
\begin{align*}
(p_{i-\half}p_ip_{i+\half}s_{i-1}L_{i-1}p_{i-\half}s_i)(s_{i-1}s_i\sigma_is_is_{i-1})&=(p_{i-\half}L_{i-1}s_{i-1}p_{i+\half}p_ip_{i-\half}s_i)(s_{i-1}s_i\sigma_is_is_{i-1})\\
&=p_{i-\half}L_{i-1}s_{i-1}p_{i+\half}p_is_is_{i-1}p_{i+\half}\sigma_is_is_{i-1}\\
&=p_{i-\half}L_{i-1}s_{i-1}p_{i+\half}s_{i-1}p_{i+1}p_{i+\half}\sigma_is_is_{i-1}\\
&=p_{i-\half}L_{i-1}s_{i}p_{i-\half}s_{i}p_{i+1}p_{i+\half}\sigma_is_is_{i-1}\\
&=p_{i-\half}L_{i-1}s_{i}p_{i-\half}p_{i}s_{i}p_{i+\half}\sigma_is_is_{i-1}\\
&=p_{i-\half}s_{i}L_{i-1}p_{i-\half}p_{i}p_{i+\half}\sigma_is_is_{i-1}\\
&=p_{i-\half}s_{i}L_{i-1}p_{i-\half}p_{i}p_{i+\half}\sigma_{i-\half}s_{i-1}s_is_{i-1}\\
&=p_{i-\half}s_{i}L_{i-1}p_{i-\half}p_{i}\sigma_{i-\half}p_{i+\half}s_{i-1}s_is_{i-1}\\
&=p_{i-\half}s_{i}L_{i-1}p_{i-\half}p_{i}\sigma_{i-\half}s_{i-1}s_ip_{i-\half}\\
&=p_{i-\half}s_{i}L_{i-1}p_{i-\half}p_{i}\sigma_{i}s_ip_{i-\half}\\
&=p_{i-\half}s_{i}L_{i-1}p_{i-\half}L_{i-1}s_ip_{i-\half}\\
&=(p_{i-\half}s_{i}L_{i-1}p_{i-\half})(p_{i-\half}L_{i-1}s_ip_{i-\half}).
\end{align*}
Now, 
\begin{align}
\label{inter:d}(p_{i-\half}s_{i}L_{i-1}p_{i-\half})(p_{i-\half}L_{i-1}s_ip_{i-\half})&=p_{i-\half}L_{i-1}s_ip_{i-\half}s_iL_{i-1}p_{i-\half}\\
&=p_{i-\half}p_i\sigma_is_ip_{i-\half}s_i\sigma_{i}p_ip_{i-\half}\notag\\
&=p_{i-\half}p_i\sigma_is_{i-1}p_{i+\half}s_{i-1}\sigma_{i}p_ip_{i-\half}\notag\\
&=p_{i-\half}p_i\sigma_{i-\half}p_{i+\half}\sigma_{i-\half}p_ip_{i-\half}\notag\\
&=p_{i-\half}p_i(\sigma_{i-\half})^2p_{i+\half}p_ip_{i-\half}\notag\\
&=p_{i-\half},\notag
\end{align}
which completes the proof of~\eqref{state:3}. From the left hand side of~\eqref{state:6},
\begin{align*}
(p_{i-\half}p_ip_{i+\half}s_{i-1}L_{i-1}p_{i-\half}s_i)(s_ip_{i-\half}L_{i-1}s_ip_{i-\half}s_i)&=p_{i-\half}L_{i-1}s_{i-1}p_{i+\half}p_ip_{i-\half}L_{i-1}s_ip_{i-\half}s_i\\
&=p_{i-\half}L_{i-1}s_{i-1}p_{i+\half}p_ip_{i-\half}p_i\sigma_is_ip_{i-\half}s_i\\
&=p_{i-\half}L_{i-1}s_{i-1}p_{i+\half}p_i\sigma_is_ip_{i-\half}s_i\\
&=p_{i-\half}L_{i-1}s_{i-1}p_{i+\half}p_i\sigma_is_{i-1}p_{i+\half}s_{i-1}\\
&=p_{i-\half}L_{i-1}s_{i-1}p_{i+\half}p_i\sigma_{i-\half}p_{i+\half}s_{i-1}\\
&=p_{i-\half}L_{i-1}s_{i-1}p_{i+\half}p_ip_{i+\half}\sigma_{i-\half}s_{i-1}\\
&=p_{i-\half}L_{i-1}s_{i-1}\sigma_{i-\half}p_{i+\half}s_{i-1}\\
&=p_{i-\half}p_i(\sigma_{i})^2p_{i+\half}s_{i-1}\\
&=p_{i-\half}p_ip_{i+\half}s_{i-1}.
\end{align*}
The right hand side of~\eqref{state:6} gives
\begin{align*}
(s_{i-1}s_i\sigma_is_is_{i-1})(s_ip_{i-\half}L_{i-1}s_ip_{i-\half}s_i)&=s_{i-1}s_i\sigma_ip_{i+\half}s_{i-1}s_iL_{i-1}s_ip_{i-\half}s_i\\
&=s_{i-1}s_i\sigma_ip_{i+\half}s_{i-1}s_iL_{i-1}s_ip_{i-\half}s_i\\
&=s_{i-1}s_is_{i-1}\sigma_{i-\half}p_{i+\half}s_{i-1}s_iL_{i-1}s_ip_{i-\half}s_i\\
&=p_{i-\half}s_is_{i-1}\sigma_{i-\half}s_{i-1}s_iL_{i-1}s_ip_{i-\half}s_i\\
&=p_{i-\half}s_i\sigma_{i-\half}(s_{i-1})^2(s_i)^2L_{i-1}p_{i-\half}s_i\\
&=p_{i-\half}s_i\sigma_{i-\half}\sigma_ip_ip_{i-\half}s_i\\
&=p_{i-\half}s_is_{i-1}p_ip_{i-\half}s_i\\
&=p_{i-\half}p_{i-1}s_ip_{i-\half}s_i\\
&=p_{i-\half}p_{i-1}s_{i-1}p_{i+\half}s_{i-1}\\
&=p_{i-\half}p_{i}p_{i+\half}s_{i-1},
\end{align*}
which demonstrates~\eqref{state:6}. The left hand side of~\eqref{state:9} gives
\begin{align*}
(s_ip_{i-\half}L_{i-1}s_{i-1}p_{i+\half}p_ip_{i-\half})(p_{i-\half}p_ip_{i+\half}s_{i-1}L_{i-1}p_{i-\half}s_i)&=s_ip_{i-\half}L_{i-1}s_{i-1}p_{i+\half}s_{i-1}L_{i-1}p_{i-\half}s_i\\
&=s_ip_{i-\half}L_{i-1}s_{i}p_{i-\half}s_{i}L_{i-1}p_{i-\half}s_i\\
&=s_ip_{i-\half}s_iL_{i-1}p_{i-\half}L_{i-1}s_ip_{i-\half}s_i\\
&=s_{i-1}p_{i+\half}s_{i-1}\sigma_ip_ip_{i-\half}p_i\sigma_is_{i-1}p_{i+\half}s_{i-1}\\
&=s_{i-1}p_{i+\half}\sigma_{i-\half}p_i\sigma_{i-\half}p_{i+\half}s_{i-1}\\
&=s_{i-1}(\sigma_{i-\half})^2p_{i+\half}s_{i-1}\\
&=s_{i-1}p_{i+\half}s_{i-1},
\end{align*}
while the right hand side of~\eqref{state:9} gives
\begin{align*}
(s_ip_{i-\half}L_{i-1}s_{i-1}p_{i+\half}p_ip_{i-\half})(s_{i-1}s_i\sigma_is_is_{i-1})&=s_ip_{i-\half}L_{i-1}s_{i-1}p_{i+\half}p_ip_{i-\half}s_i\sigma_is_is_{i-1}\\
&=s_ip_{i-\half}L_{i-1}s_{i-1}p_{i+\half}p_is_is_{i-1}p_{i+\half}\sigma_{i-\half}s_is_{i-1}\\
&=s_ip_{i-\half}L_{i-1}s_{i-1}p_{i+\half}p_{i+1}s_{i-1}p_{i+\half}\sigma_{i-\half}s_is_{i-1}\\
&=s_ip_{i-\half}L_{i-1}s_{i-1}p_{i+\half}s_{i-1}p_{i+1}p_{i+\half}\sigma_{i-\half}s_is_{i-1}\\
&=s_ip_{i-\half}L_{i-1}s_{i}p_{i-\half}s_{i}p_{i+1}p_{i+\half}\sigma_{i-\half}s_is_{i-1}\\
&=s_{i}p_{i-\half}s_{i}L_{i-1}p_{i-\half}p_{i}s_{i}p_{i+\half}\sigma_{i-\half}s_is_{i-1}\\
&=s_{i-1}p_{i+\half}s_{i-1}\sigma_ip_ip_{i-\half}p_{i}p_{i+\half}\sigma_{i-\half}s_is_{i-1}\\
&=s_{i-1}p_{i+\half}\sigma_{i-\half}p_{i}p_{i+\half}\sigma_{i-\half}s_is_{i-1}\\
&=s_{i-1}\sigma_{i-\half}p_{i+\half}p_{i}p_{i+\half}\sigma_{i-\half}s_is_{i-1}\\
&=s_{i-1}p_{i+\half}(\sigma_{i-\half})^2s_{i-1}\\
&=s_{i-1}p_{i+\half}s_{i-1},
\end{align*}
which demonstrates~\eqref{state:9}. The statement~\eqref{state:4} is equivalent to~\eqref{state:8} which has already been verified. The right hand side of~\eqref{state:7} leads to
\begin{align*}
(s_ip_{i-\half}L_{i-1}s_{i-1}p_{i+\half}p_ip_{i-\half})(p_{i-\half}L_{i-1}s_{i}p_{i-\half})&=s_ip_{i-\half}L_{i-1}s_{i-1}p_{i+\half}s_{i-1}p_{i-1}p_{i-\half}L_{i-1}s_{i}p_{i-\half}\\
&=s_ip_{i-\half}L_{i-1}s_{i}p_{i-\half}s_{i}p_{i-1}p_{i-\half}L_{i-1}s_{i}p_{i-\half}\\
&=s_{i-1}p_{i+\half}s_{i-1}L_{i-1}p_{i-\half}s_{i}p_{i-1}p_{i-\half}L_{i-1}s_{i}p_{i-\half}\\
&=s_{i-1}p_{i+\half}s_{i-1}L_{i-1}p_{i-\half}p_{i-1}s_{i}p_{i-\half}s_{i}L_{i-1}p_{i-\half}\\
&=s_{i-1}p_{i+\half}s_{i-1}L_{i-1}p_{i-\half}p_{i-1}s_{i-1}p_{i+\half}s_{i-1}L_{i-1}p_{i-\half}\\
&=s_{i-1}p_{i+\half}s_{i-1}\sigma_ip_ip_{i-\half}p_{i}p_{i+\half}s_{i-1}\sigma_ip_ip_{i-\half}\\
&=s_{i-1}p_{i+\half}\sigma_{i-\half}p_ip_{i-\half}p_{i}p_{i+\half}\sigma_{i-\half}p_ip_{i-\half}\\
&=s_{i-1}p_{i+\half}\sigma_{i-\half}p_ip_{i+\half}\sigma_{i-\half}p_ip_{i-\half}\\
&=s_{i-1}\sigma_{i-\half}p_{i+\half}p_ip_{i+\half}\sigma_{i-\half}p_ip_{i-\half}\\
&=s_{i-1}p_{i+\half}(\sigma_{i-\half})^2p_ip_{i-\half}\\
&=s_{i-1}p_{i+\half}p_ip_{i-\half},
\end{align*}
while 
\begin{align*}
(s_{i-1}s_i\sigma_is_is_{i-1})(p_{i-\half}L_{i-1}s_ip_{i-\half})&=s_{i-1}s_i\sigma_is_ip_{i-\half}L_{i-1}s_ip_{i-\half}\\
&=s_{i-1}s_i\sigma_is_ip_{i-\half}s_iL_{i-1}p_{i-\half}\\
&=s_{i-1}s_i\sigma_is_{i-1}p_{i+\half}s_{i-1}\sigma_ip_ip_{i-\half}\\
&=s_{i-1}s_i\sigma_{i-\half}p_{i+\half}\sigma_{i-\half}p_ip_{i-\half}\\
&=s_{i-1}s_ip_{i+\half}(\sigma_{i-\half})^2p_ip_{i-\half}\\
&=s_{i-1}p_{i+\half}p_ip_{i-\half},
\end{align*}
as required. Since the statement~\eqref{state:10} is equivalent to~\eqref{state:7}, we consider~\eqref{state:5}. Using the relation $p_{i-\half}s_ip_{i-\half}=p_{i-\half}p_{i+\half}$,
\begin{align*}
(s_ip_{i-\half}L_{i-1}p_{i+\half}s_{i-1}p_{i+\half}p_ip_{i-\half})(s_ip_{i-\half}L_{i-1}&s_ip_{i-\half}s_i)\\
&=s_ip_{i-\half}L_{i-1}p_{i+\half}s_{i-1}p_{i+\half}p_{i-\half}L_{i-1}s_ip_{i-\half}s_i\\
&=s_ip_{i-\half}L_{i-1}p_{i+\half}s_{i-1}p_{i+\half}p_{i-\half}L_{i-1}p_{i-\half}\\
&=s_ip_{i-\half}L_{i-1}p_{i+\half}s_{i-1}p_{i+\half}p_{i-\half}\\
&=p_{i+\half}p_{i-\half}L_{i-1}p_{i-\half}\\
&=p_{i+\half}p_{i-\half}.
\end{align*}
On the other hand, 
\begin{align*}
(p_{i-\half}L_{i-1}s_ip_{i-\half})(s_ip_{i-\half}L_{i-1}s_ip_{i-\half}s_i)&=p_{i-\half}L_{i-1}s_ip_{i+\half}p_{i-\half}L_{i-1}s_ip_{i-\half}s_i\\
&=p_{i+\half}(p_{i-\half}L_{i-1}p_{i-\half})^2\\
&=p_{i+\half}p_{i-\half}.
\end{align*}
The left hand side of~\eqref{state:11} is given by~\eqref{inter:d}, and the right hand side by
\begin{align*}
(s_{i-1}s_i\sigma_is_is_{i-1})(s_ip_{i-\half}L_{i-1}s_{i-1}p_{i+\half}p_ip_{i-\half})&=s_{i-1}s_i\sigma_is_is_{i-1}s_ip_{i-\half}L_{i-1}s_{i-1}p_{i+\half}s_{i-1}p_{i-1}p_{i-\half}\\
&=s_{i-1}s_i\sigma_is_is_{i-1}s_ip_{i-\half}L_{i-1}s_{i}p_{i-\half}s_{i}p_{i-1}p_{i-\half}\\
&=s_{i-1}s_i\sigma_is_is_{i-1}s_ip_{i-\half}s_{i}L_{i-1}p_{i-\half}s_{i}p_{i-1}p_{i-\half}\\
&=s_{i-1}s_i\sigma_is_i(s_{i-1})^2p_{i+\half}s_{i-1}\sigma_ip_ip_{i-\half}s_{i}p_{i-1}p_{i-\half}\\
&=s_{i-1}s_is_{i-1}\sigma_{i-\half}p_{i+\half}\sigma_{i-\half}p_ip_{i-\half}p_{i-1}s_{i}p_{i-\half}\\
&=s_{i-1}p_{i-\half}s_is_{i-1}(\sigma_{i-\half})^2p_ip_{i-\half}p_{i-1}s_{i}p_{i-\half}\\
&=p_{i-\half}s_is_{i-1}p_ip_{i-\half}p_{i-1}s_{i}p_{i-\half}\\
&=p_{i-\half},
\end{align*}
as required. Considering the left hand side of~\eqref{state:12}, 
\begin{align*}
(s_{i-1}s_i\sigma_is_is_{i-1})(p_{i-\half}p_ip_{i+\half}s_{i-1}L_{i-1}p_{i-\half}s_i)&=s_{i-1}s_i\sigma_is_ip_{i-\half}s_ip_{i+1}p_{i+\half}s_{i-1}\sigma_ip_ip_{i-\half}s_i\\
&=s_{i-1}s_i\sigma_is_{i-1}p_{i+\half}s_{i-1}p_{i+1}p_{i+\half}s_{i-1}\sigma_ip_ip_{i-\half}s_i\\
&=s_{i-1}s_i\sigma_{i-\half}p_{i+\half}s_{i-1}p_{i+1}p_{i+\half}\sigma_{i-\half}p_ip_{i-\half}s_i\\
&=s_{i-1}p_{i+\half}s_{i-1}\sigma_{i-\half}p_{i+1}p_{i+\half}\sigma_{i-\half}p_ip_{i-\half}s_i\\
&=s_{i}p_{i-\half}s_{i}p_{i+1}p_{i+\half}(\sigma_{i-\half})^2p_ip_{i-\half}s_i\\
&=s_{i}p_{i-\half}s_{i}p_{i+1}p_{i+\half}p_ip_{i-\half}s_i\\
&=s_ip_{i-\half}s_i,
\end{align*}
which, by~\eqref{inter:d}, is equal to the right hand side of~\eqref{state:12}. Since~\eqref{state:5} is equivalent to~\eqref{state:1}, while~\eqref{state:2} is equivalent to~\eqref{state:6}, the proof of the proposition is complete.
\end{proof}
We record for later reference further consequences of the presentation given in Theorem~\ref{hr-presentation}. 
\begin{proposition}\label{r-2}
For $i=1,2,\ldots,$ the following statements hold:
\begin{enumerate}[label=(\arabic{*}), ref=\arabic{*},leftmargin=0pt,itemindent=1.5em]
\item $p_{i+1}\sigma_{i+1}p_{i+1}=L_ip_{i+1}$,\label{r-2-q}
\item $p_{i+1}\sigma_{i+\half}p_{i+1}=(z-L_{i-\half})p_{i+1}$,\label{r-2-a}
\item $p_{i+\thalf}\sigma_{i+1}p_{i+\thalf}=p_{i+\half}p_{i+\thalf}$.\label{r-2-b}
\end{enumerate}
\end{proposition}
\begin{proof}
\eqref{r-2-q} From Proposition~\ref{prel:a}, we obtain $p_{i+\half}p_{i+1}\sigma_{i+1}=p_{i+\half}L_i$. Thus
\begin{align*}
p_{i+1}\sigma_{i+1}=p_{i+1}p_{i+\half}p_{i+1}\sigma_{i+1}=p_{i+1}p_{i+\half}L_i
&&\text{and}&&
p_{i+1}\sigma_{i+1}p_{i+1}=p_{i+1}p_{i+\half}L_ip_{i+1}=L_ip_{i+1},
\end{align*}
as required.\newline
\eqref{r-2-a} We first compute
\begin{align*}
p_{i+1}s_ip_{i-\half}p_ip_{i+\half}s_{i-1}L_{i-1}p_{i-\half}s_ip_{i+1}
&=s_ip_{i}p_{i-\half}p_ip_{i+\half}s_{i-1}L_{i-1}p_{i-\half}s_ip_{i+1}\\
&=s_ip_{i}p_{i+\half}s_{i-1}L_{i-1}p_{i-\half}s_ip_{i+1}\notag\\
&=s_ip_{i}p_{i+\half}\sigma_ip_{i-1}p_{i-\half}p_{i}s_i\notag\\
&=s_ip_{i}p_{i+\half}\sigma_is_{i-1}p_{i}s_i\notag\\
&=s_ip_{i}p_{i+\half}\sigma_{i-\half}p_{i}s_i\notag\\
&=\sigma_{i-\half}p_{i+1}. \notag
\end{align*}
Now observe that $p_{2}\sigma_{1+\half}p_{2}=(z-L_\half)p_2$ and hence, by induction,
\begin{align*}
p_{i+1}\sigma_{i+\half}p_{i+1}
&=s_{i}s_{i-1}p_i\sigma_{i-\half}p_{i}s_{i-1}s_i
+p_{i+1}p_{i-\half}L_{i-1}s_ip_{i-\half}s_ip_{i+1}\\
&\quad+p_{i+1}s_ip_{i-\half}L_{i-1}s_ip_{i-\half}p_{i+1}
-p_{i+1}p_{i-\half}L_{i-1}s_{i-1}p_{i+\half}p_ip_{i-\half}p_{i+1}\\
&\quad-p_{i+1}s_ip_{i-\half}p_ip_{i+\half}s_{i-1}L_{i-1}p_{i-\half}s_ip_{i+1}\\
&=s_{i-1}s_ip_i\sigma_{i-\half}p_is_is_{i-1}+p_{i-\half}L_{i-1}p_{i+1}+L_{i-1}p_{i-\half}p_{i+1}\\
&\quad-p_{i-\half}L_{i-1}p_{i-1}p_{i-\half}p_{i+1}-\sigma_{i-\half}p_{i+1}\\
&=(z-L_{i-\half})p_{i+1}. 
\end{align*}
\eqref{r-2-b} Observe that $p_{2+\half}\sigma_2p_{2+\half}=p_{2+\half}s_1p_{2+\half}=p_{2+\half}p_{1+\half}$ and, by induction,
\begin{align*}
p_{i+\thalf}s_{i-1}s_i\sigma_is_is_{i-1}p_{i+\thalf}&=s_{i-1}p_{i+\thalf}s_i\sigma_is_ip_{i+\thalf}s_{i-1}\\
&=s_{i-1}s_is_{i+1}p_{i+\half}s_{i+1}\sigma_is_{i+1}p_{i+\half}s_{i}s_{i+1}s_{i-1}\\
&=s_{i-1}s_is_{i+1}p_{i+\half}\sigma_ip_{i+\half}s_{i}s_{i+1}s_{i-1}\\
&=s_{i-1}s_is_{i+1}p_{i-\half}p_{i+\half}s_{i}s_{i+1}s_{i-1}\\
&=p_{i+\half}p_{i+\thalf}.
\end{align*}
Therefore, using the definition~\eqref{sigma-i-1} and the fact that $p_{i-\half}L_{i-1}p_{i-\half}=p_{i-\half}$,
\begin{align*}
p_{i+\thalf}\sigma_{i+1}p_{i+\thalf}
&=p_{i+\thalf}p_{i+\half}
+p_{i+\thalf}s_ip_{i-\half}L_{i-1}s_ip_{i-\half}s_ip_{i+\thalf}\\
&\quad+p_{i+\thalf}p_{i-\half}L_{i-1}s_ip_{i-\half}p_{i+\thalf}
-p_{i+\thalf}s_ip_{i-\half}L_{i-1}s_{i-1}p_{i+\half}p_ip_{i-\half}p_{i+\thalf}\\
&\quad-p_{i+\thalf}p_{i-\half}p_ip_{i+\half}s_{i-1}L_{i-1}p_{i-\half}s_ip_{i+\thalf}\\
&=p_{i+\thalf}p_{i+\half}+p_{i+\thalf}s_ip_{i+\thalf}p_{i-\half}L_{i-1}s_ip_{i-\half}s_i\\
&\quad+p_{i-\half}L_{i-1}p_{i+\thalf}s_ip_{i+\thalf}p_{i-\half}
-p_{i+\thalf}s_ip_{i+\thalf}p_{i-\half}L_{i-1}s_{i-1}p_{i+\half}p_ip_{i-\half}\\
&\quad-p_{i-\half}p_ip_{i+\half}s_{i-1}L_{i-1}p_{i-\half}p_{i+\thalf}s_ip_{i+\thalf}\\
&=p_{i+\thalf}p_{i+\half}+p_{i+\thalf}p_{i+\half}p_{i-\half}L_{i-1}s_{i-1}p_{i+\half}s_{i-1}\\
&\quad+p_{i-\half}L_{i-1}p_{i+\thalf}p_{i+\half}p_{i-\half}
-p_{i+\thalf}p_{i+\half}p_{i-\half}L_{i-1}s_{i-1}p_{i+\half}p_ip_{i-\half}\\
&\quad-p_{i-\half}p_ip_{i+\half}s_{i-1}L_{i-1}p_{i-\half}p_{i+\half}p_{i+\thalf}\\
&=p_{i+\thalf}p_{i+\half},
\end{align*}
as required. 
\end{proof}
The next statement gives an alternative recursion for the family $(L_{i+\half}:i=0,1\ldots)$ for use in \S\ref{a-1-5}. 
\begin{theorem}
If $i=1,2,\ldots,$ then 
\begin{align*}
L_{i+\half}=-L_ip_{i+\half}-p_{i+\half}L_i+(z-L_{i-\half})p_{i+\half}+s_iL_{i-\half}s_i+\sigma_{i+\half}.
\end{align*}
\end{theorem}
\begin{proof}
By Propositions~\ref{prel:a}, Proposition~\ref{z-0} and Proposition~\ref{r-2}, 
\begin{align*}
p_{i+\half}L_ip_ip_{i+\half}
&=p_{i+\half}p_{i+1}\sigma_{i+1}p_ip_{i+\half}
=p_{i+\half}p_{i+1}\sigma_{i+\half}p_{i+1}p_{i+\half}=(z-L_{i-\half})p_{i+\half}.
\end{align*}
Substituting the above expression into the definition of $L_{i+\half}$ given in~\eqref{jm-i-2} provides the required statement. 
\end{proof}
\section{Schur--Weyl Duality}\label{a-1-5}
In this section we use Schur--Weyl duality to show that the family $(L_{i+\half},L_{i+1}:i=1,2,\ldots)$ defined above, and the Jucys--Murphy elements given by Halverson and Ram~\cite{HR:2005} are in fact equal. 

Let $n=1,2,\ldots,$ and $V$ be a vector space over $\mathbb{C}$ with basis $v_1,\ldots, v_n$.  If $r=1,2,\ldots,$ the tensor product
\begin{align*}
V^{\otimes r}=\underbrace{V\otimes V\otimes \cdots \otimes V}_{\text{$r$ factors}}&& \text{has basis}&&
\{v_{i_1}\otimes v_{i_2}\otimes \cdots \otimes v_{i_r} \,|\,1\le i_1,\ldots,i_r\le n\},
\end{align*}
and is equipped, via the inclusion $\mathfrak{S}_n\subset GL_n(\mathbb{C})$,  with the diagonal $\mathfrak{S}_n$--action 
\begin{align*}
w(v_{i_1}\otimes v_{i_2}\otimes \cdots \otimes v_{i_r})w
=v_{(i_1){w}^{-1}}\otimes v_{(i_2){w}^{-1}}\otimes \cdots \otimes v_{(i_r){w}^{-1}},
&&\text{for $w\in\mathfrak{S}_n$.}
\end{align*}
Let $A_{r}(n)=\mathcal{A}_r(z)\otimes_{\mathbb{Z}[z]}\mathbb{C}$, where $z$ acts on $\mathbb{C}$ as multiplication by $n$. The action of ${A}_r(n)$ on $V^{\otimes r}$ is given~(\S3 of~\cite{HR:2005}) by
\begin{align*}
u(v_{i_1}\otimes v_{i_{2}}\otimes\cdots\otimes v_{i_{r}})
=v_{i_{(1)u^{-1}}}\otimes v_{i_{(2)u^{-1}}}\otimes\cdots\otimes v_{i_{(r)u^{-1}}},
&&\text{for $u\in \mathfrak{S}_r$,}
\end{align*} 
and for $k=1,\ldots,r-1$,  
\begin{align*}
p_{k+\half}(v_{i_1}\otimes v_{i_{2}}\otimes\cdots\otimes v_{i_{r}})=
\begin{cases}
v_{i_1}\otimes v_{i_{2}}\otimes\cdots\otimes v_{i_{r}}&\text{if $i_k=i_{k+1}$,}\\
0&\text{otherwise,}
\end{cases}
\end{align*}
and for $k=1,\ldots,r$, 
\begin{align*}
p_k(v_{i_1}\otimes \cdots\otimes v_{i_{k-1}}\otimes v_{i_{k}}\otimes v_{i_{k+1}}\otimes\cdots\otimes v_{i_{r}})=\sum_{j=1}^n 
v_{i_1}\otimes \cdots\otimes v_{i_{k-1}}\otimes v_{j}\otimes v_{i_{k+1}}\otimes\cdots\otimes v_{i_{r}}.
\end{align*}
The ${A}_{r+\half}(n)$--action on $V^{\otimes r}$ is obtained in \S3 of~\cite{HR:2005} from the action of ${A}_{r+1}(n)$ on $V^{\otimes r+1}$ by restricting to the subspace $V^{\otimes r}\otimes v_n$ and identifying $V^{\otimes r}$ with $V^{\otimes r}\otimes v_n$. The next statement asserts that  $\mathfrak{S}_n$ and $A_{r}(n)$ act as commuting operators on $V^{\otimes r}$. 
\begin{theorem}[Theorem~3.22 of~\cite{HR:2005}]\label{hr-sw}
Let $n,r\in\mathbb{Z}_{\ge0}$. Let $S_n^\lambda$ denote the irreducible $\mathfrak{S}_n$-module indexed by $\lambda$. 
\begin{enumerate}[label=(\arabic{*}), ref=\arabic{*},leftmargin=0pt,itemindent=1.5em]
\item As $(\mathbb{C}\mathfrak{S}_n,A_r(n))$--bimodules 
\begin{align*}
V^{\otimes r}\cong 
\bigoplus_{\lambda\in\hat{A}_r(n)}S^\lambda_n\otimes{A}_r^\lambda(n),
\end{align*}
where $\lambda\in\hat{A}_r(n)$ is an indexing set for the irreducible ${A}_r(n)$-modules, and the vector spaces ${A}_r^\lambda(n)$, for $\lambda\in\hat{A}_r(n)$, are irreducible ${A}_r(n)$-modules. 
\item As $(\mathbb{C}\mathfrak{S}_{n-1},A_{r+\half}(n))$--bimodules 
\begin{align*}
V^{\otimes r}\cong 
\bigoplus_{\lambda\in\hat{A}_{r+\half}(n)}S^\lambda_{n-1}\otimes{A}_{r+\half}^\lambda(n),
\end{align*}
where $\lambda\in\hat{A}_{r+\half}(n)$ is an indexing set for the irreducible ${A}_{r+\half}(n)$-modules, and the vector spaces ${A}_{r+\half}^\lambda(n)$, for $\lambda\in\hat{A}_{r+\half}(n)$, are irreducible ${A}_{r+\half}(n)$-modules. 
\end{enumerate}
\end{theorem}
By Theorem~3.6 of~\cite{HR:2005}, the homomorphism $A_{r}(n)\to\End_{\mathfrak{S}_n}(V^{\otimes r})$ in Theorem~\ref{hr-sw} is an isomorphism whenever $n\ge2r$.

If $1\le i,j \le n$, let $s_{i,j}\in\mathfrak{S}_n$ denote the transposition which interchanges $i$ and $j$. The next statement gives the action of the group $\langle \sigma_{i+\half},\sigma_{i+1} \,|\,i=1,\ldots,r-1\rangle$ on $V^{\otimes r}$. 
\begin{proposition}
If $k=1,2,\ldots,$ and $v_{i_1}\otimes \cdots\otimes v_{i_{k+1}}\in V^{\otimes k+1}$, then 
\begin{align}
\sigma_{k+\half}(v_{i_1}\otimes \cdots\otimes v_{i_{k+1}})
=\underbrace{s_{i_k,i_{k+1}}\otimes\cdots\otimes s_{i_k,i_{k+1}}}_{\text{$k-1$ factors}}
(v_{i_1}\otimes \cdots\otimes v_{i_{k-1}})\otimes v_{i_{k}}\otimes v_{i_{k+1}}\label{s-z-1}
\end{align}
and
\begin{align}
\sigma_{k+1}(v_{i_1}\otimes \cdots\otimes v_{i_{k+1}})
=\underbrace{s_{i_k,i_{k+1}}\otimes\cdots\otimes s_{i_k,i_{k+1}}}_{\text{$k+1$ factors}}
(v_{i_1}\otimes \cdots\otimes v_{i_{k}}\otimes v_{i_{k+1}}).\label{s-z-2}
\end{align}
\end{proposition}
\begin{proof}
The proposition is true when $k=1$. If $k=2,3,\ldots,$ observe that the linear endomorphisms defined, for $v_{i_1}\otimes \cdots\otimes v_{i_{k+1}}\in V^{\otimes k+1}$,  by 
\begin{align*}
\theta_{k+\half}:v_{i_1}\otimes \cdots\otimes v_{i_{k+1}}
&\mapsto \underbrace{s_{i_k,i_{k+1}}\otimes\cdots\otimes s_{i_k,i_{k+1}}}_{\text{$k-1$ factors}}
(v_{i_1}\otimes \cdots\otimes v_{i_{k-1}})\otimes v_{i_{k}}\otimes v_{i_{k+1}}
\intertext{and} 
\theta_{k+1}:v_{i_1}\otimes \cdots\otimes v_{i_{k+1}}
&\mapsto\underbrace{s_{i_k,i_{k+1}}\otimes\cdots\otimes s_{i_k,i_{k+1}}}_{\text{$k+1$ factors}}
(v_{i_1}\otimes \cdots\otimes v_{i_{k}}\otimes v_{i_{k+1}})
\end{align*}
commute with the diagonal action of $\mathfrak{S}_n$ on $V^{\otimes k+1}$. Thus, by Theorem~\ref{hr-sw}, $\theta_{k+\half}$ and $\theta_{k+1}$, lie in the image of the map $A_{k+1}(n)\mapsto \End_{\mathfrak{S}_n}(V^{\otimes k+1})$. Observe also that the action of $\theta_{k+\half}$ on $V^{\otimes k+1}$ commutes with the action of $A_{k-1}(n)$ on $V^{\otimes k+1}$, and the action of $\theta_{k+1}$ on $V^{\otimes k+1}$ commutes with the action of $A_{k-\half}(n)$ on $V^{\otimes k+1}$. Since $A_{k+1}(n)$ is generated by $A_{k-\half}(n)$ together with $\langle\sigma_{k+\half},\sigma_{k+1},p_k,p_{k+\half}\rangle$, to show that the map $A_{k+1}(n)\to \End_{\mathfrak{S}_n}(V^{\otimes k+1})$ sends
\begin{align*}
\sigma_{k+\half}\mapsto\theta_{k+\half}&&\text{and}&&\sigma_{k+1}\mapsto\theta_{k+1}, 
\end{align*}
it now suffices to verify that, as operators on $V^{\otimes k+1}$, 
\begin{enumerate}[label=(\roman{*}), ref=\roman{*},leftmargin=0pt,itemindent=1.5em]
\item\label{ve-i} $\theta_{k+\half}^2=\theta_{k+1}^2=1$,
\item\label{ve-ii} $\theta_{k+\half}\theta_{k+1}=\theta_{k+1}\theta_{k+\half}=s_k$, 
\item\label{ve-iii} $\theta_{k+\half}p_{k+\half}=p_{k+\half}\theta_{k+\half}=p_{k+\half}$, 
\item\label{ve-iv} $\theta_{k+\half}p_kp_{k+1}=\theta_{k+1}p_kp_{k+1}$, 
\item\label{ve-v} $p_kp_{k+1}\theta_{k+\half}=p_kp_{k+1}\theta_{k+1}$, 
\item\label{ve-vi} $\theta_{k+\half}p_k\theta_{k+\half}=\theta_{k+1}p_{k+1}\theta_{k+1}$, 
\item\label{ve-vii} $\theta_{k+\half}p_{k-\half}\theta_{k+\half}=\sigma_kp_{k+\half}\sigma_k$,
\end{enumerate}
where $s_k$ in item~\eqref{ve-ii} acts on $V^{\otimes k}$ by place permutation. Since each of~\eqref{ve-i}--\eqref{ve-vii} can be verified by inspection, the proof of the proposition is complete. 
\end{proof}
\begin{proposition}\label{d-q}
If $k=1,2,\ldots,$ and $v_{i_1}\otimes  \cdots \otimes v_{i_k}\in V^{\otimes k}$ then 
\begin{align}\label{d-q-1}
L_{k-\half}(v_{i_1}\otimes \cdots \otimes v_{i_k})
=nv_{i_1}\otimes \cdots \otimes v_{i_k}
-\sum_{j=1}^n s_{i_k,j}\otimes\cdots\otimes s_{i_k,j}(v_{i_1}\otimes \cdots \otimes v_{i_{k-1}})\otimes v_{i_k},
\end{align}
and
\begin{align}\label{d-q-2}
L_k(v_{i_1}\otimes \cdots \otimes v_{i_k})
=\sum_{j=1}^n s_{i_k,j}\otimes\cdots\otimes s_{i_k,j}(v_{i_1}\otimes \cdots \otimes v_{i_{k-1}})\otimes v_j.
\end{align}
\end{proposition}
\begin{proof}
Identify $V^{\otimes k}$ as the subspace $V^{\otimes k}\otimes \sum_{j=1}^n v_j\subseteq V^{\otimes k+1}$. Then
\begin{align*}
&p_{k+1}\sigma_{k+\half}p_{i+1}(v_{i_1}\otimes\cdots\otimes v_{i_k}\otimes v_n)
=p_{k+1}\sum_{j=1}^n \sigma_{k+\half}(v_{i_1}\otimes\cdots\otimes v_{i_k}\otimes v_j)\\
\quad&=p_{k+1}\sum_{j=1}^n s_{i_k,j}\otimes\cdots\otimes s_{i_k,j} (v_{i_1}\otimes\cdots\otimes v_{i_{k-1}})\otimes v_{i_k}\otimes v_j\\
\quad&=\sum_{j,\ell=1}^n s_{i_k,j}\otimes\cdots\otimes s_{i_k,j} (v_{i_1}\otimes\cdots\otimes v_{i_{k-1}})\otimes  v_{i_k}\otimes v_\ell.
\end{align*}
By Proposition~\ref{r-2}, as operators on $V^{\otimes k}$, 
\begin{align*}
p_{k+1}\sigma_{k+\half}p_{i+1}(v_{i_1}\otimes\cdots\otimes v_{i_k}\otimes v_n)
&=(n-L_{i-\half})p_{i+1}(v_{i_1}\otimes\cdots\otimes v_{i_k}\otimes v_n)\\
&=\sum_{\ell=1}^n(n-L_{i-\half})(v_{i_1}\otimes\cdots\otimes v_{i_k}\otimes v_\ell),
\end{align*}
and the statement~\eqref{d-q-1} follows. Next, 
\begin{align*}
&p_{k+1}\sigma_{k+1}p_{k+1}(v_{i_1}\otimes\cdots\otimes v_{i_k}\otimes v_n)
=p_{k+1}\sum_{j=1}^n\sigma_{k+1}(v_{i_1}\otimes\cdots\otimes v_{i_k}\otimes v_j)\\
\quad &=p_{k+1}\sum_{j=1}^n\sigma_{k+\half}(v_{i_1}\otimes\cdots\otimes v_{i_{k-1}}\otimes v_{j}\otimes v_{i_k})\\
\quad &=p_{k+1}\sum_{j=1}^n s_{i_k,i_{k+1}}\otimes\cdots\otimes s_{i_k,i_{k+1}}(v_{i_1}\otimes\cdots\otimes v_{i_{k-1}}\otimes v_{j})\otimes v_{i_k}\\
\quad &=\sum_{j,\ell=1}^ns_{i_k,i_{k+1}}\otimes\cdots\otimes s_{i_k,i_{k+1}}(v_{i_1}\otimes\cdots\otimes v_{i_{k-1}}\otimes v_{j})\otimes v_{\ell}.
\end{align*}
By Proposition~\ref{r-2}, as operators on $V^{\otimes k}$, 
\begin{align*}
p_{k+1}\sigma_{k+1}p_{k+1}(v_{i_1}\otimes\cdots\otimes v_{i_k}\otimes v_n)
&=L_{k}p_{k+1}(v_{i_1}\otimes\cdots\otimes v_{i_k}\otimes v_n)\\
&=\sum_{\ell=1}^nL_{k}(v_{i_1}\otimes\cdots\otimes v_{i_k}\otimes v_\ell), 
\end{align*}
which yields~\eqref{d-q-2}. 
\end{proof}
As in~\S3 of~\cite{HR:2005}, let $\kappa_n$ be the central element which is the class sum corresponding to the conjugacy class of transpositions in $\mathbb{C}\mathfrak{S}_n$, 
\begin{align*}
\kappa_{n}=\sum_{1\le i<j\le n}^n s_{i,j}.
\end{align*}
For $\ell=1,\ldots,n$, we also define
\begin{align*}
\kappa_{n,\ell}=\sum_{\substack{1\le i<j\le n \\ i,j\ne \ell}}^ns_{i,j},
\end{align*}
so that $\kappa_{n,n}=\kappa_{n-1}$. 
\begin{proposition}\label{d-a}
Let $n=\dim(V)$ and $r\in\mathbb{Z}_{>0}$. If $z_{r-\half}\in A_{r-\half}(n)$ and $z_{r}\in A_{r}(n)$ are the central elements defined by Theorem~\ref{cent}, and $v_{i_1}\otimes  \cdots \otimes v_{i_r}\in V^{\otimes r}$, then 
\begin{align}
z_{r}(v_{i_1}\otimes  \cdots \otimes v_{i_r})=\kappa_n(v_{i_1}\otimes \cdots \otimes v_{i_r})-\big({\textstyle{{n}\choose{2}}}-rn\big)(v_{i_1}\otimes \cdots \otimes v_{i_r}),\label{d-a-1}
\end{align}
and, if $r=2,3,\ldots,$ then, 
\begin{align}
z_{r-\half}(v_{i_1}\otimes \cdots \otimes v_{i_{r}})
=\kappa_{n,i_{r}}(v_{i_1}\otimes \cdots \otimes v_{i_{r}})-\big(\textstyle{\textstyle{{n}\choose{2}}}-rn+1\big)(v_{i_1}\otimes \cdots \otimes v_{i_{r}}).\label{d-a-2}
\end{align}
\end{proposition}
\begin{proof}
The proof is by induction on $r$. Let $i=1,\ldots,n$. Since $z_1=p_1$, 
\begin{align*}
\kappa_nv_i
&=\sum_{i>j} s_{j,i} v_i + \sum_{j>i}s_{i,j}v_j+\sum_{\substack{j,\ell\\ j,\ell\ne i}}s_{j,\ell}v_i\\
&=\sum_{j=1}^nv_j+\big({\textstyle{{n-1}\choose{2}}}-1\big)v_i\\
&=z_1v_{i}+\big({\textstyle{{n}\choose{2}}}-n\big)v_i,
\end{align*}
which verifies~\eqref{d-a-1} when $r=1$.  Now observe that if $v_{i_1}\otimes \cdots\otimes v_{i_k}\in V^{\otimes r}$, then the diagonal action of  $\kappa_n$ and $\kappa_{n,i_r}$ on $V^{\otimes k}$ allows us to write 
\begin{align*}
\kappa_n(v_{i_1}\otimes \cdots\otimes v_{i_k})=\kappa_{n,i_k}(v_{i_1}\otimes \cdots\otimes v_{i_k})
+\sum_{\substack{j=1\\ j\ne i_k}}^n s_{i_k,j}\otimes\cdots\otimes s_{i_k,j}(v_{i_1}\otimes \cdots\otimes v_{i_{k-1}})\otimes v_{j},
\end{align*}
and 
\begin{multline*}
\kappa_{n,i_k}(v_{i_1}\otimes \cdots\otimes v_{i_k})=\kappa_{n}(v_{i_1}\otimes \cdots\otimes v_{i_{k-1}})\otimes v_{i_k}\\
-\sum_{\substack{j=1\\ j\ne i_k}}^n s_{i_k,j}\otimes\cdots\otimes s_{i_k,j}(v_{i_1}\otimes \cdots\otimes v_{i_{k-1}})\otimes v_{i_k}.
\end{multline*}
Assuming that~\eqref{d-a-1} holds for $r=1,\ldots,k-1,$  we obtain
\begin{align*}
z_{k-\half}(v_{i_1}\otimes\cdots\otimes v_{i_k})
&=z_{k-1}(v_{i_1}\otimes\cdots\otimes v_{i_k})
+L_{k-\half}(v_{i_1}\otimes\cdots\otimes v_{i_k})\\
&=\kappa_n(v_{i_1}\otimes\cdots\otimes v_{i_{k-1}})\otimes v_{i_k}
-\big(\textstyle{\textstyle{{n}\choose{2}}}-(k-1)n\big)(v_{i_1}\otimes  \cdots \otimes v_{i_{k}})\\
&\quad+nv_{i_1}\otimes\cdots\otimes v_{i_k}
-\sum_{j=1}^ns_{i_k,j}\otimes\cdots\otimes s_{i_k,j}(v_{i_1}\otimes\cdots\otimes v_{i_{k-1}})\otimes v_{i_k}\\
&=\kappa_n(v_{i_1}\otimes\cdots\otimes v_{i_{k-1}})\otimes v_{i_k}
-\big(\textstyle{\textstyle{{n}\choose{2}}}-kn+1\big)(v_{i_1}\otimes  \cdots \otimes v_{i_{k}})\\
&\quad-\sum_{\substack{j=1\\j\ne i_k}}^ns_{i_k,j}\otimes\cdots\otimes s_{i_k,j}(v_{i_1}\otimes\cdots\otimes v_{i_{k-1}})\otimes v_{i_k}\\
&=\kappa_{n,i_k}(v_{i_1}\otimes\cdots\otimes v_{i_k})
-\big(\textstyle{\textstyle{{n}\choose{2}}}-kn+1\big)(v_{i_1}\otimes  \cdots \otimes v_{i_{k}}),
\end{align*}
which verifies~\eqref{d-a-2} for $r=k$. Next, since~\eqref{d-a-2} holds for $r=2,3,\ldots,k,$  we obtain 
\begin{align*}
z_k(v_{i_1}\otimes \cdots\otimes v_{i_k})&=z_{k-\half}(v_{i_1}\otimes\cdots\otimes v_{i_k})+L_{k}(v_{i_1}\otimes\cdots\otimes  v_{i_k})\\
&=\kappa_{n,i_k}(v_{i_1}\otimes \cdots\otimes  v_{i_k})-\big(\textstyle{\textstyle{{n}\choose{2}}}-kn+1\big)(v_{i_1}\otimes  \cdots \otimes v_{i_{k}})\\
&\quad+\sum_{j=1}^n s_{i_k,j}\otimes\cdots\otimes s_{i_k,j}(v_{i_1}\otimes \cdots \otimes v_{i_{k-1}})\otimes v_j\\
&=\kappa_{n,i_k}(v_{i_1}\otimes \cdots\otimes  v_{i_k})-\big(\textstyle{\textstyle{{n}\choose{2}}}-kn\big)(v_{i_1}\otimes  \cdots \otimes v_{i_{k}})\\
&\quad+\sum_{\substack{j=1\\j\ne i_k}}^ns_{i_k,j}\otimes\cdots\otimes s_{i_k,j}(v_{i_1}\otimes \cdots \otimes v_{i_{k-1}})\otimes v_j\\
&=\kappa_{n}(v_{i_1}\otimes \cdots\otimes  v_{i_k})-\big(\textstyle{\textstyle{{n}\choose{2}}}-kn\big)(v_{i_1}\otimes  \cdots \otimes v_{i_{k}}),
\end{align*}
which verifies~\eqref{d-a-1} for $r=k$. 
\end{proof}
Let $Z_k\in A_k(n)$ and $Z_{k+\half}\in A_{k+\half}(n)$ denote the central element defined by Halverson and Ram in~\S3 of~\cite{HR:2005}. Then the Jucys--Murphy elements of~\cite{HR:2005} are given by 
\begin{align*}
M_{k-\half}=Z_{k-\half}-Z_{k-1}&&\text{and}&&M_{k}=Z_{k}-Z_{k-\half}&&\text{for $k=1,2,\ldots.$}
\end{align*}
\begin{theorem}\label{t-r-0}
if $k=1,2,\ldots,$ then $M_{k+\half}=L_{k+\half}$ and $M_{k+1}=L_{k+1}$ as elements of $A_{k+1}(n)$.
\end{theorem}
\begin{proof}
By Theorem~3.6 of~\cite{HR:2005}, the homomorphism $A_{k}(n)\to\End_{\mathfrak{S}_n}(V^{\otimes k})$ is an isomorphism whenever $n\ge2k$. Since the coefficients in the expansions of $Z_k,Z_{k+\half}$ and $z_k,z_{k+\half}$, in terms of the diagram basis for $A_{k+1}(n)$, are polynomials in $n$, and the map $A_{k}(n)\to \End_{\mathfrak{S}_n}(V^{\otimes k})$ is an isomorphism for infinitely many values of $n$, to prove the theorem, it suffices to compare the action of $z_k$ and $Z_{k}$ (\emph{resp.} $z_{k+\half}$ and $Z_{k+\half}$) on $V^{\otimes k}$ for an arbitrary choice of $n$. Identifying $V^{\otimes k}$ with the subspace $V^{\otimes k}\otimes v_n\subseteq V^{\otimes k+1}$, and $\kappa_{n-1}$ with $\kappa_{n,n}$, then Theorem~3.35 of~\cite{HR:2005} states that, as operators on $V^{\otimes k}$,
\begin{align}\label{hr-action}
Z_k=\kappa_n-\big(\textstyle{{n}\choose{2}}-kn\big)
&&\text{and}&&
Z_{k+\half}=\kappa_{n-1}-\big(\textstyle{{n}\choose{2}}-(k+1)n+1\big).
\end{align}
Comparing~\eqref{hr-action} with the action of $z_k$ and $z_{k+\half}$ on $V^{\otimes k}$ in~\eqref{d-a-1} and~\eqref{d-a-2} completes the proof. 
\end{proof}
\begin{remark}
Whereas $L_\half=0$ and $L_1=p_1$, in~\cite{HR:2005}, the first three Jucys--Murphy elements are $M_0=M_{\half}=1$, and $M_1=p_1-1$. Thus, although $z_1=Z_1$ as elements of $A_1(n)$,  Theorem~\ref{t-r-0} cannot be extended to the case $k=0$. 
\end{remark}


\end{document}